\documentclass[final,1p,times]{elsarticle}

\advance\textwidth by 2.5cm
\advance\oddsidemargin by -1.6cm \advance\evensidemargin by -1.6cm

\usepackage{amssymb,color}
\usepackage{amsmath}
\usepackage{amsthm}

\usepackage{mathrsfs}

\newcommand{\rA}{\mathrm{ A}}

\newcommand{\rH}{\mathrm{ H}}
\newcommand{\rK}{\mathrm{K}}
\newcommand{\rV}{\mathrm{ V}}
\newcommand{\embed}{\hookrightarrow }

\newcommand{\lb}{\langle}
\newcommand{\rb}{\rangle}

\usepackage[OT2,OT1]{fontenc}
\newcommand\cyr{%
\renewcommand\rmdefault{wncyr}%
\renewcommand\sfdefault{wncyss}%
\renewcommand\encodingdefault{OT2}%
\normalfont \selectfont} \DeclareTextFontCommand{\textcyr}{\cyr}

\numberwithin{equation}{section}
\newtheorem{theorem}{Theorem}[section]
\newtheorem{assumption}[theorem]{Assumption}
 \newtheorem{remark}[theorem]{Remark}
\newtheorem{lemma}[theorem]{Lemma}
\newtheorem{cor}[theorem]{Corollary}
\newtheorem{prop}[theorem]{Proposition}
\newtheorem{definition}[theorem]{Definition}
\newtheorem{example}[theorem]{Example}

\newtheoremstyle{AppALem}{1}{1}
  {\itshape}{0pt}{\bfseries}{.}{ }
   {\thmname{Lemma }\thmnumber{A.{#2}}{\thmnote{}}}
   \theoremstyle{AppALem}\newtheorem{lemmaA}{Lemma}

\newtheoremstyle{AppBLem}{1}{1}
  {\itshape}{0pt}{\bfseries}{.}{ }
   {\thmname{Lemma }\thmnumber{B.{#2}}{\thmnote{}}}
   \theoremstyle{AppBLem}\newtheorem{lemmaB}{Lemma}

\newtheoremstyle{AppBCor}{1}{1}
  {\itshape}{0pt}{\bfseries}{.}{ }
   {\thmname{Corollary }\thmnumber{B.{#2}}{\thmnote{}}}
   \theoremstyle{AppBCor}\newtheorem{corB}[lemmaB]{Lemma}

\newtheoremstyle{AppCTh}{1}{1}
  {\itshape}{0pt}{\bfseries}{.}{ }
   {\thmname{Theorem }\thmnumber{C.{#2}}{\thmnote{}}}
   \theoremstyle{AppCTh}\newtheorem{theoremC}{Theorem}

\newtheoremstyle{AppCProp}{1}{1}
  {\itshape}{0pt}{\bfseries}{.}{ }
   {\thmname{Proposition }\thmnumber{C.{#2}}{\thmnote{}}}
   \theoremstyle{AppCProp}\newtheorem{propC}[theoremC]{Proposition}

\newtheoremstyle{AppCCounterexample}{1}{1}
  {\itshape}{0pt}{\bfseries}{.}{ }
   {\thmname{Counterexample }\thmnumber{C.{#2}}{\thmnote{}}}
   \theoremstyle{AppCCounterexample}\newtheorem{counterexampleC}[theoremC]{Counterexample}

\newcommand\rv{ \mathrm{v}}
\newcommand\rw{ \mathrm{w}}

\newcommand{\nat}{\mathbb{N}}
\newcommand{\rzecz}{\mathbb{R}}
\newcommand{\eps}{\varepsilon }
\newcommand{\vp}{\varphi }

\newcommand{\rd}{{\rzecz }^{d}}

\newcommand{\tr}{\mbox{\rm Tr}}

\newcommand{\diver}{\mbox{\rm div}\,}

\newcommand{\acal}{\mathcal{A}}

\newcommand{\ccal}{\mathcal{C}}

\newcommand{\fcal}{\mathcal{F}}

\newcommand{\lcal}{\mathcal{L}}

\newcommand{\ocal}{\mathcal{O}}

\newcommand{\vcal}{\mathcal{V}}
\newcommand{\xcal}{\mathcal{X}}

\newcommand{\zcal}{\mathcal{Z}}

\newcommand{\fmath}{\mathbb{F}}

\newcommand{\ball}{\mathbb{B}}

\newcommand{\norm}[3]{\vert  #1 {\vert }_{#2}^{#3}}
\newcommand{\Norm}[3]{\Vert  #1 {\Vert }_{#2}^{#3}}

\newcommand{\ilsk}[3]{{\bigl( #1 , #2 \bigr)}_{#3}}

\newcommand{\dual}[3]{{\lb #1 , #2 \rb}_{#3}}
\newcommand{\Dual}[3]{{\Bigl< #1  , #2 \Bigr>}_{#3}}

\newcommand{\dirilsk}[3]{{\bigl( \! \bigl( #1 , #2 \bigr) \! \bigr)}_{#3}}

\newcommand{\ddual}[4]{{}_{#1}\lb #2 ,#3 {\rb }_{#4}}

\newcommand{\p}{\mathbb{P}}

\newcommand{\ft }{{\fcal }_{t}}

\newcommand{\tOmega}{\tilde{\Omega}}

\newcommand{\tfcal}{\tilde{\fcal}}
\newcommand{\tp}{\tilde{\mathbb{P}}}

\newcommand{\te}{\tilde{\mathbb{E}}}

\newcommand{\Xn}{{X}_{n}}

\newcommand{\Pn}{{P}_{n}}

\newcommand{\un}{{u}_{n}}
\newcommand{\Jn}[1]{{J}^{n}_{#1}}

\newcommand{\unk}{{u}_{{n}_{k}}}
\newcommand{\taun}{{\tau}_{n}}

\newcommand{\lhs}{{\mathcal{T}_2}}

\newcommand{\Bn}{{B}_{n}}

\newcommand{\tW}{\tilde{W}}

\newcommand{\bi}{{b}_{i}}
\newcommand{\bij}{{b}_{i}^{j}}
\newcommand{\bik}{{b}_{i}^{k}}
\newcommand{\ci}{{c}_{i}}
\newcommand{\betai}{{\beta}_{i}}

\newcommand{\tun}{{\tilde{u}}_{n}}
\newcommand{\tunk}{{\tilde{u}}_{{n}_{k}}}

\newcommand{\tu}{\tilde{u}}

\newcommand{\e}{\mathbb{E}}

\begin{document}

\begin{frontmatter}
\title{Invariant measure for  the stochastic Navier-Stokes equations
in  unbounded 2D domains}

\author[ZB]{Zdzis\l aw Brze\'{z}niak}
\ead{zdzislaw.brzezniak@york.ac.uk}
\author[EM]{El\.zbieta Motyl}
\ead{emotyl@math.uni.lodz.pl}
\author[MO]{ and Martin Ondrejat}
\ead{ondrejat@utia.cas.cz}
\address[ZB]{Department of Mathematics, University of York, Heslington, York,
YO105DD, United Kingdom}
\address[EM]{Faculty of Mathematics and Computer Science, University of \L\'{o}d\'{z}, ul. Banacha 22,
91-238 \L \'{o}d\'{z}, Poland}
\address[MO]{The Institute of Information Theory and Automation of the Czech Academy of Sciences, Pod Vod\'arenskou v\v e\v z\'{\i} 4, CZ-182 08, Praha 8, Czech Republic}
\tnotetext[t1]{The research of the first and second  named authours was supported partially supported by the Leverhulme  grant RPG-2012-514.
The research of the third named authour was supported by the GA\v CR grant 15-08819S.}

\begin{abstract}
Building upon a recent work by two of the authours and J. Seidler on bw-Feller property  for stochastic nonlinear beam and wave equations,
 we prove the existence of an invariant measure to  stochastic 2-D Navier-Stokes (with multiplicative noise) equations in unbounded domains.  This
 answers an open question left after the first authour and Y. Li proved a corresponding result in the case of an additive noise.
\end{abstract}


\end{frontmatter}

\section{Introduction} \label{sec-Introduction}

A classical method of proving the existence of an invariant measure for a Markov proceess is  the celebrated Krylov-Bogoliubov method. Originally it was used for Markov processes with values in locally compact  state spaces, e. g.  finite dimensional Euclidean spaces, see e.g. \cite{Krylov+Bogoliubov_1993} and \cite{Oxtoby+Ulam_1939}.  In the recent years  it has been successfully generalised to  Markov processes with non-locally compact state spaces, e.g.  infinite  dimensional Hilbert and Banach spaces, see for instance the books by Da Prato and Zabczyk \cite{DaPrato+Zabczyk_1992,DaPrato+Zabczyk_1996} and a fundamental paper by Flandoli \cite{Flandoli_1994}  for the case of $2$ dimensional Navier-Stokes equations with additive noise. One should also mention here a somehow reverse problem, found for instance in the stochastic quantisation approach of Parisi and Wu \cite{Paris+Wu_1981}, of constructing a Markov process with certain properties given an 'a priori invariant measure'. In the context of Stochastic Partial Differential Equations, this approach has been successfully implemented by Da Prato and Debussche for $2$ dimensional Navier-Stokes equations with periodic boundary conditions driven by space time white noise  in \cite{DaPrato+Debussche_2002} and for the 2-D stochastic quantization equation in \cite{DaPrato+Debussche_2003}.

 The latter method is related to the approach by Dirichlet forms as for instance in \cite{Albeverio+Rockner_1988}.
In the field of deterministic dynamical systems the so called Avez method, see \cite{Avez_1968}  and \cite{Lasota+Pian_1977}, is also popular.
It seems that the first of these methods when used in order to prove the existence of an invariant measure for Markov processes generated by SPDEs one requires the existence of an auxiliary  set which is compactly embedded into the state space and in which the Markov process eventually lives. Thus, it has so far been restricted to SPDEs of parabolic type (giving necessary conditions with smoothing effect) and in bounded domains (providing the needed compactness via the Rellich Theorem).

On the other hand, as a byproduct of results obtained by  Yuhong Li and the 1st named authour in \cite{Brzezniak+Li_2006},  about the existence of a compact absorbing set for   stochastic  $2$ dimensional Navier-Stokes equations with additive noise in a certain class of unbounded domains, there exists an invariant measure for the Markov process generated by such equations.  This, to the best of the authours knowledge, provides the first example of a nontrivial SPDEs  without the previously required compactness assumption possessing an invariant measure. A posteriori, one can see that behind the proof is the continuity of the corresponding solution flow with respect to the \textbf{weak topologies}, see Example \ref{example-Brzezniak+Li_2006}.

It is has been discovered in \cite[Proposition 3.1]{Maslowski+Seidler_1999} that  a
$bw$-Feller semigroup has an invariant probability measure provided  the set
 \begin{equation}\label{tag01}
\left\{\frac1{T_n}\int^{T_n}_0 P^{\ast }_{s}\nu\,ds;\; n\ge 1\right\}
\end{equation}
 is tight on $(\rH,bw)$.  However, it is far from  straightforward to identify stochastic PDEs for which the associated transition semigroups are $bw$-Feller. This has been recently done for SPDEs of hyperbolic type (i.e. second order in time)  such as beam and nonlinear wave equations in \cite{Brz+Ondr+Seidler_2015}. The aim of this work is to show that the general approach proposed in that paper is also applicable to stochastic Navier-Stokes equations in unbounded domains. In the case of bounded domains, the first such a result has been obtained by Flandoli in the celebrated paper \cite{Flandoli_1994}. A similarity between the equations studied in \cite{Brz+Ondr+Seidler_2015} and the current paper is that the linear generator has no compact resolvent. However, in the current situation, the generator is sectorial contrary to the former case. However, the smoothing of the semigroup is rather used to counterweight the non-smoothness of the nonlinearity.

 On the other hand, in \cite{Maslowski+Seidler_1999} Maslowski and Seidler  proposed to use the of weak topologies to the proof of the existence of invariant measures but the applications of the proposed theory had limited scope.

 These two papers, i.e. \cite{Maslowski+Seidler_1999} and  \cite{Brzezniak+Li_2006} have inspired us to investigate this matter further.

Moreover,   while working on the existence of solutions to geometric wave equations it has become apparent to us that the methods of using very fine techniques  in order to overcome the difficulty arising from having only weak a'priori estimates should also allow one to prove the sequentially weak Feller property required by the Maslowski and Seidler approach.  This made it possible to prove the existence of invariant measure for SPDES of hyperbolic type as for instance wave and beam, see the recent paper \cite{Brz+Ondr+Seidler_2015} by the Seidler and the 1st and 3rd authours.

 The aim of the current work is to show that the approach worked out in \cite{Brz+Ondr+Seidler_2015} combined with the method of proving the existence of Stochastic Navier-Stokes Equations in general domains  developed recently by 1st and 2nd authours, see for instance \cite{Brz+Motyl_2013}, indeed can lead to a proof of  the existence of an invariant measure for stochastic $2$ dimensional  Navier-Stokes equations with multiplicative noise (and additive as well) in unbounded domains and thus generalizing the previously mentioned  result \cite{Brzezniak+Li_2006}.

Let us stress that the general result proved in Sections 5-10 of \cite{Brz+Ondr+Seidler_2015} does no apply directly
to Stochastic NSEs.
Instead we propose  a scheme which is general enough that it should be applicable to other equations. Let us describe it in more detail.
In a domain $\mathcal{O} \subset \mathbb{R}^2$ satisfying the Poincar\'e inequality we consider the following stochastic Navier-Stokes equations in the functional form
\begin{equation} \label{eqnintro-E:NS}
\begin{cases}
& du(t)+ A u(t) \, dt+B\bigl( u(t),u(t) \bigr)\, dt = f \, dt+G\bigl( u(t)\bigr) \, dW(t),
 \qquad t \in [0,T] , \\
& u(0) = {u}_{0} ,
\end{cases}
\end{equation}
where $A$ is the Stokes operator, $u_0\in \rH$, $f\in \rV^\prime$ and we use the standard notation, see the parts of the paper around equation \eqref{E:NS}. In particular, $W=\big(W(t)\big)_{t\geq 0}$ is a cylindrical Wiener process on a separable Hilbert space $\rK$ defined on a ceratin probability space and the nonlinear diffusion coefficient $G$ satisfy some natural assumptions. It is known (but we provide an independent proof of this fact) that the above problem has a unique global solution $u(t;u_0)$, $t\geq 0$. The corresponding semigroup $(P_t)_{t\geq 0}$ is Markov, see  Proposition \ref{prop-Markov}. This semigroup is defined by the formula, see \eqref{eqn:semigroup},
\begin{equation}
\label{eqnintro-semigroup}
 (P_t \varphi)(u_0) = \mathbb{E}[\varphi (u(t;u_0))], \quad t\geq 0,\;\;\; u_0 \in \rH ,
\end{equation}
for any  bounded  Borel  function $\varphi \in  {\mathcal{B}}_{b}(\rH) $. Then, see Proposition \ref{prop-Feller_bw}, we prove that this semigroup
is $bw$-Feller, i.e. for every  $t>0$ and every    bounded sequentially weakly continuous  function  $\phi:\rH\to\mathbb{R}$,    the function
$P_t\phi:\rH\to\mathbb{R}$ is also  bounded sequentially weakly continuous.

The idea of the proof of the last result can be traced to recent papers by  all three of us in which we proved the existence of weak martingale solutions
 to the stochastic   geometric wave and Navier-Stokes and equations  developed respectively in
 \cite{Brz+Ondrejat_2011a,Brz+Ondrejat_2011} and \cite{Brz+Motyl_2013}.

 Finally, our main result, i.e. Theorem \ref{thm-main}  about the existence of an invariant measure for the semigroup $(P_t)_{t \geq 0}$,  follows provided some natural assumptions, as inequality \eqref{E:G} holds with ${\lambda }_{0}=0$, i.e. for some\footnote{Throughout  the whole paper we use the symbol $\lhs$  to denote the space of Hilbert-Schmidt operators
between corresponding Hilbert spaces.} $\rho \geq 0$,
\begin{equation} \label{eqninto-G}
\norm{G(u )}{\lhs (\rK,\rH)}{2}
  \le (2-\eta ) \Norm{u}{}{2}+\rho , \qquad u \in \rV ,
\end{equation}
guaranteeing the uniform boundedness in probability,  are satisfied, see Corollary \ref{cor-Krylov_Bogoliubov_cond}.

In proving  Proposition \ref{prop-Feller_bw}  the  continuity/stability result contained in Theorem \ref{thm-weak continuous dependence-existence_2D} plays an essential r\^ole.

We will present now the earlier promised example based on the paper \cite{Brzezniak+Li_2006}.

\begin{example}\label{example-Brzezniak+Li_2006}
If $\vp=(\vp_t)_{t\geq 0}$ is a deterministic dynamical system on a Hilbert space $\rH$, then one can define the corresponding Markov semigroup by
\begin{equation}\label{eqn-semigroup-deterministic}
[P_t(f)](x):= f(\vp_t(x)),\;\; t\geq 0, \;\; x\in \rH.
\end{equation}
Suppose that the semiflow  is sequentially weakly continuous in the following sense.
\begin{equation}\label{eqn-weak cont-deterministic}
 \mbox{If }t_n \to t \in \mathbb{R}_+, \;\; x_n \to x \mbox{ weakly in } \rH \mbox{ then } \vp_{t_n}(x_n) \to \vp_t (x) \mbox{ weakly in } \rH.
\end{equation}
Note that the above condition is satisfied for the deterministic 2-d Navier-Stokes equations, see \cite{Rosa_1998} and also
\cite[Lemma 7.2]{Brzezniak+Li_2006}.\\
Then, the assertion of Theorem 9.4 in \cite{Brz+Ondr+Seidler_2015} holds.
 Indeed, let us choose and fix   a bounded sequentially weakly continuous function $f:\rH\to\mathbb{R}$,
a sequence $(t_n)\to t$ and  a sequence $(x_n)$  such that   $x_n\to x$ weakly in $\rH$. Then by assumption \eqref{eqn-weak cont-deterministic}
$\vp_{t_n}(x_n) \to \vp_t (x)$ weakly in  $\rH$ and since $f$ is sequentially weakly continuous we infer that
\begin{equation*}
 [P_{t_n}(f)](x_n)= f (\vp_{t_n}(x_n))\to f(\vp_t (x))= P_t f(x).
 \end{equation*}
The condition guaranteeing the existence of an invariant measure, see \cite[Theorem 10.1]{Brz+Ondr+Seidler_2015},
 now reads as follows. There exists $x\in \rH$ such     for every $\varepsilon>0$, there exists $R>0$ such that
\begin{equation}\label{eqnintro-uniboup}
\limsup_{t\to\infty}\frac 1t\int_0^t 1_{\vert \vp_s(x)\vert_\rH\ge R}\,ds\le\varepsilon
\end{equation}
which is obviously satisfied provided the dynamical system $\vp=(\vp_t)_{t\geq 0}$ is bounded at infinity, i.e.
there exists $x\in \rH$ and  $R>0$ such that $\vert \vp_s(x)\vert_\rH\leq R $ for all $s\geq 0$. It is well known that this condition holds for
 the deterministic 2-d Navier-Stokes equations in a Poincar\'e domain (as well as for the damped Navier-Stokes Equations in the whole space $\mathbb{R}^2$. Thus we conclude, that in those cases, there exists an invariant measure. Of course, these are known results, the purpose of this Example is only to elucidate our paper by showing that it is also applicable to these cases.

Let us point out that \cite[Lemma 7.2]{Brzezniak+Li_2006} played an important r\^ole in that paper.\\
 We  believe that the result described in this Example holds also for the Random dynamical system from \cite{Brzezniak+Li_2006}. In this way, we will get an alternative proof of the result existence of an invariant measure proved in  that paper.

The weak continuity  property \eqref{eqn-weak cont-deterministic}  has also been investigated \cite{Ball_1997,Rosa_1998,Brzezniak+Li_2006,Daners_2005}. In the first three of these references the weak to weak continuity is an important tool in proving the existence of an attractor for deterministic 2D  Navier-Stokes Equations  in unbounded domains, where, as we pointed out earlier,  the compactness of the embedding from the Sobolev space $H^1$ to $L^2$ does not hold. A similar type of  continuity (weak to strong), is encountered in the proof of the large deviation principle for SPDES, see for instance \cite[Lemma 6.3]{Brz+Goldys+Jegaraj_2016} for  the case of Stochastic Landau-Lifshitz Equations. It might be interesting to understand in  the relationship between these two types of continuity.

\end{example}

\bigskip
Let us finish the Introduction with a brief description of   the content of the paper. Section \ref{sec-preliminaries NSES}
is devoted to recalling some basic notation and information.
In section \ref{sec-SNSEs} we recall the fundamental facts about Navier-Stokes Equations.  This section is based on a similar presentation in \cite{Brz+Motyl_2013}, however, in the present paper, we make some modifications.
In section \ref{sec-Cont-dependence} we formulate and prove the convergence result
for a sequence of martingale solutions of the Stochastic NSEs, see for instance Theorems \ref{thm-continuous dependence-tightness} and \ref{thm-continuous dependence-existence}.
The results of section \ref{sec-Cont-dependence} hold both in 2 and 3-dimensional possibly unbounded domains. Let us stress this again, these two results are for sequence of martingale solutions of the Stochastic NSEs. In the case when these are replaced by strong solutions of the corresponding Galerkin approximations, the corresponding results have been proved in \cite{Brz+Motyl_2013}, see also Theorem
\ref{T:existence_NS} in the present paper.
In section \ref{sec-2D-SNSEs} we recall the main results from \cite{Brz+Motyl_2013} in the special case of $2$-dimensional domains. Besides,  we prove Theorem \ref{thm-weak continuous dependence-existence_2D}, needed in the main section, and being the counterpart of Theorem \ref{thm-continuous dependence-existence} for the 2-dimensional case. Theorems \ref{thm-continuous dependence-tightness}, \ref{thm-continuous dependence-existence} and \ref{thm-weak continuous dependence-existence_2D} generalise
\cite[Lemmata 7.1 and 7.2]{Brzezniak+Li_2006}.
In section \ref{sec-invariant} we state and proof the main result of this paper, i.e. the existence of invariant measures for Stochastic Navier-Stokes equations in  2-dimensional Poincar\'{e}, possibly unbounded, domains with multiplicative noise.

\subsection*{Acknowledgements}
The authours would like to thank  an
anonymous referee for careful reading of the manuscript and useful remarks.

\section{Preliminaries}  \label{sec-preliminaries NSES}

The following introductory section is for the reader convenience and hence relies heavily on \cite{Brz+Motyl_2013}, arXiv:1208.3386.

\bigskip \noindent
Let $\ocal \subset {\rzecz }^{d} $, where $d=2,3$, be an open connected subset with smooth boundary $\partial \ocal $.
For $p\in [1,\infty )$ by  ${L}^{p}(\ocal , \rd )$ we denote the Banach space of (equivalence classes) of Lebesgue measurable $\rd $-valued $p$-th power integrable functions on the set $\ocal $. The norm in ${L}^{p}(\ocal , \rd )$ is given by
\[
      \norm{u}{{L}^{p}}{} := \biggl( \int_{\ocal } |u(x){|}^{p} \, dx {\biggr) }^{\frac{1}{p}} , \qquad u \in {L}^{p}(\ocal , \rd ).
\]
By ${L}^{\infty } (\ocal , \rd )$ we denote the Banach space of Lebesgue measurable essentially bounded $\rd $-valued functions defined on $\ocal $ with  the norm defined by
\[
     \norm{u}{{L}^{\infty }}{}:= \mbox{\rm esssup} \,\{  |u(x)| , \, \, x \in \ocal  \}  , \qquad u \in {L}^{\infty } (\ocal , \rd ).
\]
If $p=2$, then ${L}^{2}(\ocal , \rd )$ is a Hilbert space with the inner product given by
\[
  \ilsk{u}{\rv}{{L}^{2}} := \int_{\ocal } u(x) \cdot \rv (x) \, dx , \qquad u,\rv \in {L}^{2}(\ocal ,\rd ).
\]
By  ${H}^{1}(\ocal ,\rd )={H}^{1,2}(\ocal ,\rd )$ we will denote  the Sobolev space consisting of all $u \in {L}^{2}(\ocal ,\rd )$ for which
there exist weak derivatives $D_iu\in {L}^{2}(\ocal ,\rd )$, $i=1,\cdots,d$.
It is a Hilbert space with the inner product given by
\[
   \ilsk{u}{\rv}{{H}^{1}} := \ilsk{u}{\rv}{{L}^{2}} + \ilsk{\nabla u}{\nabla \rv}{{L}^{2}} , \qquad u,\rv \in
   {H}^{1}(\ocal ,\rd ),
\]
where $\ilsk{\nabla u}{\nabla \rv}{{L}^{2}}:= \sum_{i=1}^{d}\int_{\ocal }D_iu(x) \cdot  D_i\rv(x) \, dx $.
Let ${\ccal }^{\infty }_{c} (\ocal , \rd )$ denote the space of all $\rd $-valued functions of class ${\ccal }^{\infty }$ with compact supports contained in $\ocal $. We will use the following classical spaces
\begin{eqnarray*}
 &  &\vcal := \{ u \in {\ccal }^{\infty }_{c} (\ocal , \rd ) : \, \, \diver u= 0 \}  ,\\
 & & \rH := \mbox{the closure of $\vcal $ in ${L}^{2}(\ocal , \rd )$} , \\
 & & \rV := \mbox{the closure of $\vcal $ in ${H}^{1}(\ocal , \rd )$} .
\end{eqnarray*}
In the space $\rH$ we consider the inner product and the norm inherited from ${L}^{2}(\ocal , \rd )$ and
denote them by $\ilsk{\cdot }{\cdot }{\rH}$ and $|\cdot {|}_{\rH}$, respectively, i.e.
\[
\ilsk{u}{\rv}{\rH}:= \ilsk{u}{\rv}{{L}^{2}} , \qquad
 | u {|}_{\rH} := \norm{u}{{L}^{2}(\ocal)}{} , \qquad u, \rv \in \rH.
\]
In the space $\rV $ we consider the inner product inherited from ${H}^{1}(\ocal , \rd )$, i.e.
\begin{equation}  \label{E:V_il_sk}
 \ilsk{u}{\rv}{\rV } := \ilsk{u}{\rv}{{L}^{2}} + \dirilsk{u}{\rv}{},
\end{equation}
where
\begin{equation}  \label{E:il_sk_Dir}
 \dirilsk{u}{\rv}{}:=\ilsk{\nabla u}{\nabla \rv}{{L}^{2}} , \qquad u,\rv \in  \rV .
\end{equation}
Note that  the norm in $\rV $ satisfies
\begin{equation} \label{E:norm_V}
  \norm{u}{\rV}{2} :=  \norm{u}{}{2}+\norm{\nabla u}{{L}^{2}}{2},  \qquad \rv \in  \rV .
\end{equation}
We  will often use the notation $\Norm{\cdot}{}{}$ for the seminorm
\[
\Norm{u}{}{2}:=\dirilsk{u}{u}{}=\ilsk{\nabla u}{\nabla u}{{L}^{2}} , \qquad u \in  \rV .
\]

\bigskip  \noindent
A domain $\ocal $
satisfying the Poincar\'{e} inequality, i.e.   there exists a constant $C>0$
 such that
\begin{equation}\label{cond-Poincare}
C\int_{\mathcal{O}}\varphi^{2}\,d\xi\leq\int_{\mathcal{O}}|\nabla
\varphi|^{2}\,d\xi\quad\hbox{\rm for all $\varphi\in H^1_0({\mathcal O})$}
\end{equation}
will be called a Poincar\'{e} domain.
It is well known that, in the case when $\mathcal{O}$ is a  Poincar\'e domain,  the inner product in the space $\rV $
inherited from ${H}^{1}(\ocal , \rd )$, i.e. $ \ilsk{u}{\rv}{\rV } := \ilsk{u}{\rv}{{L}^{2}} + \dirilsk{u}{\rv}{}$ is equivalent to the following one:
\begin{equation}  \label{E:V_il_sk_Poincare}
  \ilsk{u}{\rv}{\mathcal{P}} :=  \dirilsk{u}{\rv}{} , \qquad u,\rv \in  \rV .
\end{equation}
\it In the sequel, if $\ocal $ is a  Poincar\'e domain, then in the space $\rV $ we consider the inner product
$\dirilsk{\cdot }{\cdot }{}$ given by \eqref{E:il_sk_Dir} and the corresponding norm $\Norm{\cdot }{}{}$. \rm

\bigskip  \noindent
Denoting by $\dual{\cdot }{\cdot }{}$ the dual pairing between $\rV $ and $\rV ^{\prime }$, i.e.
$\dual{\cdot }{\cdot }{}:=\ddual{{\rV }^\prime }{\cdot }{\cdot }{}{}_{\rV }$, by the Lax-Milgram Theorem, there exists a unique bounded linear operator $\acal: \rV  \to {\rV }^\prime $ such that
 we have the following equality
\begin{equation}  \label{E:Acal_ilsk_Dir}
   \dual{\acal u}{\rv }{} = \dirilsk{u}{\rv }{} , \qquad u,\rv \in \rV .
\end{equation}
The operator $\acal$ is closely related to the Stokes operator $\rA$ defined by
\begin{equation}
\label{eqn-Sokes}
\begin{array}{rcl}
D(\rA)&=&\{ u\in \rV: \acal u\in \rH\},\\
\rA u&=&\acal u, \mbox{ if } u\in D(\rA).
\end{array}
 \end{equation}
The Stokes operator $\rA$ is a non-negative self-adjoint operator in $\rH$. Moreover, if  $\ocal $ is a 2D or 3D Poincar\'{e} domain, see \eqref{E:Poincare-NS-ineq} below, then $\rA$ is strictly positive. We will not use the Stokes operator as in this paper we will be concerned only with the weak solutions to the stochastic Navier-Stokes equations, which in particular do not take values in the domain $D(\rA)$ of $\rA$.

\bigskip  \noindent
Let us consider the following tri-linear form
\begin{equation}  \label{E:form_b}
     b(u,w,\rv ) = \int_{\ocal  }\bigl( u \cdot \nabla w \bigr) \rv \, dx .
\end{equation}
We will recall fundamental properties of the form $b$.
By the Sobolev embedding Theorem (or Gagliardo-Nirenberg Inequality) we have, see for instance \cite[Lemmata III.3.3 and III.3.5]{Temam_2001},
\begin{eqnarray}
\label{ineq-GNI}
\norm{u}{{L}^{4}{(\ocal)}}{} &\leq &2^{1/4} \norm{u}{{L}^{2}{(\ocal)}}{1-\frac{d}4} \norm{\nabla u}{{L}^{2}{(\ocal)}}{\frac{d}4}, \;\; u\in H_0^{1,2}(\ocal), \quad \mbox{ for } d=2,3.
\end{eqnarray}
by applying  the H\"{o}lder inequality, we obtain the following estimates
\begin{eqnarray}
 |b(u,w,\rv )|= |b(u,\rv,w )| &\leq&  \norm{u}{{L}^{4}}{} \vert w\vert_{L^4} \vert \nabla \rv \vert_{L^2}
\label{E:b_estimate_L^4}
 \\
& \le & c \norm{u }{\rV }{} \Norm{w }{\rV }{} \Norm{\rv }{\rV}{} , \qquad u,w,\rv \in \rV    \label{E:b_estimate_V}
\end{eqnarray}
for some positive constant $c$. Thus the form $b$ is continuous on $\rV $, see also \cite{Temam_2001}.
Moreover, if we define a bilinear map $B$ by $B(u,w):=b(u,w, \cdot )$, then by inequality \eqref{E:b_estimate_V} we infer that $B(u,w) \in {\rV }_{}^{\prime }$ for all $u,w\in \rV $ and, by the Gagliardo-Nirenberg Inequality \eqref{ineq-GNI}) that the following inequality holds, for $d=2,3$,
\begin{eqnarray}  \nonumber
 |B(u,w) {|}_{{\rV }^{\prime }} \leq {c}_{1} \norm{u}{{L}^{4}}{} \norm{w}{{L}^{4}}{}
&\leq  &c_{2}
 \norm{u}{{L}^{2}}{1-\frac{d}4}\, \norm{ \nabla u}{{L}^{2}}{\frac{d}4}\, \norm{w}{{L}^{2}}{1-\frac{d}4} \, \norm{ \nabla w}{{L}^{2}}{\frac{d}4},
\label{E:estimate_B}
\\
   &\leq& c_3  \Norm{u }{\rV }{}\Norm{w }{\rV }{},\qquad\qquad\qquad\qquad u,w \in \rV  .
\nonumber
\end{eqnarray}
In particular, the mapping $B: \rV  \times \rV  \to {\rV }^{\prime } $ is bilinear and continuous.

\bigskip  \noindent
Let us also recall the following properties of the form $b$, see Temam \cite{Temam_2001}, Lemma II.1.3,
\begin{equation}  \label{E:antisymmetry_b}
b(u,w, \rv ) =  - b(u,\rv ,w), \qquad u,\rw,\rv \in \rV .
\end{equation}
In particular,
\begin{equation}  \label{E:wirowosc_b}
\dual{ B(u,\rv) }{\rv) }{} =b(u,\rv,\rv) =0   \qquad u,\rv \in \rV.
\end{equation}
We will need the following Fr\'{e}chet topologies.

\begin{definition}\label{def-Frechet topology}
\rm
By ${L}^{2}_{\mathrm{loc}}(\mathcal{O} , {\mathbb{R}}^{d})=\mathbb{L}^{2}_{\mathrm{loc}}$ we denote the space of all Lebesgue measurable $\rd $-valued functions $\rv$ such that $\int_{K}|\rv(x){|}^{2} \, dx < \infty $ for every compact subset $K \subset \ocal $. In this space we consider the Fr\'{e}chet topology generated by the family of seminorms
\[
   {p}_{R}:=  \Bigl( \int_{{\ocal }_{R}} |\rv(x){|}^{2} \, dx {\Bigr) }^{\frac{1}{2}} , \qquad R \in \nat ,
\]
where $({\ocal }_{R}{)}_{R \in \nat }$ is
an increasing sequence
of open bounded subsets of $\,\ocal $
with smooth boundaries and such that $\bigcup_{R\in \nat } {\ocal }_{R} =\ocal $.
\footnote{Such sequence  $({\ocal }_{R}{)}_{R \in \nat }$ always exist since it is sufficient to consider as $\ocal_R$ a smoothed out version of the set $\ocal \cap B(0,R)$, see for instance \cite{Schmidt_2014_approx} and references therein.}

By ${\rH }_{\mathrm{loc}}$  we denote    the space $\rH$  endowed with  the Fr\'{e}chet topology  inherited from the space
${L}^{2}_{\mathrm{loc}}(\mathcal{O} , {\mathbb{R}}^{d})$.
\end{definition}

Let us, for any $s>0$ define the following standard scale of
Hilbert spaces
\[
  {\rV}_{s} := \mbox{the closure of $\vcal $ in ${H}^{s}(\ocal , \rd )$} .
\]
If $s > \frac{d}{2} +1$ then by the Sobolev embedding Theorem,
\[
    {H}^{s-1}(\ocal  , \rd ) \embed  {\ccal }_{b}(\ocal , \rd )
   \embed {L}^{\infty } (\ocal , \rd ).
\]
Here ${\ccal }_{b}(\ocal , \rd )$ denotes the space of continuous and bounded $\rd $-valued functions defined on $\ocal $.
If $u,w \in \rV$ and $\rv \in {\rV}_{s}$ with $s > \frac{d}{2} +1$,  then for some constant ${c}_{} >0 $,
\begin{eqnarray*}
 |b(u,w,\rv)|  &= & |b(u,\rv,w)|
  \le   \norm{u}{{L}^{2}}{} \norm{w}{{L}^{2}}{} \norm{\nabla \rv}{{L}^{\infty }}{}
 \le  {c}_{} \norm{u}{{L}^{2}}{} \norm{w}{{L}^{2}}{} \norm{\rv}{{\rV}_{s}}{}.
\end{eqnarray*}
We have the following well know result  used in the proof of \cite[Lemma 5.4]{Brz+Motyl_2013}.
\begin{lemma}\label{lem-B-H} Assume that  $s > \frac{d}{2} +1$. Then
there exists a constant $C>0$ such that
\begin{equation}
\label{eqn-B-H}
\vert B(u,\rv)\vert_{\rV_s^\prime} \leq C \norm{u}{\rH}{}\norm{\rv}{\rH}{}, \;\; u,\rv\in \rV.
\end{equation}
Hence, in particular, there exists a unique bilinear and bounded map $\tilde{B}:\rH \times \rH\to {\rV }_{s}^{\prime }$ such that $B(u,\rv)=\tilde{B}(u,\rv)$ for all $u,\rv\in \rV$. \\
In what follows, the map $\tilde B$ will be denoted by $B$ as well.
\end{lemma}

\section{Stochastic Navier-Stokes equations}   \label{sec-SNSEs}

We begin this section with listing all the main assumptions.

\begin{assumption}\label{assumption-main}
We assume that the following objects are given.
\begin{description}
\item[\textbf{(H.1)}]   A separable Hilbert space $\rK $;
\item[\textbf{(H.2)}] a measurable map $G:\rH\to\lhs (\rK ,{\rV^{\prime }})$
 that
\begin{trivlist}
\item[(i)] is of linear growth, i.e. for some $C>0$
\begin{equation} \label{E:G*}
   \Norm{G(u)}{\lhs (\rK , {\rV }^{\prime })}{2} \le C (1 + {|u|}_{\rH}^{2}) , \qquad u \in \rH .
\tag{G1}
\end{equation}

\item[(ii)] $G(\rv) \in  \lhs (\rK ,\rH) $ for $\rv\in \rV $, and the restriction map $G: \rV  \to \lhs (\rK ,\rH) $ is Lipschitz continuous, i.e. there exists a constant $L>0$ such that
\begin{equation} \label{E:G_Lipsch}
 \norm{ G({u}_{1})-G({u}_{2}) }{\lhs (\rK,\rH)}{}
 \leq {L}^{} \Norm{{u}_{1}-{u}_{2}}{\rV }{}, \;\;\; u_1,u_2\in \rV .
 \tag{G2}
\end{equation}
\item[(iii)] for some constants ${\lambda }_{0}$, $\rho $ and $\eta \in (0,2]$,
\begin{equation} \label{E:G}
\norm{G(u )}{\lhs (\rK,\rH)}{2}
  \le (2-\eta ) \Norm{u}{}{2}+{\lambda }_{0}{|u|}_{\rH}^{2}+\rho , \qquad u \in \rV ,
\tag{G3}
\end{equation}

\item[(iv)] and,  for every $\psi \in \vcal $ the function
\begin{equation} \label{E:G**}
   {\psi}^{\ast \ast } G
  :   {\rH }_{\mathrm{loc}} \ni u \mapsto \Big\{\rK \ni y\mapsto
  \ddual{\rV^\prime }{G(u)y }{\psi }{}{}_{\rV}
  \in \mathbb{R}\Big\} \in \rK^\prime \mbox{ is continuous}.
   \tag{G4}
\end{equation}
\end{trivlist}
\item[(H.3)] A real number $p$ such that
\begin{equation} \label{eqn-p_cond}
p\in  \bigl[ 2, 2+ \frac{\eta }{2-\eta } \bigr) ,
\end{equation}
where we put $\frac{\eta }{2-\eta }=\infty$ when $\eta=2$.;
\item[(H.4)]
a Borel probability measure $\mu_0$ on $H$ such that $\int_H \vert x \vert^p \mu_0(dx)< \infty$ is given.
\item[(H.5)] an linear operator
$\acal: \rV  \to {\rV }^\prime $ satisfying   equality \eqref{E:Acal_ilsk_Dir}.
\end{description}
\end{assumption}

\bigskip

Now we state definition of a martingale solution of  equation \eqref{E:NS}.
We really need to consider the infinite time interval, i.e. $[0,\infty )$, however, we need also to state some of the results on the interval $[0,T]$, where $T>0$ is fixed.
Thus, in the following definition we distinguish between the two cases of solution on a finite interval $[0,T]$ and  on $[0,\infty )$.

\begin{definition}  \rm  \label{def-sol-martingale}
Let us assume Assumption \ref{assumption-main}.
Let $T>0$ be fixed. We say that there exists a \textbf{martingale solution}  of  the following stochastic Navier-Stokes Equations (in an abstract form)
 on the interval $[0,T]$
\begin{equation} \label{E:NS}
\begin{cases}
& du(t)+\acal u(t) \, dt+B\bigl( u(t),u(t) \bigr)\, dt = f(t) \, dt+G\bigl( u(t)\bigr) \, dW(t),
 \qquad t \geq 0 , \\
& \mathcal{L}( u(0)) = {\mu }_{0},
\end{cases}
\end{equation}
iff there exist
\begin{itemize}
\item a stochastic basis $\bigl( \hat{\Omega }, \hat{\fcal }, \hat{\fmath } ,
\hat{\p }  \bigr) $ with a complete filtration $\hat{\fmath }={\{ {\hat{\fcal }}_{t} \} }_{t \in [0,T]}$,
\item a $\rK$-cylindrical Wiener process $\hat{W}=\big(\hat{W})_{t \in [0,T]}$
\item and an $\hat{\fmath }$-progressively measurable process
$u: [0,T] \times \hat{\Omega }\to \rH$ with $\hat{\p } $-a.e. paths satisfying
\begin{equation}  \label{eqn-regularity of u}
  u(\cdot , \omega ) \in \ccal \bigl( [0,T], {\rH}_{w} \bigr)
   \cap {L}^{2}(0,T;\rV )
\end{equation}
such that
\[\mbox{the law on $\rH$ of $u(0)$ is equal to  ${\mu }_{0}$} \] and,  for all $ t \in [0,T] $ and all $v \in \vcal $,
\begin{eqnarray}
 \ilsk{u(t)}{v}{\rH} +  \int_{0}^{t} \dual{\acal u(s)}{v}{} \, ds
+ \int_{0}^{t} \dual{B(u(s))}{v}{} \, ds & & \nonumber \\
  = \ilsk{u(0)}{v}{\rH} +  \int_{0}^{t} \dual{f(s)}{v}{} \, ds
 + \Dual{\int_{0}^{t} G(u(s))\, d\hat{W}(s)}{v}{},  & &\quad \mbox{ $\hat{\p }$-a.s.}  \label{E:mart_sol_int_identity}
\end{eqnarray}
 and
\begin{equation}  \label{E:mart_sol_energy-ineq}
\hat{\mathbb{E}} \Bigl[ \sup_{t \in [0,T]} \norm{u(t)}{\rH}{2} + \int_{0}^{T} \norm{\nabla u(t)}{}{2} \, dt  \Bigr] < \infty .
\end{equation}
\end{itemize}
If all the above conditions are satisfied, then the  system
\[\bigl( \hat{\Omega }, \hat{\fcal }, \hat{\fmath } ,
\hat{\p }, \hat{W},u  \bigr)\]
will be called  a  martingale  solution     to  problem  \eqref{E:NS}  on the interval $[0,T]$ with the initial distribution ${\mu }_{0}$. \rm

A system $\bigl( \hat{\Omega }, \hat{\fcal }, \hat{\fmath } ,
\hat{\p }, \hat{W},u  \bigr) $ will be called a \textbf{martingale  solution} to  problem  \eqref{E:NS}  with the initial distribution ${\mu }_{0}$  iff all the above conditions are defined with the interval $[0,T]$ being replaced by $[0,\infty)$ and the condition \eqref{eqn-regularity of u} is replaced by

\begin{equation}  \label{eqn-regularity of u_infinity}
  u(\cdot , \omega ) \in \ccal \bigl( [0,\infty ), {\rH}_{w} \bigr)
   \cap {L}_{\mathrm{loc}}^{2}([0,\infty );\rV) ,
\end{equation}
and inequality \eqref{E:mart_sol_energy-ineq} holds for every $T>0$.

Here, ${\rH}_{w}$ denotes the Hilbert space $\rH$ endowed with the weak topology and
$\mathcal{C} ([0,T],{\rH}_{w})$ and $\mathcal{C} ([0,\infty ),{\rH}_{w})$  denote the spaces of $\rH$ valued weakly continuous functions defined on $[0,T]$ and $[0,\infty )$, respectively.

\end{definition}

In the case when ${\mu }_{0}$ is equal to the law on $\rH$ of a given random variable ${u}_{0}:\Omega \to \rH$ then,
somehow incorrectly, a martingale  solution to  problem  \eqref{E:NS}  will also be called a martingale  solution    to  problem  \eqref{E:NS}
 with the initial data $u_0$. Fully correctly, it should be called a martingale  solution    to  problem  \eqref{E:NS}
  with the initial data having the same law as $u_0$.
In particular, in this case we require that the laws on $\rH$ of $u_0$ and $u(0)$ are equal.

\it If no confusion seems likely, a system $\bigl( \hat{\Omega }, \hat{\fcal }, \hat{\fmath } ,
\hat{\p }, \hat{W},u  \bigr)$ from Definition \ref{def-sol-martingale}   will be called a \bf martingale solutions. \rm

\begin{remark}\label{rem-condition G}
Let us recall the following observation from \cite{Brz+Motyl_2013}.
 Since $\Norm{u}{}{}:= \norm{\nabla u}{{L}^{2}}{}$ and  $\dual{\acal u}{u}{} = \dirilsk{u}{u}{}:= \ilsk{\nabla u}{\nabla u}{{L}^{2}}$, we have
 $$ (2-\eta ) \Norm{u}{}{2}  =  2\dual{\acal u}{u}{}-\eta \Norm{u}{}{2} , \;\; u\in \rV.$$
Hence  inequality \eqref{E:G} can be written equivalently in the following form
\begin{equation} \label{A:G_2'}
 2 \dual{\acal u }{u }{} -  \Norm{G(u )}{\lhs (\rK ,\rH)}{2}
     \ge  \eta \Norm{u}{}{2} -{\lambda }_{0} {|u|}_{\rH}^{2} - \rho , \qquad u \in \rV ,
 \tag{G3'}
\end{equation}
\end{remark}

Inequality \eqref{A:G_2'}  is the same as considered by Flandoli and G\c{a}tarek in \cite{Flandoli+G_1995} for Stochastic NSEs in bounded domains.
The assumption $\eta = 2$ corresponds to the case when the noise term does not depend on $\nabla u$.
We will prove that the set of measures induced on appropriate space by the solutions of the Galerkin equations  is tight provided that the map $G$ from
 part  \textbf{(H.2)} of Assumption \ref{assumption-main} satisfies inequalities  \eqref{E:G*}  and \eqref{E:G}.
   Inequality \eqref{E:G*} and   condition \eqref{E:G**}  from  part  \textbf{(H.2)} of Assumption \ref{assumption-main} will be important in passing to the limit as $n\to \infty $ in the Galerkin approximation.  Condition  \eqref{E:G**}   is essential in the case  of unbounded domain $\mathcal{O}$.
It is wort mentioning that the following example of the noise term, analyzed in details in \cite[Section 6]{Brz+Motyl_2013}, is covered by part \textbf{(H.2)} of Assumption  \ref{assumption-main}.

\begin{example} \rm
Let us consider the noise term written classically as
\begin{equation}
 \bigl[ G(u)\bigr](t,x) \,d W(t):= \sum_{i=1}^{\infty } \bigl[
 \bigl( \bi (x) \cdot \nabla  \bigr) u (t,x) + \ci (x) u(t,x) \bigr]
d \betai(t) ,
\end{equation}
where
\begin{eqnarray*}
& & \betai , \quad i\in \nat, \mbox{ are  i.i.d. standard $\mathbb{R}$-valued Brownian Motions,} \\
& & \bi  : \overline{\ocal} \to \rd, \quad i \in \nat,  \mbox{ are functions  of class } {\ccal }^{\infty } \mbox{class},  \\
& & \ci  : \overline{\ocal} \to \rzecz, \quad i \in \nat,  \mbox{ are functions  of }   {\ccal }^{\infty } \mbox{ - of class } ,
\end{eqnarray*}
are given. Assume that
\begin{equation}  \label{E:C_1}
 {C}_{1}:=\sum_{i=1}^{\infty }\bigl( \| \bi {\| }_{{L}^{\infty }}^{2}+\| \diver \bi {\| }_{{L}^{\infty }}^{2} + \| \ci {\| }_{{L}^{\infty }}^{2}\bigr) < \infty
\end{equation}
and there exists $a \in (0,2]$ such that for all $\zeta = ({\zeta }_{1},...,{\zeta }_{d}) \in \rd  $ and all $x\in \mathcal{O}$,
\begin{equation}  \label{E:est_bi_bj_bk}
 \sum_{i=1}^{\infty } \sum_{j,k=1}^{d}
\bij(x)\bik(x) \bigr) {\zeta }_{j} {\zeta }_{k}
\le 2  \sum_{j,k=1}^{d} {\delta }_{jk} {\zeta }_{j} {\zeta }_{k} - a {|\zeta |}^{2}
 = (2-a ){|\zeta |}^{2} .
\end{equation}
This noise term can be reformulated in the following manner.
Let $\rK := {l}^{2}(\nat )$, where ${l}^{2}(\nat )$ denotes the space of all sequences $({h}_{i}{)}_{i \in \nat } \subset \rzecz $ such that $\sum_{i=1}^{\infty }{h}_{i}^{2} < \infty $. It is a Hilbert space with the scalar product given by
$
     \ilsk{h}{k}{{l}^{2}} := \sum_{i=1}^{\infty } {h}_{i} {k}_{i},
$
where $h=({h}_{i})$ and $k=({k}_{i})$ belong to ${l}^{2}(\nat )$.
Putting
\begin{equation} \label{E:G_def}
 G(u) h = \sum_{i=1}^{\infty } \;\bigl[
 \bigl( \bi  \cdot \nabla  \bigr) u  + \ci u \bigr]\; {h}_{i} ,
\qquad u \in \rV, \quad h=({h}_{i}) \in {l}^{2}(\nat ),
\end{equation}
we infer that
the mapping $G$ fulfils all conditions stated in assumption \textbf{(H.2)}\rm , see \cite[Section 6]{Brz+Motyl_2013} for details.
\end{example}

\begin{remark} \label{R:def-mart-sol-test}    \rm
Note that by Definition \ref{def-sol-martingale}  every solution to  problem  \eqref{E:NS} satisfies equality \eqref{E:mart_sol_int_identity} for all $\rv\in \vcal $.
However, equality \eqref{E:mart_sol_int_identity} holds not only for $\rv\in \vcal $ but also for all $\rv \in \rV $.
Indeed, this follows from the density of $\vcal $ in the space $\rV$ and the fact that each term in \eqref{E:mart_sol_int_identity} is well defined and continuous with respect to $\rv\in \rV $.
This remark is important while using the It\^{o} formula in the proof of Lemma \ref{lem-apriori estimates}.
\end{remark}

\begin{remark}  \label{R:Sol_cont_V'}
Let assumptions \textbf{(H.1)}-\textbf{(H.5)} be satisfied.
If the  system $ ( \hat{\Omega }, \hat{\fcal }, \hat{\fmath } , \hat{\p }, \hat{W},u ) $
 is a martingale solution of problem \eqref{E:NS} on the interval $[0,\infty )$,
then $\hat{\p }$-a.e. paths of  the process $u(t)$, $t\in [0,\infty )$, are ${\rV }^{\prime }$-valued continuous functions, i.e.
for $\hat{\p } $-a.e. $\omega \in \hat{\Omega }$
\begin{equation}  \label{E:mart_sol_cont_V'}
  u(\cdot , \omega ) \in \ccal \bigl( [0,\infty ), {\rV }^{\prime } \bigr) ,
\end{equation}
and equality \eqref{E:mart_sol_int_identity} can be rewritten as the following one, understood  in the space ${\rV }^{\prime }$,
\begin{equation}
 u(t) +  \int_{0}^{t} \acal u(s) \, ds
+ \int_{0}^{t} B(u(s)) \, ds
  = u(0) +  \int_{0}^{t} f(s)\, ds
 + \int_{0}^{t} G(u(s))\, d\hat{W}(s) , \ \ t \in [0,\infty ) .   \label{E:mart_sol_int_identity_V'}
\end{equation}
\end{remark}

\begin{proof}
Let us fix any $T>0$.
Let us notice that since the map $G$ satisfies inequality  \eqref{E:G*} in Assumption \ref{assumption-main},  by  inequality  \eqref{E:mart_sol_energy-ineq} we infer that
\[
   \hat{\e } \Bigl[ \int_{0}^{T} \norm{G(u(s))}{\lhs (\rK ,{\rV }^{\prime })}{2} \, ds  \Bigr]
   \le   C\,  \hat{\e } \Bigl[ \int_{0}^{T} (1+\norm{u(s)}{\rH}{2}) \, ds  \Bigr] < \infty .
\]
Thus the process $\mu$ defined by
\[
  \mu (t):= \int_{0}^{t} G(u(s)) \, d \hat{W}(s), \quad  t \in [0,T],
\]
is a ${\rV }^{\prime }$-valued square integrable  continuous martingale.

\bigskip  \noindent
\bf Remark. \it
The process $\mu$ is an $\rH$-valued square integrable continuous martingale, as well. \rm

\begin{proof}
Since the map $G$ satisfies inequality  \eqref{E:G} in Assumption \ref{assumption-main},  using   inequality \eqref{E:mart_sol_energy-ineq} we deduce that
\[
   \hat{\e } \Bigl[ \int_{0}^{T} \norm{G(u(s))}{\lhs (\rK ,\rH)}{2} \, ds  \Bigr]
   \le    \hat{\e } \Bigl[ \int_{0}^{T} [(2-\eta ) \Norm{u(s)}{}{2} +{\lambda }_{0}\norm{u(s)}{\rH}{2} + \rho ] \, ds  \Bigr] < \infty .
\]
Thus $\mu (t)$, $t\in [0,T]$, is an $\rH$-valued square integrable continuous martingale.
\end{proof}

In the framework of Remark \ref{R:Sol_cont_V'}, by the regularity assumption \eqref{eqn-regularity of u}, we infer that for $\hat{\p }$-a.e. $\omega \in \hat{\Omega }$
\[
\acal u(\cdot ,\omega ) \in {L}^{2}(0,T;{\rV }^{\prime }),  \quad
B(u(\cdot ,\omega ),u(\cdot ,\omega )) \in {L}^{4/3}(0,T;{\rV }^{\prime }).
\]
By assumption (H.3), in particular, $f\in {L}^{p}(0,T;{\rV }^{\prime })$. Hence for $\hat{\p }$-a.e. $\omega \in \hat{\Omega }$   the functions
\begin{eqnarray*}
  && [0,T] \ni t \mapsto \int_{0}^{t}  \acal u(s,\omega )\, ds  \in {\rV }^{\prime } , \\
  && [0,T] \ni t \mapsto \int_{0}^{t}  B( u(s,\omega ),(u(s,\omega ))\, ds  \in {\rV }^{\prime }, \\
  && [0,T] \ni t \mapsto \int_{0}^{t}   f(s)\, ds  \in {\rV }^{\prime }
\end{eqnarray*}
are well defined and continuous. Using \eqref{E:mart_sol_int_identity} we infer that for $\hat{\p }$-a.e. $\omega \in \hat{\Omega }$
\[
u(\cdot ,\omega ) \in \mathcal{C} ([0,T],{\rV }^{\prime })
\]
and for every $t\in [0,T]$ equality \eqref{E:mart_sol_int_identity_V'} holds.
Since $T>0$ has been chosen in an arbitrary way, regularity condition \eqref{E:mart_sol_cont_V'} and equality \eqref{E:mart_sol_int_identity_V'} hold.
The proof of the claim is thus complete.
\end{proof}

\section{The continuous dependence of the solutions on the initial state and the external forces in 2D and 3D domains}
 \label{sec-Cont-dependence}

In this section we will concentrate on martingale solutions to  problem  \eqref{E:NS}  on a fixed interval $[0,T]$. The main result is Theorem \ref{thm-continuous dependence-existence}.
We will also need some modification of Theorem 5.1 in  \cite{Brz+Motyl_2013}, contained in Theorem \ref{T:existence_NS}.

\bigskip
As in \cite{Brz+Motyl_2013}
in the proofs we will use the following structure. Let us fix $s >  \frac{d}{2}+1$
and notice that the space ${\rV }_{s}$
is dense in $\rV$ and the natural embedding $ {\rV}_{s} \embed \rV$ is continuous. By  \cite[Lemma  2.5]{Holly+Wiciak_1995}, see also \cite[Lemma  C.1]{Brz+Motyl_2013}, there exists  a separable Hilbert space $U$
such that  $U$ is a dense subset of  $ {\rV}_{s}$ and
\begin{equation} \label{eqn-U}
 \mbox{\it the natural embedding ${\iota }_{s}: U \hookrightarrow  {\rV}_{s}$ is compact \rm }.
\end{equation}
Then we also have
\begin{equation}\label{eqn-U'}
U \embed \rV_s \embed \rH \cong {\rH }^{\prime } \embed \rV_s^\prime\embed U^{\prime } ,
\end{equation}
where ${\rH }^{\prime }$ and $U^{\prime }$ are the dual spaces of $\rH$ and $U$, respectively, ${\rH}^{\prime }$ being identified with $\rH$
and
the dual embedding  ${\rH }^{\prime } \embed U^{\prime }$ is compact as well.

\bigskip
In the next definition we will recall definition of a topological space ${\mathcal{Z}}_{T}$  that plays an important r\^ole in our approach, see  page 1629 and Section 3 in \cite{Brz+Motyl_2013}.

\bigskip  \noindent
To define the space ${\mathcal{Z}}_{T}$ we will need the following four spaces.
\begin{eqnarray*}
\ccal ([0,T],U^{\prime })&:=& \mbox{ the space of  continuous functions } u:[0,T] \to U^{\prime }
                          \mbox{ with the topology }  \nonumber \\
                &         & \mbox{ induced by the norm }
                          \norm{u}{\ccal ([0,T],U^{\prime })}{}:=\sup_{t\in[0,T]} |u(t){|}_{U^{\prime }}\\
 {L}^{2}_{w}(0,T;\rV )  &:=& \mbox{the  space ${L}^{2}(0,T;\rV )$ with the weak topology}, \\
 {L}^{2}(0,T;{\rH }_{\mathrm{loc}}) & :=& \mbox{the space of all measurable functions $u:[0,T]\to \rH$ such that for all $R\in \nat $}  \\
 && \qquad  {p}_{T,R}(u) := {\biggl( \int_{0}^{T}\int_{{\mathcal{O}}_{R}} {|u(t,x)|}^{2} \, dx dt  \biggr) }^{\frac{1}{2}} < \infty  \\
 && \mbox{with the topology generated by the seminorms ${({p}_{T,R})}_{R\in \nat }$.}
\end{eqnarray*}
Let ${\rH}_{w}$ denote the Hilbert space $\rH$ endowed with the weak topology and
let us put
\begin{eqnarray*}
& \ccal ([0,T];{\rH}_{w}) : = & \mbox{the space of weakly continuous functions }
                          u : [0,T] \to \rH \mbox{ endowed with }\nonumber \\
  &                     & \mbox{the weakest topology  such that for all
                          $h \in \rH $   the  mappings } \nonumber \\
  &                     &  \ccal ([0,T];{\rH}_{w}) \ni u \mapsto \ilsk{u(\cdot )}{h}{\rH}
                           \in \ccal  ([0,T];\rzecz )
                        \mbox{ are continuous. }
\end{eqnarray*}

\begin{definition}\label{def-spaces-Z}
\rm For $T>0$ let us put
\begin{equation}
\label{eqn-Z_T}
  \mathcal{Z}_T : = \ccal ([0,T]; U^{\prime }) \cap {L}_{w}^{2}(0,T;\rV)  \cap {L}^{2}(0,T;{\rH}_{\mathrm{loc}}) \cap \ccal ([0,T];{\rH}_{w})
\end{equation}
and let  $\mathcal{T}_T $ be   the supremum of the corresponding four topologies, i.e. the smallest topology on $\mathcal{Z}_T$ such that the four natural embeddings from $\mathcal{Z}_T$ are continuous. \\
The space  ${\mathcal{Z}}_T$  will  also considered with the Borel $\sigma $-algebra, i.e. the smallest $\sigma $-algebra containing the family $\mathcal{T}_T $.
\end{definition}

The following auxiliary result which is needed in the proof of Theorem
 \ref{thm-continuous dependence-existence}, cannot be deduced directly from the Kuratowski Theorem, see Counterexample C.\ref{counterexample-Kuratowski} in the  \ref{sec:Kuratowski}.

\begin{lemma}\label{lem-Kuratowski} Assume that $T>0$. Then the following fours sets
$\ccal ([0,T]; \rH)\cap \mathcal{Z}_T$, $\ccal ([0,T]; \rV)\cap \mathcal{Z}_T$, $L^2(0,T;V)\cap \mathcal{Z}_T $ and $\ccal ([0,T]; \rV^\prime)\cap \mathcal{Z}_T$ are  Borel subsets of $\mathcal{Z}_T$ and the corresponding embedding   tranforms Borel sets into Borel subsets of $\mathcal{Z}_T$. Moreover, the following $\mathbb{R}_+\cup\{+\infty\}$-valued functions
\begin{eqnarray*}
&&\mathcal{Z}_T \ni u \mapsto \begin{cases}\sup_{s \in [0,T]} \norm{u (s)}{\rH}{2}, & \mbox{ if } u\in \ccal ([0,T]; \rH)\cap \mathcal{Z}_T \\
  \infty, &\mbox{ otherwise}, \end{cases}
  \\
  &&\mathcal{Z}_T \ni u \mapsto \begin{cases} \int_{0}^T \Norm{u (s)}{}{2}\, ds, & \mbox{ if } u\in L^2(0,T;\rV)\cap \mathcal{Z}_T, \\ \infty &\mbox{ otherwise}, \end{cases}
\end{eqnarray*}
are Borel.
\end{lemma}

\begin{proof}  Because $\ccal ([0,T]; U^{\prime }) \cap {L}^{2}(0,T;{\rH}_{\mathrm{loc}})$ is a Polish space, by the Kuratowski Theorem $\ccal ([0,T]; \rH)$   is   Borel subset of $\ccal ([0,T]; U^{\prime }) \cap {L}^{2}(0,T;{\rH}_{loc})$.  Hence the intersection
$
\ccal ([0,T]; \rH) \cap \mathcal{Z}_T $ is a Borel subset of the intersection  $\ccal ([0,T]; U^{\prime }) \cap {L}^{2}(0,T;{\rH}_{loc}) \cap \mathcal{Z}_T$ which happens to be equal to  $\mathcal{Z}_T$.\\
 We can argue in the same way in the case of the spaces  $\ccal ([0,T]; \rV)\cap \mathcal{Z}_T$ and $\ccal ([0,T]; \rV^\prime)\cap \mathcal{Z}_T$.\\
The proof in case the space $L^2(0,T;\rV ) $ is analogous, one needs to begin with an observation that by the Kuratowski Theorem the set $L^2(0,T;\rV ) $ is Borel subset of $ {L}^{2}(0,T;{\rH }_{loc})$. We have used a fact that a product of Borel set in $\ccal ([0,T]; U^{\prime }) \cap {L}^{2}(0,T;{\rH}_{loc})$ and the set $\mathcal{Z}_T$ is a Borel subset of the latter. \\
The same argument applies to the proof that $i_T$ and $j_T$ map Borel subsets of their corresponding domains to Borel sets in $\mathcal{Z}_T$.
The  last  part of Lemma is a consequence  Proposition C.\ref{prop:Kuratowski}.
\end{proof}

\subsection{Tightness criterion and  Jakubowski's version of the Skorokhod theorem}

One of the main tools in this section is the tightness criterion in the space ${\mathcal{Z}}_{T}$ defined in identity  \eqref{eqn-Z_T}.
We will use a slight generalization of of the criterion stated in  Corollary 3.9 from \cite{Brz+Motyl_2013},  compare with the proof of Lemma 5.4 therein.
Namely, we will consider the sequence of stochastic processes defined on their own probability spaces. Let $({\Omega }_{n},{\mathcal{F}}_{n}, {\mathbb{F}}_{n},{\mathbb{P}}_{n})$, $n \in \nat $, be a sequence of probability spaces with the filtration ${\mathbb{F}}_{n}={({\mathcal{F}}_{n,t})}_{t \ge 0 }$.

\begin{cor} \bf (tightness criterion) \it \label{cor-tigthness_criterion}
Assume that  $(\Xn {)}_{n \in \nat }$ is  a sequence of continuous ${\mathbb{F}}_{n}$-adapted $U^{\prime }$-valued processes defined on ${\Omega }_{n}$ and
 such that
 \begin{eqnarray}
 \label{cond-a}
 &&        \sup_{n\in \nat} \mathbb{E}_n \, \bigl[ \sup_{s \in [0,T]} | \Xn (s) {|}_{\rH}^{2}  \bigr]< \infty,\\
 \label{cond-b}
 &&         \sup_{n\in \nat} \mathbb{E}_{n} \,\biggl[ \int_{0}^{T} \Norm{\Xn (s)}{}{2} \, ds    \biggr] <\infty,
 \end{eqnarray}
\begin{description}
\item[(a)]
 and for every $\, \eps >0$ and for every  $\eta >0$  there exists $ \delta >0 $ such that for every
sequence $({{\tau}_{n} } {)}_{n \in \nat }$ of  $[0,T]$-valued  $\mathbb{F}_n$-stopping times
 one has
\begin{eqnarray}
 \label{cond-c}
    \sup_{n \in \nat} \, \sup_{0 \le \theta \le \delta }  {\p }_{n}\bigl\{\,
    \vert {X}_{n} ({\tau }_{n} +\theta )- {X}_{n} ( {\tau }_{n}  )\vert_{U^\prime}  \geq \eta \,\bigr\}  \le \eps .
\end{eqnarray}
\end{description}
Let ${\tilde{\p }}_{n}$ be the law of $\Xn $ on the Borel $\sigma$-field $\mathcal{B}(\zcal_T)$. Then for every $\eps >0 $ there exists a compact subset ${K}_{\eps }$ of $\zcal_T $ such that
\[
  \sup_{n\in \nat} {\tilde{\p }}_{n}({K}_{\eps })\ge 1-\eps .
\]
\end{cor}

The proof of Corollary \ref{cor-tigthness_criterion} is essentially the same as the proof of
\cite[Corollary 3.9]{Brz+Motyl_2013}.

\bigskip
If the sequence $(\Xn {)}_{n \in \nat }$ satisfies condition \textbf{(a)} then we say that it satisfies the Aldous condition \textbf{[A]}  in $U^{\prime }$ on [0,T]. If it satisfies condition \textbf{(a)} for each $T>0$,  we say that it satisfies the Aldous condition \textbf{[A]} in $U^{\prime }$.

Obviously, the class of $U^\prime$-valued processes satisfying the Aldous condition is a real vector space.  Below we will formulate a sufficient condition for the Aldous condition. This idea has been used in the proof of Lemma 5.4 in \cite{Brz+Motyl_2013} but it has not been formulated in such a way.

\begin{lemma}\label{lem-Aldous} Assume that $Y$ is a separable Banach space, $\sigma\in (0,1]$ and that  $(\un {)}_{n \in \nat }$ is  a sequence of continuous $\mathbb{F}_{n}$-adapted $Y$-valued processes indexed by $[0,T]$ for some $T>0$,   such that
\begin{description}
\item[(a')]  there exists $C>0$ such that   for every $\theta>0$ and for every
sequence $({{\tau}_{n} } {)}_{n \in \nat }$ of $[0,T]$-valued $\mathbb{F}_{n}$-stopping times with
 one has
\begin{eqnarray}
& &\mathbb{E}_{n}\,\bigl[ {| \un (\taun + \theta ) - \un(\taun) | }_{Y}  \bigr]
 \leq C\theta^\sigma.   \label{eqn-Aldous-sufficient}
\end{eqnarray}
\end{description}
Then the sequence   $(\un {)}_{n \in \nat }$ satisfies the Aldous condition \textbf{[A]} on $[0,T]$.
\end{lemma}

\begin{proof}
Let us fix $\eta > 0 $ and $\eps >0$. By the Chebyshev inequality and the estimate \eqref{eqn-Aldous-sufficient}  we obtain
\[
{\p }_{n}\bigl( \bigl\{
{| \un (\taun + \theta ) - \un(\taun) | }_{Y} \ge \eta \bigr\} \bigr)
 \le \frac{1}{\eta } \mathbb{E}_{n} \, \bigl[ {| \un (\taun + \theta ) - \un(\taun) |}_{Y}  \bigr]
 \le \frac{C \cdot \theta^\sigma}{\eta }  ,  \qquad n \in \nat.
\]
  Let us ${\delta }:= \bigl[\frac{\eta \cdot \eps }{C}\bigr]^{\frac1{\sigma}}  $. Then we have
\[
 \sup_{n\in \nat }\sup_{1 \le \theta \le {\delta }}  {\p }_{n}\bigl\{
  {| \un (\taun + \theta ) - \un(\taun) | }_{Y} \ge \eta \bigr\} \le \eps  ,
\]
This completes the proof of Lemma \ref{lem-Aldous}.
\end{proof}

\begin{remark}\label{rem-additional}

 \rm As can be seen in \eqref{eqn-Z_T}, the space ${\mathcal{Z}}_{T}$ is defined as an intersection of four spaces, one of them being  the space $\mathcal{C}([0,T];{U}^{\prime })$. The latter space plays, in fact, only an auxiliary r\^{o}le. Let us recall that the space $U$, see \eqref{eqn-U} and \cite[Section 2.3]{Brz+Motyl_2013}, is important in the construction of the solutions to stochastic  Navier-Stokes equations via the Galerkin method in  the case of an unbounded domain, i.e.  when the embedding $\rV \subset \rH$ is not compact. (In the case of a bounded domain we can take, e.g. $U:={\rV }_{s}$ for sufficiently large $s$.) In particular, the orthonormal basis of the space $\rH$, which we use in the Galerkin method is contained in $U$, so the Galerkin solutions ''live in'' the space $U$.

With the space $U$ in hand,   in \cite{Brz+Motyl_2013} we prove an appropriate compactness and tightness criteria in the space ${\mathcal{Z}}_{T}$, see \cite[Lemma 3.3 and Corollary 3.9]{Brz+Motyl_2013}. Let us emphasize that in order to prove the relative compactness of an appropriate set in the Fr\'{e}chet space ${L}^{2}(0,T;{\rH}_{loc})$ first we need to prove a certain  generalization of the classical Dubinsky Theorem, see \cite[Lemma 3.1]{Brz+Motyl_2013}, where the space $\mathcal{C}([0,T];{U}^{\prime })$ is used. This result  is related to the Aldous condition in the space ${U}^{\prime }$ in the tightness criterion, \eqref{cond-c} in Corollary \ref{cor-tigthness_criterion} and \cite[Corollary 3.9(c)]{Brz+Motyl_2013}.

We will use Corollary \ref{cor-tigthness_criterion}  to prove Theorems \ref{thm-continuous dependence-tightness} and \ref{thm-continuous dependence-existence}, below.
Even though, the presence  of the space $\mathcal{C}([0,T];{U}^{\prime })$ in the definition of the space ${\mathcal{Z}}_{T}$ is natural in the context of the Galerkin approximation solutions, it's presence in the context of Theorems \ref{thm-continuous dependence-tightness} and \ref{thm-continuous dependence-existence} where  we consider sequences of the solutions of the Navier-Stokes equations seems to be unnecessary.
However, again because of the lack of the compactness of the embedding $\rV \subset \rH $ to prove tightness in Theorem \ref{thm-continuous dependence-tightness} we still use Corollary \ref{cor-tigthness_criterion} in its original form.
\end{remark}

In the proofs of the theorems on the existence of a martingale solution and  on the continuous dependence of the data we use a version of the Skorokhod theorem for nonmetric spaces.
For convenience of the reader let us recall  the following Jakubowski's \cite{Jakubowski_1998} version of the Skorokhod Theorem,
see also Brze\'{z}niak and Ondrej\'{a}t  \cite{Brz+Ondrejat_2011}.

\begin{theorem} \label{T:2_Jakubowski} \rm (Theorem 2 in \cite{Jakubowski_1998}). \it
Let $(\mathcal{X} , \tau )$ be a topological space such that there exists a sequence $({f}_{m}) $ of continuous functions ${f}_{m}:\mathcal{X}  \to \rzecz $ that separates points of $\mathcal{X} $.
Let $({X}_{n})$ be a sequence of  $\mathcal{X} $-valued Borel random variables. Suppose that for every $\eps >0$
there exists a compact subset ${K}_{\eps} \subset \xcal $ such that
\[
  \sup_{n \in \nat } \p (\{ {X}_{n} \in {K}_{\eps } \} ) > 1-\eps .
\]
Then there exists a subsequence $({X}_{{n}_{k}}{)}_{k\in \nat }$, a sequence $({Y}_{k}{)}_{k\in \nat }$ of $\mathcal{X} $-valued Borel random variables and an $\mathcal{X} $-valued Borel random variable $Y$ defined on some probability space  $(\Omega , \fcal ,\p )$ such that
\[
    \lcal ({X}_{{n}_{k}}) = \lcal ({Y}_{k}), \qquad k=1,2,...
\]
and for all $ \omega \in \Omega $:
\[
  {Y}_{k}(\omega ) \stackrel{\tau }{\longrightarrow } Y(\omega )  \quad \mbox{ as } k \to \infty .
\]
\end{theorem}

Note that the sequence $({f}_{m})$ defines another, weaker topology on $\mathcal{X}$. However, this topology restricted to $\sigma $-compact subsets of $\mathcal{X}$ is equivalent to the original topology $\tau $.
Let us emphasize that thanks to the assumption on the tightness of the set of laws $\{ \mathcal{L} ({X}_{n}) , n \in \nat \}$ on the space $\mathcal{X} $ the maps $Y$ and ${Y}_{k}$, $k \in \nat $, in Theorem \ref{T:2_Jakubowski} are measurable with respect to the Borel $\sigma $-field in the space $\mathcal{X} $.

The following result has been proved in the proof of \cite[Corollary 3.12]{Brz+Motyl_2013} for the spaces $\mathcal{Z}_T$.
\begin{lemma}\label{lem-Z_T} The topological space ${\mathcal{Z}}_{T}$  satisfies the assumptions of Theorem \ref{T:2_Jakubowski}.
\end{lemma}

\subsection{The existence and properties of martingale solutions on $[0,T]$}  \label{S:Existence}

We will concentrate on martingale solutions to  problem  \eqref{E:NS}  on a fixed interval $[0,T]$.
The following result is a slight generalisation of Theorem 5.1 in  \cite{Brz+Motyl_2013}.
In comparison to \cite{Brz+Motyl_2013} the deterministic initial state has been replaced by the random one satisfying assumption (H.3). However, our attention will be focused on the estimates satisfied by the solutions of the Navier-Stokes equations.
We claim that there exists a solution $u$ satisfying estimate
$\hat{\e} \bigl[ \sup_{t \in [0,T]} \norm{u(t)}{\rH}{q} \bigr] \le {C}_{1}(p,q) $ for every $q \in [2,p]$,
(and not only for $q=2$ as stated in inequality (5.1) in \cite{Brz+Motyl_2013}). Moreover, we analyse what is the relation between the constant ${C}_{1}(p,q)$ and the initial state ${u}_{0}$ and the external forces $f$. The same concerns the estimate on $\hat{\mathbb{E}} [\int_{0}^{T} \Norm{u(t)}{}{2} \, dt ]$.
These results generalise \cite[Theorem 5.1]{Brz+Motyl_2013}.
In the second part of Theorem \ref{T:existence_NS} we will prove another estimate on $u$  in the case when $\ocal $ is a 2D or 3D Poincar\'{e} domain, see \eqref{E:Poincare-NS-ineq} below. This estimate will be of crucial importance in the proof of existence of an invariant measure in 2D case. The proof of Theorem \ref{T:existence_NS} is based on the Galerkin method. The analysis of the Galerkin equations is postponed to \ref{App:Galerkin-est}.
Recall also that in assumption (H.3) we have put $\frac{\eta }{2-\eta }=\infty$ when $\eta=2$.

\begin{theorem} \label{T:existence_NS}
Let assumptions \textbf{(H.1)}-\textbf{(H.5)} be satisfied.
In particular, we assume that  $p$ satisfies \eqref{eqn-p_cond}, i.e.
\begin{equation*}
p\in  \bigl[ 2, 2+ \frac{\eta }{2-\eta } \bigr) ,
\end{equation*}
where  $\eta \in (0,2]$ is given in assumption \textbf{(H.2)}.
\begin{itemize}
\item[(1) ] For every  $T>0 $ and ${R}_{1}, {R}_{2}>0$
 if  ${\mu }_{0}$ is a Borel probability measure on $\rH$,
${f} \in  {L}^{p}([0,\infty ); \rV^\prime )$ satisfy $\int_{\rH} \vert x\vert^p {\mu }_{0}(dx) \leq R_1$
 and $ \vert f{\vert}_{{L}^{p}(0,T;\rV^\prime)}\leq R_2$, then there exists a \textbf{martingale  solution} $\bigl( \hat{\Omega }, \hat{\fcal }, \hat{\fmath } ,
\hat{\p }, \hat{W},u  \bigr) $   to  problem  \eqref{E:NS}  with the initial law ${\mu }_{0}$   which satisfy the following  estimates: for every $q\in [1,p]$ there exist constants $C_{1}(p,q)$ and $C_{2}(p)$,
depending also on $T$, ${R}_{1}$ and ${R}_{2}$, such that
\begin{equation} \label{E:H_estimate_NS}
\hat{\mathbb{E}} \bigl( \sup_{s\in [0,T] } {|u(s)|}_{\rH}^{q} \bigr) \le {C}_{1}(p,q) ,
\end{equation}
putting ${C}_{1}(p):={C}_{1}(p,p)$, in particular,
\begin{equation} \label{E:H_estimate_NS_p}
\hat{\mathbb{E}} \bigl( \sup_{s\in [0,T] } {|u(s)|}_{\rH}^{p} \bigr) \le {C}_{1}(p) ,
\end{equation}
and
\begin{equation} \label{E:V_estimate_NS}
 \hat{\mathbb{E}} \bigl[ \int_{0}^{T} \norm{ \nabla u (s)}{{L}^{2}}{2} \, ds \bigr] \le {C}_{2}(p).
\end{equation}
\item[(2) ]
Moreover, if $\ocal $ is a Poincar\'{e} domain and  the map $G$ satisfies inequality  \eqref{E:G} in Assumption \ref{assumption-main} with ${\lambda }_{0}=0$, then there exists  a martingale solution
$\bigl( \hat{\Omega }, \hat{\fcal }, \hat{\fmath },\hat{\p } ,u \bigr) $ of problem \eqref{E:NS}   satisfying additionally the following inequality for every $T>0$
\begin{equation}  \label{E:Poincare-NS-ineq}
    \frac{\eta }{2} \hat{\e } \biggl[ \int_{0}^{T} \norm{\nabla u (s)}{{L}^{2}}{2} \, ds \biggr]  \biggr)
   \le \hat{\mathbb{E}} [\, \norm{u(0)}{\rH}{2} \, ]
 + \frac{2}{\eta } \int_{0}^{T} \norm{f(s)}{{\rv }^{\prime }}{2} \, ds + \rho T.
\end{equation}
\end{itemize}
\end{theorem}

Proof of Theorem \ref{T:existence_NS} is postponed to  \ref{sec:Proof of Theorem_existence_NS}.

\subsection{The continuous dependence}

We prove the following results related to the continuous dependence on the deterministic initial condition and deterministic external forces.
Roughly speaking, we will show that if $({u}_{0,n}) \subset \rH$ and $({f}_{n}) \subset {L}^{p}(0,T;{\rV }^{\prime })$ are sequences of initial conditions and external forces   approaching  ${u}_{0}\in \rH$ and $f\in {L}^{p}(0,T,{\rV }^{\prime })$, respectively, then a sequence $(\un )$  of martingale solutions of the Navier-Stokes equations with the data $({u}_{0,n},{f}_{n})$, satisfying inequalities \eqref{E:H_estimate_NS}-\eqref{E:V_estimate_NS},
contains a subsequence  of solutions, on a changed probability basis,  convergent to a martingale solution  with the initial condition ${u}_{0}$ and the external force $f$. Note that existence of such solutions ${u}_{n}$, $n \in \nat $, is guaranteed by Theorem \ref{T:existence_NS}. This result holds both in 2D and 3D possibly unbounded domains with smooth boundaries.
Moreover, in the case of $2D$ domains, because of the existence and uniqueness of the strong solutions, stronger result holds.
Namely, the solutions ${u}_{n}$, $n \in \nat $, satisfy inequalities \eqref{E:H_estimate_NS}-\eqref{E:V_estimate_NS} and not only a subsequence but the whole sequence of solutions $({u}_{n})$ is convergent to the solution of the Navier-Stokes equation with the data ${u}_{0}$ and $f$.
Their proofs are de facto, modifications of the proofs of corresponding parts of Theorem 5.1 from \cite{Brz+Motyl_2013}, where Galerkin approximations are substituted by solutions $\un $, $n \in \nat $. However, the last part of the proof is different. Namely, contrary to the case of the Galerkin aproximations, the martingale ${\tilde{M}}_{n}$ defined by (5.16) in \cite{Brz+Motyl_2013} is, in general, not square integrable. It would be square integrable, for example, if inequality  \eqref{E:H_estimate_NS} held with some $q>4$.
This holds in the case, when the noise term does not depend on $\nabla u$ or if we impose such restriction on $\eta $ that $\frac{\eta }{2-\eta } >4$. However,  to cover the general case, this part of the proof is different.

In what follows we do not assume that $\mathcal{O}$ is a Poincar\'{e} domain.

\begin{theorem} \label{thm-continuous dependence-tightness}
Let assumptions \textbf{(H.1)}-\textbf{(H.3)}  and \textbf{(H.5)} be satisfied and let $T>0$.
Assume that  $\, \, \big({u}_{0,n}\big)_{n=1}^\infty$ is a bounded  $\rH$-valued  sequence and
$\, {({f}_{n})}_{n=1}^{\infty } \, $ is a bounded  ${L}^{p}(0,T;{\rV }^{\prime })$-valued sequence.
Let ${R}_{1}>0$ and ${R}_{2}>0$ be such that
$\sup_{n \in \nat } \norm{{u}_{0,n}}{\rH}{} \le {R}_{1}$ and $\sup_{n \in \nat } \Norm{{f}_{n}}{{L}^{p}(0,T;{\rV }^{\prime })}{} \le {R}_{2}$.
Let
 \[
  \bigl( \hat{\Omega }_n, \hat{\fcal}_n, \hat{\fmath }_n,\hat{\p }_n ,{\hat{W}}_{n}, u_n \bigr)
\]
be a   martingale solution
of problem \eqref{E:NS} with the initial data ${u}_{0,n}$ and the external force $f_n$
and satisfying inequalities \eqref{E:H_estimate_NS}-\eqref{E:V_estimate_NS}.
Then,  the set of
Borel measures $\bigl\{ \lcal ({u_n}{} ) , n \in \nat  \bigr\} $ is tight on the
space $(\mathcal{Z}_{T} , \mathcal{T}_{T} )$.
\end{theorem}

\begin{proof}
Let us fix $T>0$ and $p$ satisfying condition \eqref{eqn-p_cond}.
Let   $\, \, \big({u}_{0,n}\big)_{n=1}$ and
$\, \, \big({f}_n\big)_{n=1}$
be  bounded  $\rH $-valued, resp. ${L}^{p}(0,T;{\rV }^\prime )$-valued, sequences.
Let
 \[\bigl( \hat{\Omega }_n, \hat{\fcal}_n, \hat{\fmath }_n,\hat{\p }_n ,{\hat{W}}_{n},u_n \bigr) \]
 be a   corresponding martingale solution  of problem \eqref{E:NS} with the initial data ${u}_{0}^{n}$ and the external force $f_n$, and satisfying inequalities \eqref{E:H_estimate_NS}-\eqref{E:V_estimate_NS}.
Such a solution exists by  Theorem \ref{T:existence_NS}.

To show that the set of measures $\bigl\{ \lcal ({u_n}) , n \in \nat  \bigr\} $
are tight on the space $(\mathcal{Z}_T , \mathcal{T}_T )$, where $\mathcal{Z}_T$ is defined in \eqref{eqn-Z_T}, we argue as in the proof of Lemma 5.4 in \cite{Brz+Motyl_2013} and  apply Corollary \ref{cor-tigthness_criterion}.
We first  observe  that due to estimates  \eqref{E:H_estimate_NS} (with $q=2$) and \eqref{E:V_estimate_NS}, conditions \eqref{cond-a} and \eqref{cond-b}  of Corollary \ref{cor-tigthness_criterion} are satisfied. Thus, it is sufficient to  prove condition (\textbf{a}), i.e. that the sequence $(\un {)}_{n \in \nat }$ satisfies the Aldous condition \textbf{[A]}. By Lemma \ref{lem-Aldous} it is sufficient to proof the condition (\textbf{a'}).

We have now to choose our steps very carefully as we no longer treat strong solutions to an SDE in a finite dimensional Hilbert space but instead a weak solution to  an SPDE in an infinite dimensional Hilbert space.

Let ${(\taun )}_{n \in \nat} $ be a sequence of stopping times taking values in $[0,T]$.
Since each process satisfies equation \eqref{E:mart_sol_int_identity}, by Remark \ref{R:Sol_cont_V'}   we have
\begin{eqnarray*}
& & \un (t) \,
 =    {u}_{0,n}  - \int_{0}^{t}  \acal  \un (s) \, ds
  - \int_{0}^{t} B \bigl( \un (s) \bigr) \, ds
  + \int_{0}^{t}  {f}_{n}(s) \, ds
  + \int_{0}^{t}  G(\un (s)) \, dW(s)  \nonumber \\
& & =:    \Jn{1} + \Jn{2}(t) + \Jn{3}(t) + \Jn{4}(t) + \Jn{5}(t), \qquad t \in [0,T],
\end{eqnarray*}
where the above equality is understood in the space ${\rV }^{\prime }$.
Let us choose and $\theta  >0 $.  It is sufficient to show that each sequence $\Jn{i}$ of processes, $i=1,\cdots,5$ satisfies the sufficient condition (\textbf{a'}) from Lemma \ref{lem-Aldous}.

Obviously the term $\Jn{1}$ which is  constant in time,  satisfies whatever we want. We will only deal with the other terms.
 In fact, we will check that the terms $\Jn{2}, \Jn{4}, \Jn{5}$ satisfy condition (\textbf{a'})
 from Lemma \ref{lem-Aldous} in the space $Y={\rV }^{\prime }$ and the term $\Jn{3}$ satifies this condition in  $Y={\rV}_{s}^{\prime }$ with $s> \frac{d}{2}+1$. Since the embeddings ${\rV}_{s}^{\prime } \subset {U}^{\prime }$ and ${\rV }^{\prime } \subset {U}^{\prime } $ are continuous, we infer that
(\textbf{a'})
from Lemma \ref{lem-Aldous} holds in the space $Y={U}^{\prime }$, as well.

\bigskip \noindent
\textbf{ Ad } ${\Jn{2}}$. \rm  Since the linear operator $\acal :\rV  \to {\rV }^{\prime }$ is bounded, by the H\"older inequality and \eqref{E:V_estimate_NS}, we have
\begin{eqnarray}
\mathbb{E}_{n}\,\bigl[ {| \Jn{2} (\taun + \theta ) - \Jn{2}(\taun ) | }_{{\rV }^{\prime }}  \bigr]
   &\leq&   \mathbb{E}_{n}\,\Bigl[  \int_{\taun }^{\taun + \theta }
{\bigl|  \acal  \un (s) \bigr| }_{{\rV}^{\prime }}  \, ds \biggr]
 \nonumber \\
&
 \le&  {\theta }^{\frac{1}{2}} \Bigl( \mathbb{E}_{n}\,\Bigl[  \int_{0 }^{T }
 \Norm{  \un (s) }{}{2}  \, ds \Bigr] {\Bigr) }^{\frac{1}{2}}
\le   {C}_{2}(p) \cdot {\theta }^{\frac{1}{2}}.
\label{eqn-Jn2}
\end{eqnarray}

\bigskip \noindent
\textbf{Ad } ${\Jn{3}}$. \rm Let $s > \frac{d}{2} +1 $
Similarly, since $B: \rH \times \rH \to {\rV}_{s }^{\prime }$ is bilinear and continuous (and hence bounded so that   the norm $\| B \| $ of $B: \rH \times \rH \to {\rV}_{s }^{\prime }$ is finite),  then by \eqref{E:H_estimate_NS} we have the following estimates
\begin{eqnarray}
& &\mathbb{E}_{n}\,\bigl[ {| \Jn{3} (\taun + \theta ) - \Jn{3}(\taun) | }_{{\rV}_{s}^{\prime }}  \bigr]
 = \mathbb{E}_{n}\,\Bigl[  \bigl| \int_{\taun }^{\taun + \theta }
  B \bigl(\un (r) \bigr) \, dr {\bigr| }_{{\rV}_{s}^{\prime }} \Bigr]
 \le c\mathbb{E}_{n}\,\Bigl[  \int_{\taun }^{\taun + \theta }
 {| B\bigl( \un (r)  \bigr) | }_{{\rV}_{s }^{\prime }} \, dr \Bigr] \nonumber \\
& & \le c \| B \| \, \mathbb{E}_{n}\,\biggl[  \int_{\taun }^{\taun + \theta }
    {| \un (r)|}_{\rH}^{2}   \, dr \biggr]
 \le c\| B \|  \cdot  \mathbb{E}_{n}\,\bigl[ \sup_{r \in [0,T]} {|\un (r)|}_{\rH}^{2}\bigr] \cdot \theta
 \le c\| B \| \,  {C}_{1}(p,2)    \cdot\theta.   \label{E:Jn3}
\end{eqnarray}
\textbf{Remark.} The above argument works as well for $d=3$. However for $d=2$ we have the following different proof which exploits  inequality \eqref{E:estimate_B} (which is valid only the the two dimensional case).
\begin{eqnarray}
& &\mathbb{E}_{n}\,\bigl[ {| \Jn{3} (\taun + \theta ) - \Jn{3}(\taun) |}_{{\rV }^{\prime }}  \bigr]
 \le \mathbb{E}_{n}\,\Bigl[  \int_{\taun }^{\taun + \theta }
{ \bigl|  B\bigl( \un (r)  \bigr)  \bigr| }_{{\rV }^{\prime }} \, dr \Bigr]
\leq c_2 \mathbb{E}_{n}\, \int_{\taun }^{\taun + \theta }
{   \norm{ u_n(r)}{{L}^{2}}{} \,\norm{ \nabla u_n(r)}{{L}^{2}}{} } \, dr
\nonumber \\
& & \leq c_2 \biggl[\mathbb{E}_{n} \,   \sup_{r \in [\taun,\taun + \theta]}  \norm{u_n(r)}{H}{2}  \biggr]^{\frac12}
\biggl[ \mathbb{E}_{n} \,\int_{\taun }^{\taun + \theta }
       \norm{ \nabla u_n(r)}{{L}^{2}}{2} \, dr \biggr]^{\frac12}  {\theta }^{\frac{1}{2}}
     \nonumber  \\
       & & \leq c_2  \biggl[\mathbb{E}_{n} \,   \sup_{r \in [0,T]}  \norm{u_n(r)}{\rH }{2}  \biggr]^{\frac12}
\biggl[ \mathbb{E}_{n} \,\int_{0}^{T}
       \norm{ \nabla u_n(r)}{{L}^{2}}{2} \, dr \biggr]^{\frac12} {\theta }^{\frac{1}{2}}
    \nonumber   \\
&& \le c_2   [{C}_{1}(p,2)]^{\frac{1}{2}} [{C}_{2}(p)]^{\frac{1}{2}}\theta^{\frac12} .
\label{eqn-Jn3}
\end{eqnarray}

\bigskip \noindent
\textbf{Ad } ${\Jn{4}}$. \rm Since the sequence $({f}_{n})$ is weakly convergent in ${L}^{p}(0,T;{\rV }^{\prime })$, it is, in particular, bounded in ${L}^{p}(0,T;{\rV }^{\prime })$. Using the H\"{o}lder inequality,  we have
\begin{eqnarray}
& &\mathbb{E}_{n} \,\bigl[ {| \Jn{4} (\taun + \theta ) - \Jn{4}(\taun) | }_{{\rV }^{\prime }}  \bigr]
 = \mathbb{E}_{n} \,\Bigl[ \bigl| \int_{\taun }^{\taun + \theta }  {f}_{n} (s) \, ds {\bigr| }_{{\rV }^{\prime }} \Bigr]
  \nonumber \\
& & \le    {\theta }^{\frac{p-1}{p}}   \Bigl( \mathbb{E}_{n} \,\Bigl[  \int_{0 }^{T }
{|{f}_{n}(s)|}_{{\rV }^{\prime }}^{p}  \, ds \Bigr] {\Bigr) }^{\frac{1}{p}}
=   {\theta }^{\frac{p-1}{p}}  \norm{{f}_{n}}{{L}^{p}(0,T;{\rV }^{\prime })}{} = {c}_{4} \cdot {\theta }^{\frac{p-1}{p}}, \label{E:Jn4}
\end{eqnarray}
where ${c}_{4}:= \sup_{n \in \nat } \norm{{f}_{n}}{{L}^{p}(0,T;{\rV }^{\prime })}{}$.

\bigskip \noindent
\textbf{Ad } ${\Jn{5}}$. \rm By assumption
\eqref{E:G*} and inequality \eqref{E:H_estimate_NS}, we obtain the following inequalities
\begin{eqnarray}\label{eqn-Jn5}
\mathbb{E}_{n} \,\bigl[ {| \Jn{5} (\taun + \theta ) - \Jn{5}(\taun ) | }_{{\rV }^{\prime }}^{}  \bigr] &\leq &
\Bigl\{ \mathbb{E}_{n} \,\bigl[ {| \Jn{5} (\taun + \theta ) - \Jn{5}(\taun ) | }_{{\rV }^{\prime }}^{2}  \bigr] \Bigr\}^{\frac12}
 \nonumber \\
&=& \Bigl[ \mathbb{E}_{n} \,\int_{\taun }^{\taun + \theta }
 \Norm{ G( \un (s) )}{\lhs (Y ,{\rV }^{\prime })}{2} \, ds  \Bigr]^{\frac12}
    \nonumber \\
& \le &   \Bigl[  C\cdot \mathbb{E}_{n} \,\int_{\taun }^{\taun + \theta }
  ( 1 + | \un (s) {|}_{\rH}^{2} )\, ds  \Bigr]^{\frac12}  \nonumber \\
& \le& \Bigl[ C    \bigl( 1 +
  \Bigl[ \mathbb{E}_{n} \,\bigl[ \sup_{s \in [0,T]} {| \un (s) | }_{\rH}^{2}\bigr] \bigr) \theta \Bigr]^{\frac12} \nonumber \\
&\le & \Bigl[ C (1+ {C}_{1}(2) ) \theta\Bigr]^{\frac12}   =: {c}_{5} \cdot \theta^{\frac{1}{2}} .
\end{eqnarray}
Thus the proof of Theorem \ref{thm-continuous dependence-tightness} is complete.
\end{proof}

\begin{remark}\label{rem-thm-continuous dependence-tightness}
It is easy to be convinced that $u_n$ take values in ${\mathcal{Z}}_{T}$ but it's not so obvious to see that in fact $u_n$ are Borel measurable  functions.
This is so because our construction of the martingale solution is based on Jakubowski's version of the Skorokhod Theorem, see Theorem \ref{T:2_Jakubowski} for details.
\end{remark}

The main result about the continuous dependence of the solutions of the Navier-Stokes equations on the initial state and deterministic external forces, which covers both  cases  of 2D and   3D domains,   is expressed in the following theorem \ref{thm-continuous dependence-existence}. Stronger version for 2D domains will be formulated in the next section, see Theorem \ref{thm-weak continuous dependence-existence_2D}.

\begin{theorem} \label{thm-continuous dependence-existence}
Let conditions \textbf{(H.1)}-\textbf{(H.3)} and \textbf{(H.5)} of Assumption \ref{assumption-main}  be satisfied and  let $T>0$.
Assume that  $\, \, \big({u}_{0,n}\big)_{n=1}^\infty$ is an $\rH$-valued  sequence that is convergent  weakly in $\rH$  to  $\, \, {u}_{0} \in \rH $ and
$\, {({f}_{n})}_{n=1}^{\infty } \, $ is an ${L}^{p}(0,T;{\rV }^{\prime })$-valued sequence that is weakly convergent in ${L}^{p}(0,T;{\rV }^{\prime })$ to $f \in {L}^{p}(0,T;{\rV }^{\prime })$.
Let ${R}_{1}>0$ and ${R}_{2}>0$ be such that
$\sup_{n \in \nat } \norm{{u}_{0,n}}{\rH}{} \le {R}_{1}$ and $\sup_{n \in \nat } \Norm{{f}_{n}}{{L}^{p}(0,T;{\rV }^{\prime })}{} \le {R}_{2}$.
Let
\[
\bigl( \hat{\Omega }_n, \hat{\fcal}_n, \hat{\fmath }_n,\hat{\p }_n  {\hat{W}}_{n}, u_n \bigr)
\]
be a martingale solution
of problem \eqref{E:NS} with the initial data ${u}_{0}^{n}$ and the external force $f_n$  and satisfying inequalities \eqref{E:H_estimate_NS}-\eqref{E:V_estimate_NS}.

Then there exist
\begin{itemize}
\item a subsequence $({n}_{k}{)}_{k}$,
\item a stochastic basis
$\bigl( \tOmega , \tfcal ,  {\tilde{\mathbb{F}}} , \tp  \bigr) $, where $\tilde{\mathbb{F}}={\{ \tfcal {}_{t} \} }_{t \ge 0}$,
\item a cylindrical Wiener process $\tW=\tW (t)$, $t\in [0,\infty)$ defined on this basis,
\item and  progressively measurable processes $\tu$,  $\big(\tunk\big)_{k \ge 1} $ (defined on this basis) with  laws supported in $ \mathcal{Z}_{T} $ such that
\begin{equation}   \label{E:Skorokhod_appl}
  \tunk \mbox{ \it has the same law as } \unk \mbox{ \it on } \mathcal{Z}_{T}
    \mbox{ \it and } \tunk \to \tu \mbox{ \it in } \mathcal{Z}_{T},
    \quad  \tp \mbox{ - \it a.s.}
\end{equation}
\end{itemize}
for every $q\in [1,p]$
\begin{equation} \label{E:tu_H_estimate}
\tilde{\mathbb{E}} \,\bigl[ \,\sup_{ s\in [0,T]  } {|\tu (s)| }_{\rH}^{q}\,\bigr] < \infty ,
\end{equation}
and  the  system
\[\bigl( \tOmega , \tfcal ,  {\tilde{\mathbb{F}}}, \tp , \tW,\tu  \bigr)\]
is   a solution    to  problem  \eqref{E:NS}. \\
In particular,
for all $t\in [0,T]$ and all $\rv \in \rV $
\begin{eqnarray*}
& &\ilsk{\tu (t)}{\rv}{\rH} \,
 - \ilsk{\tu (0)}{\rv}{\rH}  +  \int_{0}^{t} \dual{ \acal \tu (s) }{\rv}{} \, ds
  +   \int_{0}^{t} \dual{ B \bigl( \tu (s)\bigr)}{\rv}{} \, ds   \\
& &\qquad \qquad  =  \int_{0}^{t} \dual{ f(s)}{\rv}{} \, ds
+ \Dual{ \int_{0}^{t} G \bigl( \tu (s) \bigr) \, d \tW (s)}{\rv}{}
\end{eqnarray*}
and
\begin{equation} \label{E:tu_V_estimate}
   \tilde{\mathbb{E}} \,\Bigl[ \,\int_{0}^{T} \Norm{\tu (s)}{}{2}\, ds \,\Bigr] < \infty .
\end{equation}
\end{theorem}

\begin{proof} Since the product topological space $\zcal_{T} \times \mathcal{C} ([0,T],\rK )$ satisfies
the assumptions of Theorem \ref{T:2_Jakubowski}, by applying  it together with Theorem \ref{thm-continuous dependence-tightness},
 there exists a subsequence $({n}_{k})$, a probability space $(\tOmega , \tfcal ,\tp )$ and $\zcal_{T}\times \mathcal{C} ([0,T],\rK )$-valued Borel random variables $\big( \tilde{u },\tilde{W }\big), \big({\tilde{u }}_{k},{\tilde{W }}_{k}\big) $, $k \in \nat $ such that each  $\tilde{W }$ and $ {\tilde{W }}_{k} $, $k \in \nat $ is  an $\rK$-valued Wiener process and
 such that
 \begin{equation}
 \label{eqn-equal laws}
\mbox{ the laws on } \mathcal{B}(\zcal_{T} \times \mathcal{C} ([0,T],\rK )) \mbox{  of } (u_{n_{k}},W)
\mbox{ and }({\tilde{u}}_{k},{\tilde{W}}_{k}) \mbox{ are equal.}
\end{equation}
where $\mathcal{B}(\zcal_{T} \times \mathcal{C} ([0,T],\rK ))$ is the Borel $\sigma$-algebra on $\zcal_{T} \times \mathcal{C} ([0,T],\rK )$,
and, with $\hat{\rK}$ being an auxiliary Hilbert space such that $\rK \subset \hat{\rK}$ and the natural embedding $\rK \embed \hat{\rK}$ is Hilbert-Schmidt,
\begin{equation}
 \label{eqn-conv-as}
\mbox{$\big({\tilde{u }}_{k},{\tilde{W }}_{k}\big)$ converges to $\big(\tilde{u },\tilde{W }\big)$ in $\zcal_{T}\times \mathcal{C} ([0,T],\hat{\rK})$ \;\;\;$\tp$-almost surely on $\tOmega$.}
\end{equation}
Note that since $\mathcal{B}(\zcal_{T} \times \mathcal{C} ([0,T],\rK )) \subset \mathcal{B}(\zcal_{T}) \times \mathcal{B}(\mathcal{C} ([0,T],\rK ))$,
the function $u$ is $\zcal_{T}$ Borel random variable.

Define a corresponding  sequence of  filtrations by
\begin{equation}\label{eqn-new filtration}
{\tilde{\mathbb{F}}}_{k} =({\tilde{\fcal }}_{k}(t))_{t\geq 0}, \mbox{ where }
{\tilde{\fcal }}_{k}(t) = \sigma \big(\{ \big({\tilde{u }}_{k} (s),{\tilde{W }}_{k}(s)\big), \, \, s \le t \}\big),\;\;  t \in [0,T].
\end{equation}

\noindent
To conclude the proof,  we need to show that the random variable $\tilde{u}$ gives rise to a martingale solution. The proof of this claim is very similar  to the proof of Theorem 2.3 in \cite{EM_2014}.
Let us denote the subsequence $(\tunk{)}_{k}$ again by $(\tun {)}_{n}$.

The few differences are:
\begin{trivlist}
\item[(i)] The finite dimensional space $H_n$ is replaced by the whole space $\rH$. But now, by Lemma \ref{lem-Kuratowski}  the space
$\ccal ([0,T];{\rV }^{\prime }) \cap \mathcal{Z}_T$ is a Borel subset of $\mathcal{Z}_T$
and since by Remark \ref{R:Sol_cont_V'} $\, \un \in \ccal ([0,T];{\rV }^{\prime })$, $\p $-a.s. and $\tun $ and $\un $ have the same laws on $\mathcal{Z}_{T}$,
    we infer that
\begin{equation*}
  \tun \in  \ccal ([0,T]; {\rV }^{\prime })  \qquad n \ge 1, \;\;\;  \tp\mbox{-a.s.}
\end{equation*}
    \item[(ii)] The  operator $P_n$ has to be replaced by the identity. But this is rather a simplification as  for instance we do not need Lemmas 2.3 and 2.4 from \cite{Brz+Motyl_2013}.
\end{trivlist}

\bigskip  \noindent
In addition to point (i) above, we have that  for  every $q\in [1,p]$, we have
 \begin{equation} \label{E:H_estimate'}
 \sup_{n\in \nat }\; \te\,\; \bigl( \sup_{0\le s\le T } {|\tun (s)|}_{\rH}^{q}\bigr) \le
 {C}_{1}(p,q),
\end{equation}
Similarly,
\[
  \tun \in  L^2(0,T;\rV )  \qquad n \ge 1, \;\;\;  \p\mbox{-a.s.}
\]
and
\begin{equation} \label{E:V_estimate'}
  \sup_{n\in \nat } \tilde{\mathbb{E}}\, \Bigl[ \int_{0}^{T} {\| \tun (s)\| }_{\rV }^{2}\, ds \Bigr]
  \le {C}_{2}(p).
\end{equation}
By inequality \eqref{E:V_estimate'} we infer that the sequence $(\tun )$ contains a subsequence, still denoted by $(\tun )$, convergent weakly  in the space ${L}^{2}([0,T]\times \tOmega ; \rV  )$.
Since by \eqref{eqn-conv-as}  $\tp $-a.s. $\tun \to \tu $ in $\zcal_{T}$, we conclude that
$\tu \in {L}^{2}([0,T]\times \tOmega ; \rV  )$, i.e.
\begin{equation} \label{E:tu_V_estimate'}
   \tilde{\mathbb{E}}\, \Bigl[ \int_{0}^{T} \norm{\tu (s)}{}{2}\, ds \Bigr] < \infty .
\end{equation}
Similarly, by inequality \eqref{E:H_estimate'} with $q=p$ we can choose a subsequence of $(\tun )$ convergent weak star in the space ${L}^{p}(\tOmega ; {L}^{\infty }(0,T;\rH ))$ and, using \eqref{eqn-conv-as}, infer that
\begin{equation} \label{E:tu_H_estimate'}
\tilde{\mathbb{E}}\, \bigl[ \sup_{0\le s\le T } {|\tu (s)|}_{\rH}^{p}\bigr] < \infty .
\end{equation}
Then, of course, for every  $q\in [1,p]$,
\begin{equation} \label{E:tu_H_estimate_q'}
\tilde{\mathbb{E}}\, \bigl[ \sup_{0\le s\le T } {|\tu (s)|}_{\rH}^{q}\bigr] < \infty .
\end{equation}

The remaining proof will be done in two steps.

\noindent
\bf Step 1. \rm
Let us choose and  fix $s>\frac{d}{2}+1$.  We will first prove the following Lemma.

\begin{lemma}  \label{L:convergence_existence}
For all $\varphi \in {\rV}_{s}$
\begin{itemize}
\item[(a)] $\lim_{n\to \infty } \tilde{\e } \bigl[ \int_{0}^{T} {|\ilsk{{\tilde{u}}_{n}(t)-\tilde{u}(t)}{\varphi}{\rH }|}^{2} \, dt \bigr] =0 $,
\item[(b)] $\lim_{n\to \infty } \tilde{\e } \bigl[ {|\ilsk{{\tilde{u}}_{n}(0)-\tilde{u}(0)}{\varphi}{\rH}|}^{2}  \bigr] =0 $,
\item[(c)] $\lim_{n\to \infty } \tilde{\e } \bigl[ \int_{0}^{T}
 \bigl| \int_{0}^{t}\dual{ \acal {\tilde{u}}_{n}(s)-\acal \tilde{u}(s)}{\varphi}{} \,ds \bigr| \, dt \bigr] =0 $,
\item[(d)] $\lim_{n\to \infty } \tilde{\e } \bigl[ \int_{0}^{T}
 \bigl| \int_{0}^{t}\dual{ B ({\tilde{u}}_{n}(s))- B(\tilde{u}(s))}{\varphi}{} \,ds \bigr| \, dt \bigr] =0 $,
\item[(e)] $\lim_{n\to \infty } \tilde{\e } \bigl[ \int_{0}^{T}
 \bigl| \int_{0}^{t}\dual{ {f}_{n}(s)-f(s)}{\varphi}{} \,ds \bigr| \, dt \bigr] =0 $,
\item[(f)] $\lim_{n\to \infty } \tilde{\e } \bigl[ \int_{0}^{T}
 \bigl| \dual{ \int_{0}^{t} [ G({\tilde{u}}_{n}(s)) - G(\tilde{u}(s))] \, d \tilde{W}(s) }{\varphi}{}
 {\bigr| }^{2} \, dt \bigr] =0 $.
\end{itemize}
\end{lemma}

\begin{proof}[Proof of Lemma \ref{L:convergence_existence}] Let us fix $\varphi \in {\rV}_{s} $.
\bf Ad  (a). \rm
Since by \eqref{eqn-conv-as}
$ {\tilde{u}}_{n} \to \tilde{u} $ in $\mathcal{C} ([0,T];{\rH}_{w})$ $\tilde{\p }$-a.s.,
$\ilsk{{\tilde{u}}_{n}(\cdot )}{\varphi}{\rH} \to \ilsk{\tilde{u}(\cdot )}{\varphi}{\rH} $
in $\mathcal{C} ([0,T];\rzecz )$, $\tilde{\p }$-a.s. Hence, in particular, for  all $t \in [0,T]$
\[
  \lim_{n \to \infty } \ilsk{{\tilde{u}}_{n}(t)}{\varphi}{\rH} = \ilsk{\tilde{u}(t)}{\varphi}{\rH},
  \qquad  \mbox{ $\tilde{\p }$-a.s. }
\]
Since by \eqref{E:H_estimate'},
 $\sup_{t\in [0,T]} {|{\tilde{u}}_{n}(t)|}_{\rH}^{2} < \infty $, $\tilde{\p }$-a.s.,
using the dominated convergence theorem we infer that
\begin{equation} \label{E:Vitali_{u}_{n}_conv}
  \lim_{n \to \infty } \int_{0}^{T} {| \ilsk{{\tilde{u}}_{n}(t) -\tilde{u}(t)}{\varphi}{\rH} |}^{2} \, dt =0  \qquad  \mbox{ $\tilde{\p }$-a.s. }.
\end{equation}
By the H\"{o}lder inequality and \eqref{E:H_estimate'}
for every $n \in \nat$ and every $r\in \bigl( 1, 1+ \frac{p}{2}\bigr] $
\begin{eqnarray}
&&  \tilde{\e }\Bigl[ \Bigl|  \int_{0}^{T} {|{\tilde{u}}_{n}(t) -\tilde{u}(t) | }_{\rH}^{2} \, dt
 {\Bigr| }^{r}\Bigr]
  \le c \tilde{\e }\Bigl[  \int_{0}^{T} \bigl( {| {\tilde{u}}_{n}(t)  | }_{\rH}^{2r}
+ {| \tilde{u}(t)| }_{\rH}^{2r} \bigr) \, dt \Bigr]   \le \tilde{c} {C}_{1}(p,2r) ,
 \label{E:Vitali_{u}_{n}_est}
\end{eqnarray}
where $c, \tilde{c}$ are some positive constants.
To conclude the proof of
 assertion (a) it is sufficient to use \eqref{E:Vitali_{u}_{n}_conv}, \eqref{E:Vitali_{u}_{n}_est} and the Vitali Theorem.

 \bigskip  \noindent
\bf Ad  (b). \rm
Since by \eqref{eqn-conv-as} ${\tilde{u}}_{n} \to \tilde{u}$ in $\mathcal{C} (0,T;{\rH }_{w})$ $\tilde{\p }$-a.s. and $\tilde{u}$ is continuous at $t=0$, we infer that
$ \ilsk{{\tilde{u}}_{n}(0)}{\varphi}{\rH } \to \ilsk{\tilde{u}(0)}{\varphi}{\rH } $, $\tilde{\p }$-a.s.
Now, assertion (b) follows from \eqref{E:H_estimate'} and the Vitali Theorem.

\bigskip  \noindent
\bf Ad  (c). \rm
Since by \eqref{eqn-conv-as} ${\tilde{u}}_{n} \to \tilde{u}$ in ${L}_{w}^{2}(0,T;\rV )$,
$\tilde{\p }$-a.s.,
 by   \eqref{E:Acal_ilsk_Dir} we infer that $\tilde{\p }$-a.s.
\begin{equation}
 \lim_{n\to \infty } \int_{0}^{t} \dual{ \acal {\tilde{u}}_{n}(s)}{\varphi}{} \, ds
 = \lim_{n \to \infty } \int_{0}^{t} \dirilsk{{\tilde{u}}_{n}(s)}{ \varphi}{} \, ds 
 = \int_{0}^{t} \dirilsk{ \tilde{u}(s)}{\varphi}{} \, ds
 = \int_{0}^{t} \dual{ \acal \tilde{u}(s)}{\varphi}{} \, ds \qquad \mbox{}
 \label{E:Vitali_{A}_{n}_conv}
\end{equation}
By   \eqref{E:Acal_ilsk_Dir}, the H\"{o}lder inequality and estimate \eqref{E:V_estimate'} we infer that
 for all $t \in [0,T]$ and  $n \in \nat $
\begin{eqnarray}
&&\hspace{-2truecm}\lefteqn{  \tilde{\e }\Bigl[ \Bigl | \int_{0}^{t}  \dual{  \acal {\tilde{u}}_{n}(s)}{\varphi}{} \, ds {\Bigr| }^{2} \Bigr]
  = \tilde{\e } \Bigl[ \Bigl | \int_{0}^{t} \dirilsk{{\tilde{u}}_{n}(s)}{ \varphi}{} \, ds {\Bigr| }^{2} \Bigr]}
   \nonumber \\
& \le & c \, \| { \varphi}{\| }_{{\rV}_{s} }^{2} \,  \tilde{\e }
 \Bigl[  \int_{0}^{T} \Norm{ {\tilde{u}}_{n}(s)}{\rV  }{2} \, ds  \Bigr]
  \le \tilde{c} {C}_{2}(p) ,
  \label{E:Vitali_{A}_{n}_est}
\end{eqnarray}
where  $c, \tilde{c}>0$ are some constants.
By \eqref{E:Vitali_{A}_{n}_conv}, \eqref{E:Vitali_{A}_{n}_est} and the Vitali Theorem we conclude that for all $t \in [0,T]$
\[
  \lim_{n\to \infty } \tilde{\e } \Bigl[ \bigl| \int_{0}^{t}
\dual{\acal {\tilde{u}}_{n}(s)-\acal \tilde{u}(s)}{\varphi}{} \, ds
 {\bigr| }^{} \Bigr] =0.
\]
Assertion (c) follows now from \eqref{E:V_estimate'} and the dominated convergence theorem.

\bigskip  \noindent
\bf Ad  (d). \rm
Since by \eqref{E:V_estimate'} and \eqref{E:norm_V} the sequence $({\tilde{u}}_{n})$ is bounded in ${L}^{2}(0,T;\rH )$ and
by \eqref{eqn-conv-as} ${\tilde{u}}_{n} \to \tilde{u}$ in
${L}^{2}(0,T;{\rH}_{\mbox{\rm {loc}}}) $, $\tilde{\p }$-a.s.,
by Lemma B.1 in \cite{Brz+Motyl_2013} we infer that $\tilde{\p }$-a.s.
for all $t \in [0,T]$ and  $\varphi \in {\rV}_{s}$
\begin{equation}
 \lim_{n\to \infty } \int_{0}^{t} \dual{ B ({\tilde{u}}_{n}(s)) - B ({\tilde{u}}(s))}{\varphi}{} \, ds =0 .  \label{E:Vitali_{B}_{n}_conv}
\end{equation}
Using the H\"{o}lder inequality,  Lemma \ref{lem-B-H}  and  \eqref{E:H_estimate'} we infer that for all
 $t \in [0,T]$,  $r\in \bigl(0,\frac{p }{2}\bigr] $ and  $n \in \nat $ the following inequalities hold
\begin{eqnarray}
&&  \tilde{\e } \Bigl[ \Bigl| \int_{0}^{t} \dual{B ({\tilde{u}}_{n}(s)) }{\varphi}{} \, ds {\Bigr| }^{1+r}  \Bigr]
\le \tilde{\e }  \Bigl[ \Bigl( \int_{0}^{t} {|B ({\tilde{u}}_{n}(s))| }_{{\rV}_{s}^{\prime }}  \norm{\varphi}{{\rV}_{s}}{} \, ds {\Bigr) }^{1+r}  \Bigr] \nonumber \\
&&  \le ({c}_{2}\norm{\varphi}{{\rV}_{s}}{} {)}^{1+r} \, {t}^{r} \,
 \e \Bigl[ \int_{0}^{t} {| {\tilde{u}}_{n}(s) | }_{\rH }^{2+2r} \, ds \Bigr]
 \le \tilde{C} \tilde{\e } \bigl[ \sup_{s\in [0,T]} {|{\tilde{u}}_{n}(s) | }_{\rH }^{2+2r}   \bigr]
  \le \tilde{C} {C}_{1}(p,2+2r).  \label{E:Vitali_{B}_{n}_est}
\end{eqnarray}
By \eqref{E:Vitali_{B}_{n}_conv}, \eqref{E:Vitali_{B}_{n}_est} and  the Vitali Theorem we
obtain for all $t\in [0,T]$
\begin{equation} \label{E:Lebesque_{B}_{n}_conv}
  \lim_{n \to \infty } \tilde{\e } \Bigl[ \bigl| \int_{0}^{t}
 \dual{ B ({\tilde{u}}_{n}(s)) - B(\tilde{u}(s))}{\varphi}{} \, ds {\bigr| }^{}  \Bigr] =0 .
\end{equation}
Using again Lemma \ref{lem-B-H}  and estimate \eqref{E:H_estimate'}, we obtain  for all $t \in [0,T]$ and  $n \in \nat $
\[
   \tilde{\e } \Bigl[ \bigl| \int_{0}^{t} \dual{ B ({\tilde{u}}_{n}(s)) }{\varphi}{} \, ds {\bigr| }^{} \Bigr] \le c \tilde{\e } \bigl[ \sup_{s\in [0,T]} {|{\tilde{u}}_{n}(s))| }_{\rH }^{2}   \bigr]
 \le c {C}_{1}(p,2) ,
\]
where $c>0$ is a  constant. Hence
by \eqref{E:Lebesque_{B}_{n}_conv} and the dominated convergence theorem, we infer that
assertion (d) holds.

\bigskip  \noindent
\bf Ad  (e). \rm
Assertion (e) follows because the sequence $({f}_{n})$ converges weakly in ${L}^{p}(0,T;{\rV }^{\prime })$ to $f$ and ${\rV }_{s} \subset \rV $.

\bigskip  \noindent
\bf Ad  (f). \rm
Let us notice  that for all $\varphi \in \rV  $ we have
\begin{align*}
&  \int_{0}^{t}  {\| \dual{G( {\tilde{u}}_{n}(s))- G(\tilde{u}(s))}{\varphi}{}
  \| }^{2}_{\lhs (\hat{\rK };\rzecz )} \, ds  \\
 & = \int_{0}^{t}  {\|  {\varphi}^{\ast \ast }G({\tilde{u}}_{n})(s) - {\varphi}^{\ast \ast }G(\tilde{u})(s)
    \| }^{2}_{\lhs (\hat{\rK };\rzecz )} \, ds
 \le \Norm{{\varphi}^{\ast \ast }G({\tilde{u}}_{n}) - {\varphi}^{\ast \ast }G(\tilde{u})}{{L}^{2}([0,T];\lhs (\hat{\rK };\rzecz )) }{2} ,
 \end{align*}
where   ${\varphi}^{\ast \ast }G$ is the map defined by \eqref{E:G**} in assumption \textbf{(H.2)}.
Since by \eqref{eqn-conv-as}
${\tilde{u}}_{n} \to \tilde{u}$ in ${L}^{2}(0,T;{\rH}_{\mathrm{loc}}) $, $\tilde{\p }$-a.s.,
by \eqref{E:G**} we infer that for all $t\in [0,T]$ and  $\varphi \in \rV  $
\begin{align}
&  \lim_{n \to \infty } \int_{0}^{t} {\| \dual{G({\tilde{u}}_{n}(s))- G(\tilde{u}(s))}{\varphi}{}
  \| }^{2}_{\lhs (\hat{\rK };\rzecz )} \, ds =0  .
   \label{E:Vitali_{Gaussian}_{n}_conv}
\end{align}
By \eqref{E:G*} and \eqref{E:H_estimate'} we obtain the following inequalities
for every $t \in [0,T]$, $r \in \bigl( 1,1+ \frac{p }{2} \bigr] $ and  $n \in \nat $
\begin{align}
& \tilde{\e } \Bigl[  \bigl| \int_{0}^{t} {\| \dual{G( {\tilde{u}}_{n}(s))- G(\tilde{u}(s))}{\varphi}{}
  \| }^{2}_{\lhs (\hat{\rK };\rzecz )} \, ds {\bigr| }^{r} \Bigr]  \nonumber \\
&  \le  c \, \tilde{\e } \Bigl[ \norm{\varphi}{\rV  }{2r} \cdot \int_{0}^{t} \bigl\{
  \norm{G( {\tilde{u}}_{n}(s))}{\lhs (\hat{\rK };{\rV  }^{\prime })}{2r}
  + \norm{G( \tilde{u}(s))}{\lhs (\hat{\rK };{\rV  }^{\prime })}{2r} \bigr\} \, ds   \Bigr] \nonumber \\
&  \le  {c}_{1} \, \tilde{\e } \Bigl[ \int_{0}^{T} (1+ {|{\tilde{u}}_{n}(s)|}_{\rH }^{2r}
    + {|\tilde{u}(s)|}_{\rH }^{2r}) \, ds  \Bigr] \nonumber  \\
&   \le  \tilde{c} \Bigl\{ 1+ \tilde{\e } \Bigl[ \sup_{s\in [0,T]} {|{\tilde{u}}_{n}(s)|}_{\rH }^{2r}
    + \sup_{s\in [0,T]} {|\tilde{u}(s)|}_{\rH }^{2r})  \Bigr] \Bigr\} \le \tilde{c} (1+2 {C}_{1}(p,2r))
 \label{E:Vitali_{Gaussian}_{n}_est}
\end{align}
where $c,{c}_{1},\tilde{c}$ are some positive constants. Using the Vitali
theorem, by \eqref{E:Vitali_{Gaussian}_{n}_conv}, \eqref{E:Vitali_{Gaussian}_{n}_est} we infer that
for all  $\varphi \in \rV $
\begin{equation}
\lim_{n\to \infty } \tilde{\e } \Bigl[  \int_{0}^{t}
 {\| \dual{G( {\tilde{u}}_{n}(s))- G(\tilde{u}(s))}{\varphi}{} \| }^{2}_{\lhs (\hat{ \rK };\rzecz )} \, ds
  \Bigr] =0  .  \label{E:Lebesgue_{Gaussian}_{n}_conv_V}
\end{equation}
Hence by the properties of the It\^{o} integral we infer that for all $t \in [0,T]$ and  $\varphi \in {\rV }_{} $
\begin{equation}  \label{E:Lebesque_{Gaussian}_{n}_conv}
\lim_{n\to \infty } \tilde{\e } \Bigl[ \bigl|
 \Dual{\int_{0}^{t}\bigl[  G( {\tilde{u}}_{n}(s))- G(\tilde{u}(s)) \bigr] \, d\tilde{W}(s) }{\varphi}{}
 {\bigr| }^{2} \Bigr] =0 .
\end{equation}
By the It\^{o} isometry, since the map $G$ satisfies inequality  \eqref{E:G*} in part \textbf{(H.2)} of  Assumption \ref{assumption-main}, and estimate
\eqref{E:H_estimate'} we have for all $\varphi \in {\rV}_{} $, $t\in [0,T]$ and  $n \in \nat $
\begin{align}
&  \tilde{\e } \Bigl[ \bigl|
 \Dual{\int_{0}^{t}\bigl[ G( {\tilde{u}}_{n}(s))- G(\tilde{u}(s)) \bigr] \, d\tilde{W}(s) }{\varphi}{}
 {\bigr| }^{2} \Bigr] \nonumber \\
&   = \tilde{\e } \Bigl[\int_{0}^{t} {\| \dual{ G( {\tilde{u}}_{n}(s))- G(\tilde{u}(s))}{\varphi}{}
  \| }^{2}_{\lhs (\hat{\rK };\rzecz )} \, ds \Bigr] \nonumber \\
&  \le  c \Bigl\{ 1+ \tilde{\e } \Bigl[ \sup_{s\in [0,T]}  {|{\tilde{u}}_{n}(s)|}_{\rH }^{2}
    + \sup_{s\in [0,T]} {|\tilde{u}(s)|}_{\rH }^{2})  \Bigr] \Bigr\} \le c (1+2 {C}_{1}(p,2)) ,
 \label{E:Lebesque_{Gaussian}_{n}_est}
\end{align}
where $c>0$ is some constant. Thus
by \eqref{E:Lebesque_{Gaussian}_{n}_conv}, \eqref{E:Lebesque_{Gaussian}_{n}_est} and the Lebesgue Dominated Convergence Theorem we infer that for all $\varphi \in {\rV }_{} $
\begin{eqnarray}
 \lim_{n\to \infty } \int_{0}^{T}\tilde{\e } \Bigl[ \bigl|
 \Dual{\int_{0}^{t}\bigl[  G( {\tilde{u}}_{n}(s))- G(\tilde{u}(s)) \bigr]
 \, d\tilde{W}(s) }{\varphi}{}
 {\bigr| }^{2} \Bigr] =0.  \label{E:{Gaussian}_{n}_conv}
\end{eqnarray}
To conclude the proof of assertion (f), it is sufficient to notice that
since $s>\frac{d}{2}+1$, ${\rV }_{s} \subset \rV $ and thus \eqref{E:{Gaussian}_{n}_conv} holds for  all $\varphi \in {\rV }_{s}$.
The proof of Lemma \ref{L:convergence_existence} is thus complete.
\end{proof}

As a direct consequence of  Lemma \ref{L:convergence_existence} we get the following corollary which we precede by introducing some auxiliary notation.
Analogously to \cite{Brzezniak_Hausenblas_2010} and \cite{EM_2014}, let us denote
\begin{eqnarray}
 & & {\Lambda }_{n}({\tilde{u}}_{n},{\tilde{W}}_{n}, \varphi) (t)
:=  \ilsk{{\tilde{u}}_{n}(0)}{\varphi}{\rH }
  - \int_{0}^{t} \dual{\acal {\tilde{u}}_{n}(s)}{\varphi}{}  ds
- \int_{0}^{t} \dual{B({\tilde{u}}_{n}(s))}{\varphi}{}  ds  \nonumber \\
& & + \int_{0}^{t} \dual{ {f}_{n}(s)}{\varphi}{}\, ds
  + \Dual{\int_{0}^{t} G({\tilde{u}}_{n}(s))\,  d{\tilde{W}}_{n}(s)}{\varphi}{} , \quad t \in [0,T] ,
\label{E:Lambda_n_bar_u_n}
\end{eqnarray}
and
\begin{align}
 &  \Lambda  (\tilde{u}, \tilde{W}, \varphi) (t)
:=  \ilsk{\tilde{u}(0)}{\varphi}{\rH}
    - \int_{0}^{t} \dual{ \acal \tilde{u}(s)}{\varphi}{}  ds
 - \int_{0}^{t} \dual{ B(\tilde{u}(s))}{\varphi}{}  ds  \nonumber \\
 & + \int_{0}^{t} \dual{ f(s)}{\varphi}{}\, ds
    +  \Dual{\int_{0}^{t}G(\tilde{u}(s))\,  d\tilde{W}(s)}{\varphi}{}  ,
\quad t \in [0,T] .
\label{E:Lambda_u*}
\end{align}

\begin{cor} \label{C:convergence_existence}
For every $\varphi \in {\rV}_{s}$,
\begin{equation} \label{E:bar_u_n_convergence}
 \lim_{n \to \infty } \norm{\ilsk{{\tilde{u}}_{n}(\cdot )}{\varphi}{\rH }
 - \ilsk{\tilde{u}(\cdot )}{\varphi}{\rH } }{{L}^{2}([0,T]\times \tilde{\Omega })}{} =0
\end{equation}
and
\begin{equation} \label{E:Lambda_n_bar_u_n_convergence}
     \lim_{n \to \infty } \norm{  {\Lambda}_{n}({\tilde{u}}_{n}, {\tilde{W}}_{n},\varphi)
 - \Lambda  (\tilde{u},\tilde{W}, \varphi) }{{L}^{1}([0,T]\times \tilde{\Omega })}{} =0 .
\end{equation}
\end{cor}

\begin{proof}[Proof of Corollary \ref{C:convergence_existence}]
Assertion  \eqref{E:bar_u_n_convergence} follows from the  equality
\[
 \norm{\ilsk{{\tilde{u}}_{n}(\cdot )}{\varphi}{\rH }
 - \ilsk{\tilde{u}(\cdot )}{\varphi}{\rH } }{{L}^{2}([0,T]\times \tilde{\Omega })}{2} \\
=  \tilde{\e } \Bigl[ \int_{0}^{T}
{| \ilsk{{\tilde{u}}_{n}(t) -\tilde{u}(t)}{\varphi}{\rH } |}^{2} \, dt \Bigr]
\]
and Lemma \ref{L:convergence_existence}  (a).
Let us move to the proof of assertion \eqref{E:Lambda_n_bar_u_n_convergence}.
Note that by the Fubini theorem, we have
\begin{align*}
&  \norm{  {\Lambda  }_{n}({\tilde{u}}_{n}, {\tilde{W}}_{n},\varphi)
 - \Lambda  (\tilde{u},\tilde{W}, \varphi) }{{L}^{1}([0,T]\times \tilde{\Omega })}{} \\
& = \int_{0}^{T} \tilde{\e } \bigl[ {| {\Lambda  }_{n}({\tilde{u}}_{n}, {\tilde{W}}_{n},\varphi)(t)
 - \Lambda  (\tilde{u}, \tilde{W} ,\varphi)(t) |}^{}\, \bigr] dt .
 \end{align*}
To conclude the proof of Corollary \ref{C:convergence_existence} it is sufficient to note that
by Lemma \ref{L:convergence_existence} (b)-(f),  each  term on the right hand side of
\eqref{E:Lambda_n_bar_u_n} tends at least in ${L}^{1}([0,T]$ $\times \tilde{\Omega })$ to the corresponding term in \eqref{E:Lambda_u*}.
\end{proof}
\noindent

\bigskip  \noindent
\bf Step 2. \rm Since ${u}_{n}$ is a solution of the Navier-Stokes equation, for all $t\in [0,T]$ and $\varphi \in \vcal $
\[
   \ilsk{{u}_{n}(t)}{\varphi}{\rH } = {\Lambda  }_{n} ({u}_{n},{W}_{},\varphi)(t) , \qquad \p \mbox{-a.s.}
\]
In particular,
\[
  \int_{0}^{T} \e \bigl[ {|  \ilsk{\un (t)}{\varphi}{\rH }
 - {\Lambda  }_{n} ({u}_{n},{W}_{},\varphi)(t) |}^{} \, \bigr] \, dt  =0.
\]
Since $\lcal ({u}_{n},{W}_{}) = \lcal ({\tilde{u}}_{n},{\tilde{W}}_{n})$,
\[
  \int_{0}^{T} \tilde{\e } \bigl[ {| \ilsk{{\tilde{u}}_{n}(t)}{\varphi}{\rH }
- {\Lambda  }_{n} ({\tilde{u}}_{n},{\tilde{W}}_{n},\varphi)(t) |}^{} \, \bigr] \, dt  =0.
\]
Moreover, by \eqref{E:bar_u_n_convergence} and \eqref{E:Lambda_n_bar_u_n_convergence}
\[
\int_{0}^{T} \tilde{\e } \bigl[ {|
\ilsk{\tilde{u}(t)}{\varphi}{\rH } - {\Lambda  }_{} (\tilde{u},\tilde{W},\varphi)(t) |}^{} \, \bigr] \, dt  =0.
\]
Hence for $l$-almost all $t \in [0,T]$ and $\tilde{\p }$-almost all $\omega \in \tilde{\Omega }$
\[
\ilsk{\tilde{u}(t)}{\varphi}{\rH } - {\Lambda  }_{} (\tilde{u},\tilde{W},\varphi )(t)=0,
\]
i.e. for $l$-almost all $t \in [0,T]$ and $\tilde{\p }$-almost all $\omega \in \tilde{\Omega }$
\begin{align}
& \ilsk{\tilde{u}(t)}{\varphi}{\rH }
+\int_{0}^{t} \dual{\acal \tilde{u}(s)}{\varphi}{} \, ds
+ \int_{0}^{t} \dual{B(\tilde{u}(s))}{\varphi}{} \, ds
  \nonumber \\
 & =  \ilsk{\tilde{u}(0)}{\varphi}{\rH }
+\int_{0}^{t} \dual{f(s)}{\varphi}{}\, ds
 + \Dual{ \int_{0}^{t} G(\tilde{u}(s)) \, d \tilde{W}(s) }{\varphi}{} . \label{E:solution}
\end{align}
Since a Borel $\tilde{u}$ is ${\zcal }_{T}$-valued random variable, in particular $\tilde{u}\in \mathcal{C} ([0,T];{\rH }_{w})$, i.e. $\tilde{u}$ is weakly continuous,
we infer that
equality \eqref{E:solution} holds  for all $t\in [0,T]$ and all $\varphi \in \vcal  $.
Since $\vcal  $ is dense in $\rV $,   equality \eqref{E:solution} holds for all $\varphi \in \rV $, as well.
Putting $\tilde{\mathfrak{A}}:=(\tilde{\Omega }, \tilde{\fcal },\tilde{\p }, \tilde{\fmath })$,
 we infer that the system
$(\tilde{\mathfrak{A}}, \tilde{W}, \tilde{u})$ is a martingale solution of  equation
\eqref{E:NS}.
By \eqref{E:tu_V_estimate'} and \eqref{E:tu_H_estimate_q'} the process $\tilde{u}$ satisfies inequalities
\eqref{E:tu_V_estimate} and \eqref{E:tu_H_estimate}.
The proof of Theorem \ref{thm-continuous dependence-existence} is thus complete.
\end{proof}

\begin{remark}\label{rem-V'} It seems to us that the same argument works if the space $\mathcal{Z}_T$ defined in \eqref{eqn-Z_T}
is replaced by a bigger space $\hat{\mathcal{Z}}_T$  defined by
\begin{equation}
\label{eqn-Z_T-bigger}
  \hat{\mathcal{Z}}_T : = {L}_{w}^{2}(0,T;\rV)  \cap {L}^{2}(0,T;{\rH }_{\mathrm{loc}}) \cap \ccal ([0,T];{\rH}_{w}).
\end{equation}
In particular, to prove that the sequence $({\tilde{u}}_{n})$ given in \eqref{eqn-equal laws}, whose existence follows from the Skorokhod Theorem, converges to a solution of the Navier-Stokes equation, it is sufficient to use the convergence of $({\tilde{u}}_{n})$ in the space
$\hat{\mathcal{Z}}_T $.

\end{remark}

\section{The case of 2D domains}  \label{sec-2D-SNSEs}

A special result proved recently in \cite{Brz+Motyl_2013} is about the existence and uniqueness of strong solutions for $2$-D stochastic Navier Stokes equations in unbounded domains with a general noise.

Let us present the framework and the results.
Let us  recall Lemma 7.2  from \cite{Brz+Motyl_2013}.

\begin{lemma} \label{L:2D_NS_regularity}
Let $d=2$ and assume that all conditions in parts  \textbf{(H.1)}-\textbf{(H.3)} and \textbf{(H.5)} of Assumption \ref{assumption-main} are satisfied. Assume that ${\mu }_{0}=\delta_{u_0}$ for some  deterministic  $\, \, {u}_{0} \in \rH $.  Let
$(\hat{\Omega }, \hat{\fcal }, \hat{\fmath },  \hat{W}, \hat{\p }, u)$ be a martingale solution
of problem  \eqref{E:NS}, in particular,
\begin{equation}  \label{E:Th_2D_u_est}
\hat{\mathbb{E}} \Bigl[ \sup_{t \in [0,T]} \norm{u(t)}{\rH}{2} + \int_{0}^{T} \norm{\nabla u(t)}{}{2} \, dt  \Bigr] < \infty .
\end{equation}
Then for $\hat{\p }$-almost all $\omega \in \hat{\Omega }$ the trajectory $u(\cdot , \omega )$ is equal almost everywhere to a continuous $\rH$-valued function defined on $[0,T]$.
$\hat{\p }$-a.s. and
\begin{equation} \label{E:2D_NS}
  u(t) = {u}_{0} - \int_{0}^{t}\bigl[ \acal u(s) + B (u(s)) \bigr] \, ds + \int_{0}^{t}f(s)\, ds
    + \int_{0}^{t}G(u(s)) \, d \hat{W}(s), \quad t \in [0,T].
\end{equation}
\end{lemma}
\noindent
Let us emphasize that equality \eqref{E:2D_NS} is understood as the one in the space ${\rV }^{\prime }$, see Remark \ref{R:Sol_cont_V'}.

\bigskip  \noindent
The next result is \cite[Lemma  7.3]{Brz+Motyl_2013}.

\begin{lemma}  \label{L:2D_pathwise_uniqueness}
Assume that all conditions in parts  \textbf{(H.1)}-\textbf{(H.3)} and \textbf{(H.5)} of Assumption \ref{assumption-main}
are satisfied.  In addition we assume that the Lipschitz constant  of $G$ is smaller than $\sqrt{2}$, i.e. the map $G$ satisfies condition \eqref{E:G_Lipsch} in part \textbf{(H.2)} of  Assumption \ref{assumption-main}
 with  $L<\sqrt{2}$.
 Assume that    $\, \, {u}_{0} \in \rH $.
If ${u}_{1}$ and $ {u}_{2}$ are two  solutions of problem  \eqref{E:NS}
defined on the same filtered probability space
$(\hat{\Omega }, \hat{\fcal }, \hat{\fmath }, \hat{\p })$ and the same Wiener process $\hat{W}$,
then $\hat{\p }$-a.s. for all $t \in \mathbb{R}_+$,
${u}_{1}(t)={u}_{2}(t)$.
\end{lemma}

Because from now  we will be dealing  with the pathwise uniqueness of  solutions let us formulate the following assumption on the  stochastic basis.
\begin{assumption}
\label{assumption-second}
Assume that $\bigl( \Omega , \fcal , \mathbb{F} , \p  \bigr) $ is  a stochastic basis with a filtration $\mathbb{F}={\{ \ft \} }_{t \ge 0}$ and $ W=\bigl(W(t)\bigr)_{t\geq 0}$ is a cylindrical  Wiener process in a separable Hilbert space $\rK$ defined on this stochastic basis.
\end{assumption}

We will often consider problem \eqref{E:NS} with the initial data ${\mu }_{0}=\delta_{u_0}$ for some  deterministic  $\, \, {u}_{0} \in \rH $, and hence we explicitly rewrite that  problem in the following way:
\begin{equation} \label{E:NS'}
\begin{cases}
& du(t)+\acal u(t) \, dt+B\bigl( u(t),u(t) \bigr)\, dt = f(t) \, dt+G\bigl( u(t)\bigr) \, dW(t),
 \qquad t \geq 0 , \\
& u(0) = u_0,
\end{cases}
\end{equation}
To avoid any confusion, a martingale solution to problem \eqref{E:NS'} with initial data  $\, \, {u}_{0} \in \rH $, is a
martingale solution to problem \eqref{E:NS} with ${\mu }_{0}=\delta_{u_0}$ .

For the completeness  of the exposition let us also recall a notion of a strong solution.

\begin{definition}\label{def-strong solution}  \rm Assume that  $\, \, {u}_{0} \in \rH $ and  $f:[0,\infty ) \to {\rV }^{\prime }$. Assume Assumption \ref{assumption-second}.
We say that an $\mathbb{F}$-progressively measurable process
$u: [0,\infty) \times \Omega \to \rH$ with $\p  $ - a.e. paths
\[
  u(\cdot , \omega ) \in \ccal \bigl( [0,\infty), {\rH}_{w} \bigr)
   \cap {L}^{2}_{\textrm{loc}}([0,\infty);\rV )
\]
is a \bf strong solution \rm to  problem \eqref{E:NS'}, i.e.,
\begin{equation*} \label{E:NS''}
\begin{cases}
& du(t)+\acal u(t) \, dt+B\bigl( u(t),u(t) \bigr)\, dt = f(t) \, dt+G\bigl( u(t)\bigr) \, dW(t),
 \qquad t \geq 0 , \\
& u(0) = u_0,
\end{cases}
\end{equation*}
if and only if  for all $ t \in [0,\infty) $ and all $\rv \in \vcal $ the following identity holds $\p $ - a.s.
\begin{eqnarray*}
 \ilsk{u(t)}{\rv}{\rH} &+&  \int_{0}^{t} \dual{\acal u(s)}{\rv}{} \, ds
+ \int_{0}^{t} \dual{B(u(s),u(s))}{\rv}{} \, ds  \nonumber \\
  &=& \ilsk{{u}_{0}}{\rv}{\rH} +  \int_{0}^{t} \dual{f(s)}{\rv}{} \, ds
 + \Dual{\int_{0}^{t} G(u(s))\, dW(s)}{\rv}{}
\end{eqnarray*}
and for all $T>0$,
\begin{equation}
\mathbb{E} \Bigl[ \sup_{t \in [0,T]} \norm{u(t)}{\rH}{2} + \int_{0}^{T} \norm{\nabla u(t)}{}{2} \, dt  \Bigr] < \infty .
\end{equation}
\end{definition}

Let us recall two basic concepts of uniqueness of the solution, i.e. pathwise uniqueness and uniqueness in law, see \cite{Ikeda+Watanabe_1981},
\cite{Ondrejat_2004}. Please note the following difference between problems   \eqref{E:NS} and \eqref{E:NS'}. In the former, a law of the initial data is prescribed, while in the latter a initial data is given.

\begin{definition}  \rm
We say that solutions of problem  \eqref{E:NS'}
has  \bf pathwise uniqueness property  \rm if and only if for all $\, \, {u}_{0} \in \rH $ and  $f:[0,\infty ) \to {\rV }^{\prime }$ the following condition holds
\begin{equation}\label{eqn-uniqueness}
\begin{array}{rcl}
&&\hspace{-5truecm}\lefteqn{\mbox{\it if } {u}^{i}, \, \,  i=1,2, \mbox{\it are strong solutions of problem }   \;\;\eqref{E:NS'} \;\;\mbox{\it on }  \;\; (\Omega ,\fcal , \fmath ,\p ,W)\;\; \mbox{\it satisfying Assumption \ref{assumption-second},} }  \\
& &  \mbox{\it then $\p $-a.s. for all } t \in [0,\infty), \, \, \, {u}^{1}(t) = {u}^{2}(t).
\end{array}
\end{equation}
\rm Assume that  $\, \, {u}_{0} \in \rH $ and  $f:[0,\infty ) \to {\rV }^{\prime }$. A solution ${u}$  to problem \eqref{E:NS'} on $(\Omega ,\fcal , \fmath ,\p ,W)$ satisfying Assumption \ref{assumption-second}, is said to be pathwise unique iff  for every solution $\tilde{u}$
 to problem \eqref{E:NS'} on the same $(\Omega ,\fcal , \fmath ,\p ,W)$, one has
\[\mbox{$\p $-a.s. for all } t \in [0,\infty), \, \, \, {u}(t) = \tilde{u}(t).  \]
\end{definition}

\begin{definition}  \rm
We say that  problem  \eqref{E:NS}
has  \bf  uniqueness  in law property \rm iff for every
Borel measure $\mu$ on $\rH$
 and every $f:[0,\infty ) \to {\rV }^{\prime }$ the following condition holds
\begin{eqnarray} \label{cond-unique in law}
& & \mbox{\it if } ({\Omega }^{i},{\fcal }^{i}, {\fmath }^{i}, {\p }^{i},{W}^{i},{u}^{i}), \, \,  i=1,2, \mbox{ \it  are  such solutions
  of problem  \eqref{E:NS}
  that }  \\
& &  \mbox{ \it then  } {Law}_{{\p }^{1}}({u}^{1}) = {Law}_{{\p }^{2}}({u}^{2}) \mbox{ on } \ccal \bigl( [0,\infty), {\rH}_{w} \bigr)    \cap {L}^{2}_{\textrm{loc}}([0,\infty);\rV ),
\nonumber
\end{eqnarray}
where  ${Law}_{{\p }^{i}}({u}^{i})$, $i=1,2$, are by  definition  probability measures on $\ccal \bigl( [0,\infty), {\rH}_{w} \bigr)
   \cap {L}^{2}_{\textrm{loc}}([0,\infty);\rV )$.
\end{definition}

\begin{cor}\label{cor-unique}
Assume that conditions \textbf{(H.1)}-\textbf{(H.3)} and \textbf{(H.5)} of Assumption \ref{assumption-main} are satisfied and
 that  the map $G$ satisfies inequality  \eqref{E:G_Lipsch} in part \textbf{(H.2)} of  Assumption \ref{assumption-main}
with a constant $L$ smaller than $\sqrt{2}$. Assume also that $\bigl( \Omega , \fcal , \mathbb{F} , \p, W  \bigr) $ satisfies  Assumption \ref{assumption-second}. Then for every   $\, \, {u}_{0} \in \rH $.
\begin{itemize}
\item[1)]   There exists a pathwise unique strong solution $u$ on $\bigl( \Omega , \fcal , \mathbb{F} , \p, W  \bigr) $ of problem  \eqref{E:NS'}.
\item [2)] Moreover, if  $u$  is a strong solution
of problem  \eqref{E:NS'} on $\bigl( \Omega , \fcal , \mathbb{F} , \p, W  \bigr) $,
then for $\p $-almost all $\omega \in \Omega $ the trajectory $u(\cdot , \omega )$ is equal almost everywhere to a continuous $\rH$-valued function defined on $[0,\infty)$.
\item[3)] The martingale solution  of problem  \eqref{E:NS} with ${\mu }_{0}=\delta_{u_0}$  is unique in law.  In particular, if
   $\bigl({\Omega }^{i},{\fcal }^{i}, {\fmath }^{i}, {\p }^{i},{W}^{i},{u}^{i}\bigr)$, $i=1,2$ t  are  such solutions to problem  \eqref{E:NS}, then
 for all $t\geq 0$, the laws on $\rH$  of $\rH$-valued random variables  $u^1(t)$ and $u^2(t)$ coincide.
\end{itemize}
\end{cor}

\begin{proof}
The proof of part (3) given in \cite{Brz+Motyl_2013} yields the uniqueness in law in the trajectory the space $\ccal \bigl( [0,\infty), {\rH}_{w} \bigr)    \cap {L}^{2}_{\textrm{loc}}([0,\infty);\rV )$, hence in $\ccal \bigl( [0,T], {\rH}_{w} \bigr) \cap {L}^{2}(0,T;\rV )$ for every $T>0$.
\end{proof}

Let us emphasize that, by definition, we require a martingale solution of the Navier-Stokes equation to satisfy inequality \eqref{E:mart_sol_energy-ineq}, i.e.
\begin{equation*}
\hat{\mathbb{E}} \Bigl[ \sup_{t \in [0,T]} \norm{u(t)}{\rH}{2} + \int_{0}^{T} \norm{\nabla u(t)}{}{2} \, dt  \Bigr] < \infty .
\end{equation*}
In Theorem \ref{T:existence_NS}, covering both 2D and 3D domains, we  proved that there exists a martingale solution satisfying stronger estimates, i.e. \eqref{E:H_estimate_NS}-\eqref{E:Poincare-NS-ineq}.
However, in the case when $\ocal $ is a 2D domain, we can prove that every martingale solution satisfies
these inequalities.

\begin{lemma} \label{lem-apriori estimates}
Assume that  $d=2$ and that conditions  \textbf{(H.1)}-\textbf{(H.3)} and \textbf{(H.5)} from Assumption \ref{assumption-main} are satisfied.
 Then the following holds.
\begin{itemize}
\item[(1) ] For every  $T>0$, ${R}_{1}>0$ and ${R}_{2}>0$ there exist constants ${C}_{1}(p)$ and ${C}_{2}(p)$ depending also on $T$, $R_1$ and $R_2$ such that
 if
$\mu_0$ is  a Borel probability measure  on $H$,
${f} \in  {L}^{p}(0,T; \rV^\prime )$ satisfy  $\int_H \vert x \vert^p \mu_0(dx) \leq R_1$
 and $ \vert f{\vert}_{{L}^{p}(0,T;\rV^\prime)}\leq R_2$, then every  martingale solution of problem \eqref{E:NS}
 with the initial data ${\mu }_{0}$ and the external force $f$,
 satisfies the following  estimates
\begin{equation} \label{E:H_estimate-2D}
\hat{\mathbb{E}} \bigl( \sup_{s\in [0,T] } {|u (s)|}_{\rH}^{p} \bigr) \le {C}_{1}(p)
\end{equation}
and
\begin{equation} \label{E:HV_estimate-2D}
 \hat{\mathbb{E}} \bigl[ \int_{0}^{T} {|u(s)|}_{\rH}^{p-2} \norm{ \nabla u (s)}{}{2} \, ds \bigr] \le {C}_{2}(p)  .
\end{equation}
In particular,
\begin{equation} \label{E:V_estimate-2D}
\hat{\mathbb{E}} \bigl[ \int_{0}^{T} \norm{ \nabla u (s)}{}{2} \, ds \bigr] \le {C}_{2}:= {C}_{2}(2).
\end{equation}
\item[(2) ]
Moreover, if $\ocal $ is a Poincar\'{e} domain and
the map $G$ satisfies inequality \eqref{E:G}  in part \textbf{(H.2)} of  Assumption \ref{assumption-main} with ${\lambda }_{0}=0$ (and with  $\rho \in [0,\infty)$ and $\eta \in (0,2]$), then the process $u$  satisfies additionally the following inequality for every $t \ge 0$
\begin{equation}  \label{E:Poincare-NS-ineq-2D}
  \hat{\mathbb{E}} [\, \norm{u(t)}{\rH}{2} \, ]  +  \frac{\eta }{2} \hat{\e } \biggl[ \int_{0}^{t} \norm{\nabla u (s)}{}{2} \, ds \biggr]  \biggr)
   \le \hat{\mathbb{E}} [\, \norm{u(0)}{\rH}{2} \, ]
 + \frac{2}{\eta } \int_{0}^{t} \norm{f(s)}{{\rv }^{\prime }}{2} \, ds + \rho t.
\end{equation}
\end{itemize}
\end{lemma}

\bigskip
The proof of Lemma \ref{lem-apriori estimates}  is  similar to the proof of estimates (5.4), (5.5) and (5.6) from Appendix in \cite{Brz+Motyl_2013}. The difference is that the solution process $u$ to which the It\^o formula (in a classical form, see for instance \cite{Ikeda+Watanabe_1981}) was applied was taking values in a finite dimensional Hilbert space $H_n$ and $u$ was a  solution in the most classical way. Now, $u_n$ is martingale solution to problem \eqref{E:NS}, see Definition \ref{def-sol-martingale}.

If we assume that
$d=2$, by Lemma III.3.4 p. 198 in \cite{Temam_2001},  we infer that the regularity assumption
\eqref{eqn-regularity of u}   implies that
\[
  B\big(u(\cdot , \omega ),u(\cdot , \omega )\big) \in
    {L}^{2}_{\textrm{loc}}([0,\infty);\rV^\prime ) \mbox{ for } \hat{\mathbb{P}}\mbox{-a.a. } \omega\in \Omega.
\]
This however does not imply that
\[
\hat{\mathbb{E}} \int_0^T \vert B(u(t),u(t))\vert^2_{\rV^\prime}\, dt<\infty
\]
what is necessary in order to apply the infinite dimensional It\^o Lemma from \cite{Pardoux_1979}.

\bigskip  \noindent
Fortunately, we can proceed as in the proof of the uniqueness result, i.e. Lemma 7.3 from \cite{Brz+Motyl_2013}, i.e. introduce a family $\tau_N$, $N\in\mathbb{N}$ of  the stopping times defined by
\begin{equation}  \label{eqn-tau_N}
   {\tau }_{N} :=  \inf \{ t \in [0,\infty) : \norm{u(t)}{\rH}{} \geq N  \} ,
\qquad N \in \nat  .
\end{equation}
and then consider a stopped process $u(t\wedge \tau_N)$, $t\geq 0$. Note that with this definition of the stopping time $\tau_N$, we have
\[
\hat{\mathbb{E}} \int_0^{T \wedge \tau_N} \vert B(u(t),u(t))\vert^2_{\rV^\prime}\, dt\leq CN^{2} \hat{\mathbb{E}} \int_0^{T} \Vert u(t)\Vert^2\, dt <\infty.
\]

\bigskip \noindent
\bf Remark. \rm
If $d=3$, then
\[
  B\big(u(\cdot , \omega ),u(\cdot , \omega )\big) \in
    {L}^{4/3}_{}(0,T;\rV^\prime ) \quad \mbox{ for } \hat{\mathbb{P}}\mbox{-a.a. } \omega\in \Omega.
\]
Thus, in this case the above procedure with the stopping time ${\tau }_{N}$ does not help.

\bigskip  \noindent
\begin{proof}[Proof of Lemma \ref{lem-apriori estimates}]
Let us fix $p $ satisfying  condition \eqref{eqn-p_cond}.
As in the proof of Lemma A.\ref{L:Galerkin_estimates},
we apply the It\^{o} formula from \cite{Pardoux_1979} to the function $F$ defined by
\[F: \rH \ni x \mapsto  {|x|}_{\rH}^{p}\in \mathbb{R}.\]
With the above comments in mind and using Remark \ref{R:def-mart-sol-test}, we have, for $\, t \in [0,\infty)$,
\begin{eqnarray}
{|{u} (t\wedge \tau_N)|}^{p} &-&  {|{u} (0)|}^{p}
  =  \int_0^{t\wedge \tau_N} \Bigl[ p \,  {|{u} (s)|}^{p-2} \dual{{u} (s)}{- \acal {u} (s)
       -  B \bigl( {u} (s) \bigr) +  f(s)}{}  \nonumber \\
 &&\hspace{2.2truecm}\lefteqn{+    \frac{1}{2} \tr \;\bigl[  F^{\prime\prime}(u(s))\bigl( G ({u} (s)), G ({u} (s)) \bigr)  \bigr] \Bigr] \, ds}
\nonumber\\&+& p \, \int_0^{t\wedge \tau_N}   {|{u} (s)|}^{p-2} \dual{{u} (s)}{ G ( {u} (s) ) \, d \hat{W}(s) }{}
 \nonumber\\
&=&
  \int_0^{t\wedge \tau_N}  \Bigl[ - p \, {|{u} (s)|}^{p-2} \Norm{{u} (s)}{}{2}
   + p \, {|{u} (s)|}^{p-2} \dual{{u} (s)}{  f(s)}{}
    \nonumber\\&&\hspace{1.2truecm}\lefteqn{+ \frac{1}{2} \tr \bigl[  F^{\prime\prime}(u(s))\bigl( G ({u} (s)), G ({u} (s)) \bigr)  \bigr] \Bigr] \, ds} \nonumber\\
 &+&  p \, \int_0^{t\wedge \tau_N}  {|{u} (s)|}^{p-2} \dual{{u} (s)}{ G ( {u} (s) ) \, d \hat{W}(s) }{}  .
   \label{ineq-01-2D}
\end{eqnarray}
Proceeding as  in the proof of Lemma A.\ref{L:Galerkin_estimates},  we obtain
\begin{equation}
\begin{array}{rcl}
  {|u (t\wedge \tau_N )|}^{p}   &+&\delta
  \, \int_{0}^{t\wedge \tau_N}{|u (s)|}^{p-2} \norm{\nabla u (s)}{}{2} \, ds \\
  && \\
 & \le&    {|u (0)|}^{p} +K_p({\lambda }_{0},\rho)
\int_{0}^{t\wedge \tau_N} \, {|u (s)|}^{p} \, ds
  +   \frac{2\rho}{p} t \,
  + \eps^{-p/2} \int_{0}^{t\wedge \tau_N } {|f(t)|}_{\rV^{\prime }}^{p}\, ds
  \\
  && \\
& +& p \,  \int_{0}^{t}{|u (s)|}^{p-2} \dual{u (s)}{ G ( u (s) ) \, d \hat{W}(s) }{} ,
\quad t \in [0,\infty ) ,
\end{array}     \label{E:apriori-2D}
\end{equation}
where ${K}_{p}({\lambda }_{0},\rho )=\frac{p-1}{2}[{\lambda }_{0}p+2+\rho (p-2)]$.

By  the definition of the stopping time $\tau_N$  we infer that the process
\[
 {\mu }_{N}(t):= \int_{0}^{t \wedge \tau_N} {|u (s)|}^{p-2} \dual{u (s)}{G(u (s)) \, d \hat{W}(s) }{} ,
 \qquad t \in [0,\infty)
\]
is a martingale. Indeed, if we define a map
\[
g:\rV\ni u\mapsto \{ \rK\ni k \mapsto   \dual{u }{G(u)k }{} \in \rH\} \in \lhs (\rK,\mathbb{R})
\]
then $\mu_N(t)=  \int_{0}^{t \wedge \tau_N} {|u(s)|}^{p-2} g(u(s))dW(s)$ and,
since the map $G$ satisfies inequality \eqref{E:G}  in part \textbf{(H.2)} of  Assumption \ref{assumption-main}, we infer that for every $t\geq 0$,
\begin{eqnarray}
\label{ineq-auxiliary_01-2D}
&&\int_{0}^{t \wedge \tau_N} \Vert\, {|u(s)|}^{p-2} g(u(s))  \Vert^2_{\lhs (\rK,\mathbb{R})}\, ds=
\int_{0}^{t \wedge \tau_N} {|u(s)|}^{p-2} \Vert  \,g(u(s))  \Vert^2_{\lhs (\rK,\mathbb{R})}\, ds
\\
\nonumber
&&\leq \int_{0}^{t \wedge \tau_N} {|u(s)|}^{p-2} {|u(s)|}^{2} \Vert  G(u(s))  \Vert^2_{\lhs (\rK,\mathbb{\rH})} \, ds
\leq \int_{0}^{t \wedge \tau_N} {|u(s)|}^{p} \big[ (2- \eta ) \, \norm{\nabla u (t)}{}{2}
  + {\lambda }_{0} {|u (t)|}^{2} + \rho \big] \, ds
\\
\nonumber
&&\leq  (2- \eta ) N^p \int_{0}^{t \wedge \tau_N}   \, \norm{\nabla u (t)}{}{2}\, dt+ t N^p(\lambda_0N^2+\rho).
\end{eqnarray}
Hence by inequality \eqref{E:mart_sol_energy-ineq}
 we infer that
\[ \hat{\mathbb{E}} \, \int_{0}^{t \wedge \tau_N} \Vert\, {|u(s)|}^{p-2} g(u(s))  \,\Vert^2_{\lhs (\rK,\mathbb{R})}\, ds<\infty, \;\; t\geq 0.
\]
and thus we infer, as claimed,  that the process $ {\mu }_{N}$ is a martingale. Hence,
  $\mathbb{E}[{\mu }_{N} (t) ] = 0 $. Let us now fix $T>0$.  By taking expectation in inequality \eqref{E:apriori-2D} we infer that
\begin{equation}
\begin{array}{rcl}
  \hat{\mathbb{E}}\,\bigl[ {|u (t \wedge \tau_N)|}^{p} \,\bigr]
  \, \, &\leq&   \, \, \hat{\mathbb{E}}[ {|u(0)|}^{p} ] +
K_p({\lambda }_{0},\rho)\int_{0}^{t \wedge \tau_N} \,
 \hat{\mathbb{E}}\,\bigl[ {|u (s)|}^{p} \bigr] \, ds
+ \frac{2\rho}{p} (t \wedge \tau_N)
  + \eps^{-p/2} (t\wedge \tau_N) {|f|}_{{\rV }^{\prime }}^{p}
    \nonumber\\
 &&\nonumber\\
    &\leq &
  \hat{\mathbb{E}}[ {|u(0)|}^{p} ] +
K_p({\lambda }_{0},\rho)\int_{0}^{t \wedge \tau_N} \,
 \hat{\mathbb{E}}\,\bigl[ {|u (s \wedge \tau_N)|}^{p} \bigr] \, ds   + T\big( \frac{2\rho}{p}+ \eps^{-p/2} {|f|}_{{\rV }^{\prime }}^{p} \big),   \quad t \in [0,T].
  \end{array}   \label{E:apriori'-2D}
\end{equation}
Hence by the Gronwall Lemma there exists a constant
$C=C_p(  T,\eta ,{\lambda }_{0},\rho ,\hat{\mathbb{E}}[{|u(0)|}^{p} ],\norm{f}{{L}^{p}(0,T;{\rV }^{\prime })}{} )>0$ such that
\begin{equation} \label{E:app_H_est-2D}
  \hat{\mathbb{E}}\,\bigl[ {|u (t\wedge \tau_N )|}^{p} \bigr]  \le C , \qquad  \, t \in [0,T].
\end{equation}
Using this bound in \eqref{E:apriori-2D} we also obtain
\begin{equation} \label{E:app_HV_est-1-2D}
 \hat{\mathbb{E}}\,\biggl[ \int_{0}^{T\wedge \tau_N}{|u (s)|}^{p-2} \norm{\nabla u (s)}{}{2} \, ds  \biggr]
  \le C
\end{equation}
for a new constant  $C=\tilde C_p(\eta,\hat{\mathbb{E}}\, |u(0) |^{p},\hat{\mathbb{E}}\,\int_{0}^{T} {|f(s)|}_{{\rV }^{\prime }}^{p}\, ds )>0$.
Finally, taking the limit $N\to\infty$ and observing that $T\wedge \tau_N\to T$, by the Lebesgue dominated convergence Theorem we infer that for the same constant $C$ we have
\begin{equation} \label{E:app_H_est-1-2D}
 \sup_{t \in [0,T]} \hat{\mathbb{E}}\,\bigl[ {|u (t )|}^{p} \bigr]  \le C ,
\end{equation}
\begin{equation} \label{E:app_HV_est-2D}
 \hat{\mathbb{E}}\,\biggl[ \int_{0}^{T}{|u (s)|}^{p-2} \norm{\nabla u (s)}{}{2} \, ds  \biggr]
  \le C.
\end{equation}
This completes the proof of estimates \eqref{E:HV_estimate-2D} and \eqref{E:V_estimate-2D}.
The proof of inequality \eqref{E:H_estimate-2D} is the same as the proof of inequality \eqref{E:H_estimate_Galerkin_p} and thus omitted.

To prove inequality \eqref{E:Poincare-NS-ineq-2D} in the case $\mathcal{O}$ is a Poincar\'{e} domain we use the same arguments as the proof of inequality \eqref{E:Poincare_Galerkin}. This time however, the  solution to the Galerkin approximating equation  is replaced by the stopped process $u(t\wedge {\tau }_{N})$, $t\ge 0$. Let us recall that  in the space $\rV $ we consider the inner product $\dirilsk{\cdot }{\cdot }{}$ given by \eqref{E:il_sk_Dir}.

By identity \eqref{ineq-01-2D}  with $p=2$,   we have
\begin{eqnarray}
\label{ineq-02-2D}
  {|{u} (t\wedge {\tau }_{N})|}^{2} &-&  {|{u} (0)|}^{2}
=
  \int_0^{t\wedge {\tau }_{N}}  \Bigl\{ - 2 \,  \Norm{{u} (s)}{}{2}
   + 2 \,  \dual{{u} (s)}{  f}{} + \frac{1}{2} \tr \bigl[  F^{\prime\prime}(u(s))\bigl( G ({u} (s)), G ({u} (s)) \bigr)  \bigr] \Bigr\} \, ds \nonumber\\
 &+&  2 \, \int_0^{t\wedge {\tau }_{N}}   \dual{{u} (s)}{ G ( {u} (s) ) \, d \hat{W}(s) }{}, \;\; t\geq 0  .
 \nonumber
\end{eqnarray}
Since $\hat{\mathbb{E}}\big(\int_{0}^{t\wedge {\tau }_{N}} \dual{G(u(s))}{u(s)\, d\hat{W}(s)}{} \big)=0$, we infer that
\begin{eqnarray*}
\hat{\mathbb{E}} |u(t\wedge {\tau }_{N}){|}_{\rH}^{2}
&  \leq &   \hat{\mathbb{E}} [ \, {|u(0)|}_{\rH}^{2} \, ] +
\hat{\mathbb{E}}  \int_{0}^{t\wedge {\tau }_{N}}  \bigl\{  -2  \Norm{u(s)}{}{2}  +2\dual{f(s)}{u(s)}{}
 \bigr\} \, ds  + \hat{\mathbb{E}} \int_{0}^{t\wedge {\tau }_{N}}  \norm{G(u(s)) }{\lhs (\rK,\rH)}{2} \, ds.
\end{eqnarray*}
Taking next the $N\to \infty$ limit, since the map $G$ satisfies inequality \eqref{E:G}  in part \textbf{(H.2)} of  Assumption \ref{assumption-main} with ${\lambda }_{0}=0$, i.e $\norm{G(u(s))}{\lhs (K,\rH)}{2} \le (2-\eta )\Norm{{u} (s)}{}{2} +\varrho $,  we get
\begin{eqnarray}
\hat{\mathbb{E}} \norm{u(t)}{\rH}{2}
&  \leq & -\eta
\mathbb{E}  \int_{0}^{t}      \Norm{{u}(s)}{}{2}   \, ds  +\hat{\mathbb{E}} [ \,{|u(0)|}_{\rH}^{2} \, ]
 + 2\hat{\mathbb{E}} \int_{0}^{t} \dual{ f(s)  }{u(s)}{} \, ds +\varrho t.
  \label{ineq-apriori_2-2D}
\end{eqnarray}
Since
$2\dual{f}{u(s)}{} \leq \frac{\eta}2 \norm{\nabla u(s)}{}{2}+\frac{2}{\eta} \norm{f }{{\rV }^\prime}{2} $ we infer that
\begin{eqnarray}
\hat{\mathbb{E}} \norm{u(t)}{\rH}{2}
&  \leq & -\frac{\eta}2
\hat{\mathbb{E}}  \int_{0}^{t}      \Norm{u(s)}{}{2}   \, ds
   +\hat{\mathbb{E}} [{|u(0)|}_{\rH}^{2}] + \frac{2}{\eta} \int_{0}^{t}\norm{f(s) }{{\rV }^\prime}{2}
 +\varrho  t,\;\; t\geq 0.
\end{eqnarray}
The proof of inequality \eqref{E:Poincare-NS-ineq-2D}  is thus complete.
This completes the proof of Lemma \ref{lem-apriori estimates}.
\end{proof}

Note that if $f:[0,\infty ) \to {\rV }^{\prime }$ is constant, it satisfies assumption (H.3). In this case we will write $f\in {\rV }^{\prime }$.

\bigskip
By Theorem \ref{thm-continuous dependence-existence}  Corollary \ref{cor-unique} and Lemma \ref{lem-apriori estimates} we obtain the following result about  the continuous dependence of the  solutions to 2D SNSEs with respect to the initial data  and the external forces.

\begin{theorem} \label{thm-weak continuous dependence-existence_2D}
Let $d=2$. Let parts  \textbf{(H.1)}-\textbf{(H.2)}, \textbf{(H.5)} and \eqref{E:G_Lipsch} with a constant $L$ smaller than $\sqrt{2}$, of Assumption \ref{assumption-main}, be satisfied.
Assume that    $\, \, {u}_{0} \in \rH $, $ f \in  \rV^{\prime }$ and  that an $\rH$-valued  sequence $\, \, \big({u}_{0,n}\big)_{n=1}^\infty$ is  weakly convergent   in $\rH$ to $u_0$,
and that an $\rV^\prime$-valued  sequence $\, \, \big({f}_n\big)_{n=1}$ is  weakly convergent  in $\rV^\prime$  to $f$.
Let
\[
\bigl( \Omega_n , \fcal_n , \fmath_n ,\p_n  , W_n, u_n \bigr)
\]
be a  martingale solution of problem \eqref{E:NS'} on $[0,\infty )$  with the initial data ${u}_{0,n}$ and the external force $f_n$.
Then for every $T>0$ there exist
\begin{itemize}
\item a subsequence $({n}_{k}{)}_{k}$,
\item a stochastic basis
$\bigl( \tOmega , \tfcal ,  {\tilde{\mathbb{F}}} , \tp  \bigr) $, where $\tilde{\mathbb{F}}={\{ \tfcal {}_{t} \} }_{t \ge 0}$,
\item a cylindrical Wiener process $\tW=\tW (t)$, $t\in [0,\infty)$ defined on this basis,
\item and  an $\mathbb{F}$-progressively measurable processes $\tu (t)$,  $\big( \tunk (t)\big)_{k \ge 1}, t \in [0,T] $ (defined on this basis) with  laws supported in
$ \mathcal{Z}_{T} $ such that
\end{itemize}
\begin{equation}   \label{E:Skorokhod_appl_2D}
  \tunk \mbox{ \it has the same law as } \unk \mbox{ \it on } \mathcal{Z}_{T}
    \mbox{ \it and } \tunk \to \tu \mbox{ \it in } \mathcal{Z}_{T},
    \quad  \tp \mbox{ - \it a.s.}
\end{equation}
and  the  system
\begin{equation}   \label{eqn-system}
\bigl( \tOmega , \tfcal ,  {\tilde{\mathbb{F}}}, \tp , \tW,\tu  \bigr)
\end{equation}
is   a martingale solution    to  problem  \eqref{E:NS'} on the interval $[0,T]$ with the initial law ${\delta }_{{u}_{0}}$.
In particular,
for all $t\in [0,T]$ and  $\rv \in \rV$
\begin{eqnarray*}
& &\ilsk{\tu (t)}{\rv }{\rH} \,
 - \ilsk{\tu (0)}{\rv }{\rH}  +  \int_{0}^{t} \dual{ \acal \tu (s) }{\rv }{} \, ds
  +   \int_{0}^{t} \dual{ B \bigl( \tu (s)\bigr)}{\rv }{} \, ds   \\
& &\qquad \qquad  =  \int_{0}^{t} \dual{ f}{\rv }{} \, ds
+ \Dual{ \int_{0}^{t} G \bigl( \tu (s) \bigr) \, d \tW (s)}{\rv }{}.
\end{eqnarray*}
Moreover,  the process $\tu $ satisfies  the following inequality for every $p$ satisfying condition \eqref{eqn-p_cond} and $q\in [1,p]$
\begin{equation}  \label{tu_estimates_2D}
  \tilde{\mathbb{E}}\,\bigl[ \,\sup_{ s\in [0,T]  } {|\tu (s)| }_{\rH}^{q}\,\bigr]
 + \tilde{\mathbb{E}}\,\Bigl[ \,\int_{0}^{T} \Norm{\tu (s)}{}{2}\, ds \,\Bigr] < \infty  .
\end{equation}
\end{theorem}

\begin{proof} Let $p$ be any exponent satisfying condition \eqref{eqn-p_cond}.
Since the sequences $\big({u}_{0,n}\big)_{n=1}^\infty  \subset \rH $
and $\, {({f}_{n})}_{n=1}^{\infty } \subset {\rV }^{\prime } \,  $
convergent  weakly in $\rH$ and ${\rV }^{\prime }$, respectively, we infer that there exist ${R}_{1}>0$ and ${R}_{2}>0$  such that
\[
 \sup_{n \in \nat } \norm{{u}_{0,n}}{\rH}{} \le {R}_{1}  \quad \mbox{ and } \quad
\sup_{n \in \nat } \Norm{{f}_{n}}{{\rV }^{\prime }}{} \le {R}_{2}.
\]
By Lemma \ref{lem-apriori estimates} we infer that the processes $\un $, $n \in \nat $,
 satisfy inequalities \eqref{E:H_estimate_NS}-\eqref{E:V_estimate_NS}.
Thus the first part of the assertion follows directly from Theorem \ref{thm-continuous dependence-existence}. Inequality \eqref{tu_estimates_2D} follows again from Lemma \ref{lem-apriori estimates}. The proof of theorem is thus complete.
\end{proof}

\begin{remark}
Although this has not been studied in the present paper, we believe that methods developed here can be used to study the continuous dependence of the solutions on other parameters entering our equations, for instance the linear operator $A$, the nonlinearity $B$ and the diffusion operator $G$.
\end{remark}

\section{Existence of an invariant measure for Stochastic NSEs on 2-dimensional domains}\label{sec-invariant}

In this section we assume that  $d=2$.
Since we are interested in the existence of invariant measures we assume that the domain $\mathcal{O}$ satisfies the Poincar\'e  condition see \eqref{cond-Poincare}.
\footnote{It is well known that this condition  holds  if the domain $\mathcal{O} $ is bounded in some direction, i.e. there exists a vector $h \in \rd $ such that $\mathcal{O} \cap (h+\mathcal{O}) = \emptyset  .$}
However,
our results are true for general  domains for the stochastic damped Navier-Stokes equations, see for instance \cite{Constantin+Ramos_2007}.

Since we assume that $\mathcal{O}$ is a Poincar\'{e} domain, by the Poincar\'{e} inequality, see \eqref{cond-Poincare},  the functional given by the formula
\begin{equation} \label{E:norm_V_Poincare}
\Norm{u}{}{} = \norm{\nabla u}{{L}^{2}}{} , \qquad u \in  \rV ,
\end{equation}
is a norm in the space $\rV $  equivalent to the norm given by \eqref{E:norm_V}.

\it In the sequel, in the space $\rV $ we  consider the norm given by \eqref{E:norm_V_Poincare}. \rm

\bigskip
We aim in this section to prove  that, under some natural assumptions,   problem \eqref{E:NS} has
an invariant measure. Let us fix, as in Assumptions \ref{assumption-second},  a stochastic basis $\bigl( \Omega , \fcal , \mathbb{F} , \p  \bigr) $ with a filtration $\mathbb{F}={\{ \ft \} }_{t \ge 0}$; a canonical cylindrical  Wiener process $W= W(t)$ in a separable Hilbert space $\rK$ defined on the stochastic basis $\bigl( \Omega , \fcal , \mathbb{F} , \p  \bigr) $. We also fix  a function $G: \rH \to \lhs (\rK,{\rV }^{\prime }) $ satisfying condition \textbf{(H.2)} in Assumption \ref{assumption-main} and,
in addition,  the Lipschitz condition \eqref{E:G_Lipsch}  with a  constant $L$  smaller than $\sqrt{2}$,  and inequality \eqref{E:G}  with ${\lambda}_{0}=0$.
The last assumption on ${\lambda }_{0}$ corresponds to the fact that in $\mathcal{O} $ we consider the norm given by \eqref{E:norm_V_Poincare}.
In what follows the initial data $u_0$  will be an element of the space $\rH$.
By $u(t,{u}_{0})$, $t\geq 0$,  we denote the unique solution to the problem \eqref{E:NS'} (defined on the above stochastic basis satisfying Assumptions \ref{assumption-second}).

For any  bounded  Borel  function
$\varphi \in {\mathcal{B}}_{b}(\rH)$ and $t \geq 0$ we define

\begin{equation}
\label{eqn:semigroup}
 (P_t \varphi)({u}_{0}) = \mathbb{E}[\varphi (u(t,{u}_{0}))], \quad {u}_{0} \in \rH.
\end{equation}
Since by Lemma \ref{L:2D_NS_regularity} the trajectories $u(\cdot ,{u}_{0})$ are continuous, $(P_t)_{t \geq 0}$
is a stochastically continuous semigroup on the Banach space $\mathcal{C}_b(\rH)$. This means that for every $\varphi \in {\mathcal{C}}_{b}(\rH)$ and ${u}_{0} \in \rH$
\[
   \lim_{t\to 0} {P}_{t}\varphi ({u}_{0}) = {u}_{0} .
\]

As a consequence of Corollary \ref{cor-unique} we have the following result.

\begin{prop} \label{prop-Markov}
 The family $u(t,{u}_{0})$, $t\geq 0$, ${u}_{0}\in \rH$ is
 Markov. In particular, $P_{t+s}=P_tP_s$ for $t,s \ge 0$.
\end{prop}

The proof of Proposition \ref{prop-Markov} is standard and thus omitted, see e.g.
\cite{A+Brz+Wu_2010}, \cite[Section 9.2]{DaPrato+Zabczyk_1992},  \cite[Section 9.7]{Peszat_Z_2007}.

\begin{prop}
\label{prop-Feller_bw} The  semigroup $P_t$ is $bw$-Feller, i.e.
if  $\phi:\rH\to\mathbb{R}$ is a bounded sequentially weakly continuous  function and  $t>0$ then
$P_t\phi:\rH\to\mathbb{R}$ is also a bounded sequentially weakly continuous  function. In particular, if
 ${u}_{0n} \to {u}_{0}$ weakly in $\rH$ then
\[
P_t\phi({u}_{0n}) \to P_t\phi({u}_{0}).\]
\end{prop}

\begin{proof}[Proof of Proposition \ref{prop-Feller_bw}]
Let us choose and fix $t>0$, ${u}_{0}\in \rH$ and an $\rH$-valued sequence $({u}_{0n})$ that is weakly convergent to ${u}_{0}$ in $\rH$.   Let also
 $\phi:\rH\to\mathbb{R}$ be a bounded sequentially weakly continuous  function. Let us choose an auxiliary time $T\in (t,\infty)$.

Since obviously the function $P_t\phi:\rH\to\mathbb{R}$ is  bounded, we only need to prove that it is  sequentially weakly continuous.

Let $\un (\cdot ) = u (\cdot ,{u}_{0n})$, respectively $u (\cdot ) = u (\cdot ,{u}_{0})$, be a strong solution of problem \eqref{E:NS'} on $[0,\infty )$  with the initial data  $u_{0n}$, resp. $u_{0}$.
We assume that these processes are defined on the stochastic basis  $( \Omega , \fcal ,  {\mathbb{F}}, \p , W )$.
By Theorem \ref{thm-weak continuous dependence-existence_2D} there exist (depending on $T$)
\begin{itemize}
\item a subsequence $({n}_{k}{)}_{k}$,
\item a stochastic basis
$\bigl( \tOmega , \tfcal ,  {\tilde{\mathbb{F}}} , \tp  \bigr) $, where $\tilde{\mathbb{F}}={\{ \tfcal {}_{s} \} }_{s \in [0,T]}$,
\item a cylindrical Wiener process $\tW=\tW (s)$, $s\in [0,T]$ defined on this basis,
\item and  an $\mathbb{F}$-progressively measurable processes $\tu (s)$,  $\big( \tunk (s)\big)_{k \ge 1}, s \in [0,T]$ (defined on this basis) with  laws supported in
$ \mathcal{Z}_{T} $ such that
\end{itemize}
\begin{equation}   \label{E:Skorokhod_appl_2D_bw_Feller}
  \tunk \mbox{ \it has the same law as } \unk \mbox{ \it on } \mathcal{Z}_{T}
    \mbox{ \it and } \tunk \to \tu \mbox{ \it in } \mathcal{Z}_{T},
    \quad  \tp \mbox{ - \it a.s.}
\end{equation}
and  the  system
\begin{equation}   \label{eqn-system_bw_Feller}
\bigl( \tOmega , \tfcal ,  {\tilde{\mathbb{F}}}, \tp , \tW,\tu  \bigr)
\end{equation}
is   a martingale solution to  problem  \eqref{E:NS'} on the interval $[0,T]$ with the initial data $u_0$.

In particular, by \eqref{E:Skorokhod_appl_2D_bw_Feller}, $\tp$-almost surely
\[
 \tunk (t) \to \tu (t) \ \mbox{ weakly in } \rH .
\]
Since the function $\phi:\rH\to\mathbb{R}$ is sequentially weakly continuous, we infer that $\tp $-a.s.,
\[
\phi (\tunk (t) ) \to \phi (\tu (t)) \ \mbox{  in } \mathbb{R}.
\]
Therefore,  since the function $\phi:\rH\to\mathbb{R}$ is also bounded, by the Lebesgue Dominated Convergence Theorem we infer that
\begin{equation}  \label{E:bw_Feller_1}
\lim_{k\to \infty }\tilde{\mathbb{E}}[\phi (\tunk (t) )] = \tilde{\mathbb{E}}[\phi(\tu (t))].
\end{equation}
From the equality of laws of $\tunk $ and $\unk $, $k \in \nat $, on the space ${\mathcal{Z}}_{T}$ we infer that
\begin{equation} \label{E:bw_Feller_2}
\tilde{\mathbb{E}}[\phi (\tunk (t) )] = \mathbb{E}[\phi (\unk (t) )] = {P}_{t}\phi ({u}_{0{n}_{k}}).
\end{equation}
Since by assumptions $( \Omega , \fcal ,  {\mathbb{F}}, \p , W,u )$ is  a martingale  solution of equation \eqref{E:NS'} with the initial data $u_0$
and $\bigl( \tOmega , \tfcal ,  {\tilde{\mathbb{F}}}, \tp , \tW,\tu  \bigr)$ is also a martingale  solution with the initial of equation \eqref{E:NS'} with the initial data $u_0$
 and since the solution of \eqref{E:NS'} is unique in law, we infer that
\[
  \mbox{the processes $u$ and $\tu $ have the same law on the space } {\mathcal{Z}}_{t}.
\]
Hence
\begin{equation}  \label{E:bw_Feller_3}
     \tilde{\mathbb{E}}[\phi(\tu (t))] = \mathbb{E}[\phi(u (t))] = {P}_{t}\phi ({u}_{0}).
\end{equation}
Thus by \eqref{E:bw_Feller_1}, \eqref{E:bw_Feller_2} and \eqref{E:bw_Feller_3}, we infer that
\[
   \lim_{k\to \infty } {P}_{t}\phi ({u}_{0{n}_{k}}) =  {P}_{t}\phi ({u}_{0}).
\]
Using the sub-subsequence argument, we infer that the whole sequence ${({P}_{t}\phi ({u}_{0n}))}_{n\in \nat }$ is convergent and
\[
   \lim_{n\to \infty } {P}_{t}\phi ({u}_{0n}) =  {P}_{t}\phi ({u}_{0}),
\]
which completes the proof of Proposition \ref{prop-Feller_bw}.
\end{proof}

\begin{remark} \rm
From inequality \eqref{E:Poincare-NS-ineq-2D}
 and the Poincar\'{e} inequality \eqref{cond-Poincare}, it follows that the following inequality holds for the strong solution $u$ of problem \eqref{E:NS'} defined on the stochastic basis  $( \Omega , \fcal ,  {\mathbb{F}}, \p , W )$
\begin{equation}
  \int_{0}^{t} \mathbb{E} {|u(s)|}_{\rH}^{2} \, ds
\le  \frac{2}{C\eta }{|u_0|}_{\rH}^{2}
 + \frac{2}{C\eta } \Bigl( \frac{2}{\eta} \norm{f }{{\rV }^\prime}{2} +\varrho \Bigr) t,
\qquad t \ge 0 .
 \label{E:ineq_apriori}
\end{equation}
\end{remark}

\begin{proof} [Proof of inequality \eqref{E:ineq_apriori}]
Let us fix $t \ge 0 $. By the Poincar\'{e} inequality \eqref{cond-Poincare} for almost all $s \in [0,t]$,
\[
 {|u(s)|}_{\rH}^{2} \le \frac{1}{C} {|\nabla u(s)|}_{{L}^{2}}^{2}.
\]
By \eqref{E:Poincare-NS-ineq-2D}, in particular, we obtain
\[
  \frac{\eta}{2}  \mathbb{E}\, \int_{0}^{t}\norm{\nabla u(s)}{}{2} \, ds
  \leq {|u_0|}_{\rH}^{2} + \Bigl( \frac{2}{\eta} \vert f\vert^2_{{\rV }^\prime} +\varrho  \Bigr) {t}
\]
Hence we infer that
\[
  \int_{0}^{t} \mathbb{E} {|u(s)|}_{\rH}^{2} \, ds
\le \frac{1}{C}  \mathbb{E}\, \int_{0}^{t}\norm{\nabla u(s)}{}{2} \, ds
\le  \frac{2}{C\eta }{|u_0|}_{\rH}^{2}
 + \frac{2}{C\eta } \Bigl( \frac{2}{\eta} \norm{f }{{\rV }^\prime}{2} +\varrho \Bigr) t , \;\;\ t\geq 0,
\]
i.e. inequality \eqref{E:ineq_apriori} holds.
\end{proof}

Using inequality \eqref{E:ineq_apriori} we deduce the following result.

\begin{cor} \label{cor-Krylov_Bogoliubov_cond}
Let ${u}_{0} \in \rH$ and let $u(t)$, $t\ge 0$, be the unique solution to the problem \eqref{E:NS'} starting from ${u}_{0}$. Then there exists ${T}_{0} \ge 0 $ such that for every $\eps >0$ there exists $R>0$ such that
\begin{equation}
  \sup_{T\ge {T}_{0}} \frac{1}{T} \int_{0}^{T} ({P}^{\ast }_{s}{\delta }_{{u}_{0}}) (\rH \setminus {\bar \ball }_{R}) \, ds \le \eps ,
\end{equation}
where ${\bar \ball }_{R}= \{ \rv\in \rH: \ {|\rv|}_{\rH} \le R \} $.
\end{cor}

\begin{proof}
Using the Chebyshev inequality and inequality \eqref{E:ineq_apriori}  we infer that for every $T \ge 0$ and $R>0 $
\begin{eqnarray*}
\frac{1}{T} \int_{0}^{T} ({P}^{\ast }_{s}{\delta }_{{u}_{0}}) (\rH \setminus {\bar \ball }_{R}) \, ds
&=&\frac{1}{T} \int_{0}^{T} \mathbb{P} (\{ {|u(s)|}_{\rH} >R \} ) \, ds
\le \frac{1}{T{R}^{2}} \int_{0}^{T} \mathbb{E}{|u(s)|}_{\rH}^{2} \, ds \\
&\le & \frac{1}{T{R}^{2}} \Bigl[ \frac{2}{C\eta }{|{u}_{0}|}_{\rH}^{2}
 + \frac{2}{C\eta } \Bigl( \frac{2}{\eta} \norm{f }{{\rV }^\prime}{2} +\varrho \Bigr) T \Bigr] \\
&=&  \frac{1}{T{R}^{2}} \frac{2}{C\eta }{|{u}_{0}|}_{\rH}^{2}
+\frac{1}{{R}^{2}} \frac{2}{C\eta }\Bigl( \frac{2}{\eta} \norm{f }{{\rV }^\prime}{2} +\varrho \Bigr) .
\end{eqnarray*}
Thus the assertion follows.
\end{proof}

By Proposition \ref{prop-Feller_bw}, Corollary \ref{cor-Krylov_Bogoliubov_cond} and the Maslowski-Seidler Theorem \cite[Proposition 3.1]{Maslowski+Seidler_1999}
 we deduce the following main result of our paper.

\begin{theorem}\label{thm-main}
Let $\mathcal{O} \subset {\mathbb{R}}^{2}$ be a Poincar\'{e} domain. Let assumptions \textbf{(H.1)}-\textbf{(H.2)} and \textbf{(H.5)} be satisfied.
In addition  we assume that the function $G$ satisfies condition \eqref{E:G_Lipsch}  with  $L<\sqrt{2}$
 and  inequality \eqref{E:G}  with ${\lambda}_{0}=0$.
Then there exists an invariant measure of the semigroup ${({P}_{t})}_{t\ge 0 }$ defined by \eqref{eqn:semigroup}, i.e. a probability measure $\mu $ on $\rH$ such that
\[
   {P}^{\ast }_{t} \mu = \mu .
\]
\end{theorem}

\begin{remark}
In this section we have used strong solutions. In particular, in order to show a global inequality \eqref{E:ineq_apriori} which was a basis for Corollary \ref{cor-Krylov_Bogoliubov_cond}. However, we could have easily avoided this. For instance, instead of the global inequality \eqref{E:ineq_apriori} we could prove that every martingale solution $( \Omega , \fcal ,  {\mathbb{F}}, \p , W,u )$ of equation \eqref{E:NS'} with the initial data $u_0$ on the time interval $[0,T]$ satisfies inequality \eqref{E:ineq_apriori} for only $t\in [0,T]$ but with constants $C$, $\eta$ and $\rho$ independent of $T$.
\end{remark}

\appendix

\section{Uniform estimates of the solutions Galerkin approximatin equations} \label{App:Galerkin-est}

Let us recall that the proof of existence of a martingale solution of the Navier-Stokes equations, given in \cite{Brz+Motyl_2013}, is based on the  Faedo-Galerkin approximation in the space ${H}_{n}$, see (5.2) in the cited paper. In order to continue we need to choose and fix a stochastic basis and thus we assume that Assumption \ref{assumption-second} holds. We also fix an $\mathcal{F}_0$-measurable $\rH$-valued random variable. Then the $n$-th equation is the following one in the space ${H}_{n}$.

\begin{equation} \label{E:Galerkin}
\begin{cases}
  & d \un (t) =  - \bigl[ \Pn \acal \un (t)  + \Bn  \bigl(\un (t)\bigr)  - \Pn f(t) \bigr] \, dt
   + \Pn G\bigl( \un (t)\bigr) \, dW(t),   \quad t > 0 , \\
  &  \un (0) = \Pn {u }_{0} .
\end{cases}
\end{equation}
Recall that ${H}_{n}$ is a finite dimensional subspace spanned  by the $n$ first eigenvectors of the operator $L$ given by (2.19) in  \cite{Brz+Motyl_2013}, $\Pn $ is defined by \cite[(2.25)]{Brz+Motyl_2013} and ${B}_{n}$ is defined on p. 1636 in \cite{Brz+Motyl_2013}.
For details see \cite[Lemmas 2.3 and 2.4]{Brz+Motyl_2013}.
In particular, $\Pn $ restricted to $\rH$ is the orthogonal projection.
The  existence of a solution of equation \eqref{E:Galerkin} is guaranteed by Lemma 5.2 in \cite{Brz+Motyl_2013}.

The following result corresponds to Lemma 5.3 from \cite{Brz+Motyl_2013}.
The proof of estimates \eqref{E:H_estimate_Galerkin_p},  \eqref{E:HV_estimate_Galerkin}  and \eqref{E:Poincare_Galerkin},  is similar to the proof of estimates (5.4), (5.5) and (5.6) from Appendix A in \cite{Brz+Motyl_2013}. However, we provide the details to indicate the dependence of appropriate constants on the data, which will be important in the proof of  continuous dependence of the solutions of the Navier-Stokes equations on the initial state ${u}_{0}$ and the external forces $f$.
 Moreover, if   $\ocal $ is the Poincar\'{e} domain,  we prove a new  estimate, see \eqref{E:Poincare_Galerkin}. This  estimate is of crucial importance in the proof of the existence of invariant measure.
Recall that we have put $\frac{\eta }{2-\eta }=\infty$ when $\eta=2$.

\begin{lemmaA} \label{L:Galerkin_estimates}
Let Assumption \ref{assumption-second} and parts  (\textbf{H.2}),(\textbf{H.3}) and (\textbf{H.5}) of Assumption \ref{assumption-main} be satisfied. In particular, we assume  that  $p$ satisfies \eqref{eqn-p_cond}, i.e.
\begin{equation*}
p\in  \bigl[ 2, 2+ \frac{\eta }{2-\eta } \bigr) ,
\end{equation*}
where $\eta \in (0,2]$ is given in  \textbf{(H.2)}.
\begin{itemize}
\item[(1) ]
Then for every $T>0$,    $\nu$, $R_1$ and $R_2$  there exist constants  $C_1(p)$, ${\tilde{C}}_2(p)$,
 ${C}_{2}(p)$,  such that if ${u}_{0} \in {L}^{p}(\Omega ,\mathcal{F}_0,\rH) $,
${f} \in  {L}^{p}([0,\infty ); \rV^\prime )$ satisfy $ \mathbb{E} [ \norm{u_0 }{\rH}{p} ]  \leq R_1$
 and $ \vert f{\vert}_{{L}^{p}(0,T;\rV^\prime)}\leq R_2$, then every    solution $\un $ of Galerkin equation  \eqref{E:Galerkin} with the initial data ${u}_{0}$ and the external force $f$
 satisfies the following  estimates
\begin{equation} \label{E:H_estimate_Galerkin_p}
\sup_{n \in \nat }\mathbb{E} \bigl( \sup_{s\in [0,T] } {|\un (s)|}_{\rH}^{p} \bigr) \le {C}_{1}(p)
\end{equation}
and
\begin{equation} \label{E:HV_estimate_Galerkin}
\sup_{n \in \nat } \mathbb{E} \bigl[ \int_{0}^{T} {|\un (s)|}_{\rH}^{p-2} \norm{ \nabla \un (s)}{}{2} \, ds \bigr] \le {\tilde{C}}_{2}(p)  ,
\end{equation}
and
\begin{equation} \label{E:V_estimate_Galerkin}
\sup_{n \in \nat } \mathbb{E} \bigl[ \int_{0}^{T} \norm{ \nabla \un (s)}{}{2} \, ds \bigr] \le  {C}_{2}(p).
\end{equation}
\item[(2) ]
Moreover, if $\ocal $ is a Poincar\'{e} domain and inequality \eqref{E:G} holds with ${\lambda }_{0}=0$, then for every $t>0$
\begin{equation}  \label{E:Poincare_Galerkin}
 \sup_{n \in \nat } \biggl( \mathbb{E} [\, \norm{\un (t)}{\rH}{2} \, ]  +  \frac{\eta }{2} \e \biggl[ \int_{0}^{t} \norm{\nabla \un (s)}{}{2} \, ds \biggr]  \biggr)
   \le \mathbb{E} [\, \norm{{u}_{0}}{\rH}{2} \, ]
 + \frac{2}{\eta } \int_{0}^{t} \norm{f(s)}{{\rv }^{\prime }}{2} \, ds + \rho t.
\end{equation}
\end{itemize}
\end{lemmaA}

\begin{proof}[Proof of Lemma A.\ref{L:Galerkin_estimates}]
Let us fix $p $ satisfying  condition \eqref{eqn-p_cond}. We apply the It\^{o} formula from \cite{Pardoux_1979} to the function $F$ defined by
\[F: \rH \ni x \mapsto  {|x|}_{\rH}^{p}\in \mathbb{R}.\]
In the sequel we will omit the subscript $\rH$ and write $|\cdot |:= {|\cdot |}_{\rH}$. Note that
\[
F^\prime(x)=d_xF= p \cdot {|x|}^{p-2} \cdot x  ,
\qquad \|  F^{\prime\prime}(x) \|= \|  d^2_xF \| \leq  p (p-1) \cdot {|x|}^{p-2} ,
\qquad x \in \rH.
\]
With the above comments in mind, we have, for $\, t \in [0,\infty)$,
\begin{eqnarray}
{|{\un } (t)|}^{p} -  {|{\un } (0)|}^{p}
  &=&  \int_0^{t} \Bigl[ p \,  {|{\un } (s)|}^{p-2} \dual{{\un } (s)}{- \acal {\un } (s)
       -  {B}_{n} \bigl( {\un } (s) \bigr) +  \Pn f(s)}{} \nonumber \\
 & &\qquad +    \frac{1}{2} \tr \;\bigl[  F^{\prime\prime}(\un (s))\bigl( \Pn G ({\un } (s)), \Pn G ({\un } (s)) \bigr)  \bigr] \Bigr] \, ds
\nonumber \\
&& \qquad  + p \, \int_0^{t}   {|{\un } (s)|}^{p-2} \dual{{\un } (s)}{ \Pn  G ( {\un } (s) ) \, d W(s) }{}
 \nonumber  \\
&=&
  \int_0^{t}  \Bigl[ - p \, {|{\un } (s)|}^{p-2} \Norm{{\un } (s)}{}{2}
   +  p \, {|{\un } (s)|}^{p-2} \dual{{\un } (s)}{ \Pn f(s)}{}
    \nonumber \\
& &\qquad  + \frac{1}{2} \tr \bigl[  F^{\prime\prime}(\un (s))\bigl( \Pn G ({\un } (s)), \Pn G ({\un } (s)) \bigr)  \bigr] \Bigr] \, ds  \nonumber  \\
 && \qquad +   p \, \int_0^{t}  {|{\un } (s)|}^{p-2} \dual{\un (s)}{\Pn G (\un (s)) \, d W(s)}{}  .
   \label{ineq-01_Galerkin}
\end{eqnarray}
Since
\[
\tr \,\bigl[  F^{\prime\prime}(u )\bigl( \Pn G (u ), \Pn G (u) \bigr)  \bigr]
\leq   p(p-1) \, {|u|}^{p-2} \cdot \norm{  G(u)}{\lhs (\rK,\rH)}{2}, \;\; u \in {\rV }_{},
\]
and by  \eqref{E:G}
\[
  \norm{ G(u)}{\lhs (\rK,\rH)}{2}   \le (2- \eta ) \, \norm{\nabla u}{}{2}
  + {\lambda }_{0} {|u|}^{2} + \rho , \;\; u  \in {\rV }_{},
\]
and since  by \eqref{E:norm_V} and the  Young  inequality with exponents $2,\frac{2p}{p-2}$ and $p$,
\begin{eqnarray*}
  \,  {|u |}^{p-2} \dual{f}{u }{}
 &\leq &   \,  {|u |}^{p-2} \Norm{u }{\rV}{}  \, {|f|}_{\rV^\prime }
 \ = \ {|u |}^{p-2}   {({|u |}^{2}+\norm{\nabla u }{}{2})}^{\frac{1}{2}}
   \, {|f|}_{\rV^\prime }
 \\
 &\leq &
 \frac{\eps}{2} ({|u |}^{2}+\norm{\nabla u }{}{2})\,  {|u |^{p-2}+(\frac12-\frac1p){|u |}^{p}+\frac{\eps^{-p/2}}{p}\, {|f|}^p_{\rV^\prime } }
\\
 &\leq &
  \frac{\eps}{2}  \norm{\nabla u }{}{2}\,  {|u |^{p-2}+(\frac{1+\eps}{2}-\frac1p){|u |}^{p}+\frac{\eps^{-p/2}}{p}\, {|f|}^p_{\rV^\prime } },
   \;\; u \in \rV,\, f\in \rV^\prime,
\end{eqnarray*}
we infer that
\begin{eqnarray*}
  {|\un (t)|}^{p}   &+& \bigl[ p - p \frac{\eps}{2} - \frac{1}{2} p(p-1) (2- \eta ) \,   \bigr]
  \,\int_{0}^{t} {|\un (s)|}^{p-2} \norm{\nabla \un (s)}{}{2} \, ds  \\
&    \le & |\un (0)|
 + \int_{0}^{t} \Bigl[ (\frac{p(1+\eps)}2-1){|\un (s)|}^{p}+\eps^{-p/2}\, {|f(s)|}^p_{\rV^\prime }
   + \frac{1}{2} p(p-1) \, {|\un (s)|}^{p-2} \cdot \bigl(
  {\lambda }_{0} {|\un (s)|}^{2} + \rho \bigr) \Bigr]  \, ds \\
 & &\hspace{+2truecm}\lefteqn{
  + p \, \int_{0}^{t} {|\un (s)|}^{p-2} \dual{\un (s)}{ \Pn G ( \un (s) ) \, d W(s) }{} }
\\
&=& \int_{0}^{t} \Bigl[ \Bigl(\frac{\lambda_{0}}{2}p(p-1)+\frac{p(1+\eps)}2-1\Bigr) \,
  {|\un (s)|}^{p} + \frac{\rho }{2} p(p-1)\, {|\un (s)|}^{p-2} +\eps^{-p/2}\, {|f(s)|}^p_{\rV^\prime } \Bigr]  \, ds
\\
&&\hspace{+2truecm}\lefteqn{  + p \, \int_{0}^{t} {|\un (s)|}^{p-2} \dual{\un (s)}{ \Pn G (\un (s) ) \, d W(s) }{} }
\end{eqnarray*}
Let us choose $\eps \in (0,1) $ such that  $ \delta=\delta(p,\eta):=p - p \frac{\eps}2 - \frac{1}{2} p(p-1) (2- \eta ) >0 $,
or equivalently,
\[
  \eps <  1 \wedge [ 2 -  (p-1) (2-\eta )] .
\]
Notice that under  condition \eqref{eqn-p_cond} such $\eps $ exists. Denote also
\[
K_p({\lambda }_{0},\rho ):= \frac{\lambda_{0}}{2}p(p-1)+p-1+\rho p(1-\frac2p)\frac{p-1}{2}
=\frac{p-1}{2} [{\lambda }_{0}p+2 +\rho (p-2)] .
\]
Thus, since by Young inequality $x^{p-2}\leq(1-\frac2p)x^p+\frac2p1^{p/2}$ for $x\geq 0$,   we obtain
\begin{equation}
\begin{array}{rcl}
  {|\un (t)|}^{p}   &+&\delta
  \, \int_{0}^{t}{|\un (s)|}^{p-2} \norm{\nabla \un (s)}{}{2} \, ds \\
  && \\
 & \le&    {|u (0)|}^{p} +K_p({\lambda }_{0},\rho)
\int_{0}^{t} \, {|\un (s)|}^{p} \, ds
  + \rho (p-1)  t \,
  + \eps^{-p/2} \int_{0}^{t} {|f(s)|}_{\rV^{\prime }}^{p}\, ds
  \\
  && \\
& +& p \,  \int_{0}^{t}{|\un (s)|}^{p-2} \dual{\un (s)}{\Pn G(\un (s)) \, d W(s) }{} ,
\quad t \in [0,\infty).
\end{array}     \label{E:apriori}
\end{equation}
Since $\un $ is the solutions of the Galerkin equation,  we infer that the process
\[
 {\mu }_{n}(t):= \int_{0}^{t} {|\un (s)|}^{p-2} \dual{\un (s)}{\Pn  G(\un (s)) \, d W(s) }{} ,
 \qquad t \in [0,\infty)
\]
is a square integrable martingale. Indeed, if we define a map
\[
g:\rV\ni u\mapsto \{ \rK\ni k \mapsto   \dual{u }{\Pn  G(u)k }{} \in \rH\} \in \lhs (\rK,\mathbb{R})
\]
then $\mu_n(t)=  \int_{0}^{t} {|\un (s)|}^{p-2} g(\un (s))dW(s)$ and hence, by  assumption
\eqref{E:G} and the fact that $\Pn $ is the orthogonal projection in $\rH$ we infer that for every $t\geq 0$,
\begin{eqnarray}
&&\int_{0}^{t } \Vert\, {|\un (s)|}^{p-2} g(\un (s))  \Vert^2_{\lhs (\rK,\mathbb{R})}\, ds=
\int_{0}^{t} {|\un (s)|}^{p-2} \Vert  \,g(\un (s))  \Vert^2_{\lhs (\rK,\mathbb{R})}\, ds
\\
\nonumber
&&\leq \int_{0}^{t } {|\un (s)|}^{p-2} {|\un (s)|}^{2} \Vert \Pn G(\un (s))  \Vert^2_{\lhs (\rK,{\rH})} \, ds
\leq \int_{0}^{t} {|\un (s)|}^{p} \big[ (2- \eta ) \, \norm{\nabla \un (t)}{}{2}
  + {\lambda }_{0} {|\un (t)|}^{2} + \rho \big] \, ds.
\end{eqnarray}
Hence by the fact that $\un $ is a Galerkin solution  we infer that
\[ \mathbb{E} \, \int_{0}^{t} \Vert\, {|\un (s)|}^{p-2} g(\un (s))  \,\Vert^2_{\lhs (\rK,\mathbb{R})}\, ds<\infty, \;\; t\geq 0.
\]
and thus we infer, as claimed,  that the process $ {\mu }_{n}$ is a square integrable  martingale. Hence,
  $\mathbb{E}[{\mu }_{n} (t) ] = 0 $. Let us now fix $T>0$.  By taking expectation in inequality \eqref{E:apriori} we infer that
\begin{equation}
\begin{array}{rcl}
  \mathbb{E}\bigl[ \, {|\un (t)|}^{p} \,\bigr]
  \, \, &\leq&   \, \,  {\mathbb{E} \bigl[ \, |u_0|}^{p}\, \bigr] +
K_p({\lambda }_{0},\rho)\int_{0}^{t} \,
 \mathbb{E}\,\bigl[ {|\un (s)|}^{p} \bigr] \, ds
+   \rho (p-1) t
  + \eps^{-p/2} \mathbb{E}\int_{0}^{t} {|f(s)|}_{{\rV }^{\prime }}^{p}\, ds
    \nonumber\\
 &&\nonumber\\
    &\leq &
  \mathbb{E} \bigl[ \, {|u_0|}^{p} \, ]  +
K_p({\lambda }_{0},\rho)\int_{0}^{t} \,
 \mathbb{E}\,\bigl[ {|\un (s)|}^{p} \bigr] \, ds   + \rho (p-1) T
 + \eps^{-p/2} \mathbb{E}\int_{0}^{T} {|f(s)|}_{{\rV }^{\prime }}^{p}\, ds
 ,   \quad t \in [0,T].
  \end{array}   \label{E:apriori'}
\end{equation}
Hence by the Gronwall Lemma there exists a constant
${\tilde{C}}_{p}={\tilde{C}}_p(T,\eta ,{\lambda }_{0},\rho , \mathbb{E}[ |u_0 |^{p}],  \Norm{f}{{L}^{p}(0,T;{\rV }^{\prime })}{} )
={\tilde{C}}_p(T,\eta ,{\lambda }_{0},\rho ,{R}_{1},{R}_{2} )>0$  such that
\begin{equation*}
  \mathbb{E}\,\bigl[ {|\un (t )|}^{p} \bigr]  \le {\tilde{C}}_{p} , \qquad  \, t \in [0,T],
  \quad n \in \nat ,
\end{equation*}
i.e.
\begin{equation} \label{E:app_H_est}
 \sup_{n \in \nat } \sup_{t\in [0,T]} \mathbb{E}\,\bigl[ {|\un (t )|}^{p} \bigr]  \le {\tilde{C}}_{p} .
\end{equation}
Using this bound in \eqref{E:apriori} we also obtain
\begin{equation} \label{E:app_HV_est}
\sup_{n \in \nat }  \mathbb{E}\,\biggl[ \int_{0}^{T}{|\un (s)|}^{p-2} \norm{\nabla \un (s)}{}{2} \, ds  \biggr]
  \le {\tilde{C}}_{2}(p)
\end{equation}
for a new constant ${\tilde{C}}_{2}(p) ={C}_{2}(p,T,\eta ,{\lambda }_{0},\rho , \mathbb{E}[ |u_0 |^{p}],  \Norm{f}{{L}^{p}(0,T;{\rV }^{\prime })}{})
={\tilde{C}}_{2}(p,T,\eta ,{\lambda }_{0},\rho , {R}_{1},{R}_{2})$.
This completes the proof of estimates \eqref{E:HV_estimate_Galerkin}.
Since $\mathbb{E} [{|{u}_{0}|}^{2}] \le {(\mathbb{E} [{|{u}_{0}|}^{p}] )}^{\frac{2}{p}} \le {R}_{1}^{2/p}$,
we infer that \eqref{E:V_estimate_Galerkin} holds with another constant ${C}_{2}(p)$.

Let us move to the proof of estimate \eqref{E:H_estimate_Galerkin_p}.
By the Burkholder-Davis-Gundy inequality, see \cite{DaPrato+Zabczyk_1996}, the Schwarz inequality and inequality \eqref{E:G}, there exists a constant ${c}_{p}$ such that  for any $t\geq 0$,
\begin{eqnarray}
 & & \mathbb{E}\,\biggl[ \sup_{0 \le s \le t}
 \biggl| \int_{0}^{s} p \, {|\un (\sigma )|}^{p-2} \dual{\un (\sigma )}{\Pn G ( \un (\sigma ) ) \, d W(\sigma ) }{}
  \biggr| \biggr] \nonumber \\
& &\le {c}_{p} \cdot \mathbb{E}\,\biggl[
 {\biggl( \int_{0}^{t} \, { |\un (\sigma )|}^{2p-2} \cdot
 \norm{ \Pn G ( \un (\sigma ) ) }{\lhs (\rK,\rH)}{2}
   \, d\sigma
  \biggr) }^{\frac{1}{2}} \biggr] \nonumber \\
& &\le {c}_{p} \cdot \mathbb{E}\,\biggl[ \sup_{0\le \sigma \le t } {|\un (\sigma )|}^{\frac{p}{2}}
 {\biggl( \int_{0}^{t} \,  { |\un (\sigma )|}^{p-2} \cdot
 \norm{  G ( \un (\sigma ) ) }{\lhs (\rK,\rH)}{2}
   \, d\sigma
  \biggr) }^{\frac{1}{2}} \biggr] \nonumber  \\
 &&  \le \frac{1}{2} \mathbb{E}\,\bigl[ \sup_{0\le s \le t } {|\un (s )|}^{p} \bigr]
   + \frac{1}{2}{c}_{p}^{2} \, \int_{0}^{t} \,  { |\un (\sigma )|}^{p-2} \cdot
 \norm{  G ( \un (\sigma ) ) }{\lhs (\rK,\rH)}{2}
   \, d\sigma
 \nonumber  \\
&&  \le \frac{1}{2} \mathbb{E}\,\bigl[ \sup_{0\le s \le t } {|\un (s )|}^{p} \bigr]
   + \frac{1}{2}{c}_{p}^{2} \, \int_{0}^{t} \,  { |\un (\sigma )|}^{p-2} \cdot
 \bigl[ (2-\eta )\norm{\un (\sigma )}{}{2} + {\lambda }_{0} \,  {|\un (\sigma )|}^{2} + \rho \bigr]
\, d\sigma \nonumber \\
 & &\le \frac{1}{2} \mathbb{E}\,\bigl[ \sup_{0\le s \le t } {|\un (s )|}^{p} \bigr]
  + \frac{1}{2}{c}_{p}^{2} \frac{2\rho }{p} t
 + \frac{1}{2}{c}_{p}^{2} (2-\eta ) \mathbb{E}\,\biggl[ \int_{0}^{t} { |\un (\sigma )|}^{p} \Norm{\un (\sigma )}{}{2}\, d \sigma \biggr]
\nonumber \\
&& \qquad +  \frac{1}{2} {c}_{p}^{2} \biggl( {\lambda }_{0} + \rho \Bigl( 1-\frac{2}{p}\Bigr) \biggr) \,   \cdot
\mathbb{E}\,\biggl[ \int_{0}^{t} \, { |\un (\sigma )|}^{p} \, d\sigma \biggr] .
 \label{E:BDG_ineq_Galerkin}
\end{eqnarray}
Using \eqref{E:BDG_ineq_Galerkin}  in \eqref{E:apriori},
by inequalities \eqref{E:app_H_est} and \eqref{E:app_HV_est}  we infer that
\begin{eqnarray*}
 \mathbb{E}\,\bigl[ \sup_{0 \le s \le t }{|\un (s)|}^{p} \bigr]
   &\leq &   \mathbb{E} [ \, {|u_0 |}^{p} \, ] +
\biggl[ K_p({\lambda }_{0},\rho)
 + \frac{1}{2} {c}_{p}^{2} \biggl( {\lambda }_{0} + \rho \Bigl( 1-\frac{2}{p}\Bigr) \biggr) \biggr]
\int_{0}^{t} \,
 \mathbb{E}\,\bigl[ {|\un (s)|}^{p} \bigr] \, ds \\
&+& \biggl( \frac{2\rho}{p} +{c}_{p}^{2} \frac{\rho }{p} \biggr) \, t
  + \eps^{-p/2} \,\int_{0}^{t} {|f(s)|}_{{\rV }^{\prime }}^{p}\, ds
 \\
& +&  \frac{1}{2} \mathbb{E}\,\bigl[ \sup_{0\le s \le t } {|\un (s )|}^{p} \bigr]
+ \frac{1}{2}{c}_{p}^{2} (2-\eta ) \mathbb{E}\,\biggl[ \int_{0}^{t} { |\un (\sigma )|}^{p} \Norm{\un (\sigma )}{}{2}\, d \sigma \biggr] \\
&\leq &   \mathbb{E} [ \, {|u_0 |}^{p} \, ] +
\biggl[ K_p({\lambda }_{0},\rho)
 + \frac{1}{2} {c}_{p}^{2} \biggl( {\lambda }_{0} + \rho \Bigl( 1-\frac{2}{p}\Bigr) \biggr) \biggr]
  {\tilde{C}}_{p} t \\
&+& \frac{\rho }{p} ( 2 +{c}_{p}^{2} ) \, t
  + \eps^{-p/2} \,\int_{0}^{t} {|f(s)|}_{{\rV }^{\prime }}^{p}\, ds
 \\
& +&  \frac{1}{2} \mathbb{E}\,\bigl[ \sup_{0\le s \le t } {|\un (s )|}^{p} \bigr]
+ \frac{1}{2}{c}_{p}^{2} (2-\eta ) {C}_{2}(p)
, \quad t\geq 0.
\end{eqnarray*}
Thus for a fixed $T>0$
\begin{eqnarray*}
 \mathbb{E}\,\bigl[ \sup_{0 \le s \le T }{|\un (s)|}^{p} \bigr]
   &\leq & {C}_{1}(p)  ,
\end{eqnarray*}
where
\begin{eqnarray*}
 {C}_{1}(p)  &=& {C}_{1}(p,T,\eta ,{\lambda }_{0},\rho ,{R}_{1}, {R}_{2})  \\
&:=& 2 {R}_{1} +
2\biggl[ K_p({\lambda }_{0},\rho)
 + \frac{1}{2} {c}_{p}^{2} \biggl( {\lambda }_{0} + \rho \Bigl( 1-\frac{2}{p}\Bigr) \biggr) \biggr]
  {\tilde{C}}_{p} T  \\
&&  + 2\biggl( \frac{2\rho}{p} +{c}_{p}^{2} \frac{\rho }{p} \biggr) \, T
  + 2\eps^{-p/2} {R}_{2}
 +  {c}_{p}^{2} (2-\eta ) {C}_{2}(p).
\end{eqnarray*}
This completes the proof of estimate \eqref{E:H_estimate_Galerkin_p}.

To prove  inequality \eqref{E:Poincare_Galerkin}
let us assume that $\ocal $ is a Poincar\'{e} domain and inequality \eqref{E:G} holds with ${\lambda }_{0}=0$.  Recall that now in the space $\rV $ we consider the inner product $\dirilsk{\cdot }{\cdot }{}$ given by \eqref{E:il_sk_Dir}.
By identity \eqref{ineq-01_Galerkin} from the previous proof with $p=2$,   we have
\begin{eqnarray}
\label{ineq-02_Galerkin}
  {|{\un } (t)|}^{2} &-&  {|{u} (0)|}^{2}
=
  \int_0^{t}  \Bigl\{ - 2 \,  \Norm{\un (s)}{}{2}
   + 2 \,  \dual{\un (s)}{ f(s)}{}
 + \frac{1}{2} \tr \bigl[  F^{\prime\prime}(\un (s))\bigl( G ({\un } (s)), G ({\un } (s)) \bigr)  \bigr] \Bigr\} \, ds \nonumber\\
 &+&  2 \, \int_{0}^{t}   \dual{\un (s)}{\Pn G (\un (s)) \, d W(s)}{}, \;\; t\geq 0  .
 \nonumber
\end{eqnarray}
Since $\mathbb{E}\big(\int_{0}^{t} \dual{\Pn G(\un (s))}{\un (s)\, dW(s)}{} \big)=0$, we infer that
\begin{eqnarray*}
\mathbb{E} |\un (t){|}_{\rH}^{2}
&  \leq &   \mathbb{E} [ \, {|u_0|}_{\rH}^{2} \, ] +
\mathbb{E}  \int_{0}^{t}  \bigl\{  -2  \Norm{\un  (s)}{}{2}  +2\dual{ f(s)  }{\un (s)}{}
 \bigr\} \, ds  + \mathbb{E} \int_{0}^{t}  \norm{ \Pn G(\un (s)) }{\lhs (\rK,\rH)}{2} \, ds
\end{eqnarray*}
Using assumption \eqref{E:G}
 with ${\lambda }_{0}=0$ (i.e $\norm{G(\un (s))}{\lhs (K,\rH)}{2} \le (2-\eta )\Norm{{\un } (s)}{}{2} +\varrho $) we get
\begin{eqnarray}
\mathbb{E} \norm{u(t)}{\rH}{2}
&  \leq & -\eta
\mathbb{E}  \int_{0}^{t}      \Norm{{\un } (s)}{}{2}   \, ds  + \mathbb{E} [ \,{|u_0|}_{\rH}^{2} \, ]
 + 2\, \mathbb{E} \int_{0}^{t} \dual{ f(s)  }{u(s)}{} \, ds +\varrho t.
  \label{ineq-apriori_2}
\end{eqnarray}
Since
$2\dual{f(s)}{u(s)}{} \leq \frac{\eta}2 \norm{\nabla \un (s)}{}{2}+\frac{2}{\eta} \norm{f }{{\rV }^\prime}{2} $ we infer that
\begin{eqnarray}
\mathbb{E} \norm{\un (t)}{\rH}{2}
&  \leq & -\frac{\eta}2
\mathbb{E}  \int_{0}^{t}      \Norm{{\un } (s)}{}{2}   \, ds
   +\mathbb{E} [{|u_0|}_{\rH }^{2}] + \frac{2}{\eta} \int_{0}^{t}\norm{f(s) }{{\rV }^\prime}{2}
 +\varrho  t,\;\; t\geq 0.
  \label{ineq-apriori_3_Galerkin}
\end{eqnarray}
The proof of inequality \eqref{E:Poincare_Galerkin} is thus complete.
\end{proof}

\section{Proof of Theorem \ref{T:existence_NS}}  \label{sec:Proof of Theorem_existence_NS}

Similarly to the proof of Theorem 5.1 in  \cite{Brz+Motyl_2013} the present proof is based on the Galerkin method. We will use the fact the the laws of the Galerkin solutions form a tight set of probability measures on ${\mathcal{Z}}_{T}$. We will use the  Jakubowski's version of the Skorokhod theorem \ref{T:2_Jakubowski}, as well. However, some details are different.

Let us fix positive numbers $T$, ${R}_{1}$ and ${R}_{2}$.  Let us assume that  $\mu$ is a Borel probability measure on $\rH$,
${f} \in  {L}^{p}([0,\infty ); \rV^\prime )$ which satisfy  $\int_{\rH} \vert x\vert^p \mu(dx) \leq R_1$
 and $ \vert f{\vert}_{{L}^{p}(0,T;\rV^\prime)}\leq R_2$.
Similarly to the previous section  we choose and fix a stochastic basis and thus we assume that Assumption \ref{assumption-second} holds.
We also fix  an $\mathcal{F}_0$-measurable $\rH$-valued random variable whose law is equal to $\mu$.

As in the proof of    \cite[Theorem 5.1]{Brz+Motyl_2013} let ${({u}_{n})}_{n \in \nat }$ be a sequence of the solutions of the Galerkin equations. Then the set of  laws $\{ \mathcal{L}({u}_{n}, n \in \nat \} $ is tight on the space $({\mathcal{Z}}_{T},\sigma ({\mathcal{T}}_{T}))$, where $\sigma ({\mathcal{T}}_{T})$ denotes the topological $\sigma $-field.
By  theorem \ref{T:2_Jakubowski} there exists a subsequence ${({n}_{k})}$, a probability space $(\tOmega , \tfcal , \tilde{\p } )$ and, on this space ${\mathcal{Z}}_{T}$-valued random variables $u $, $\tunk $, $k \in \nat $,
and a sequence of $\rK$-valued Wiener processes $\tilde{W }, {\tilde{W}}_{{n}_{k}} $, $k \in \nat $
 such that
 \begin{equation}
 \label{eqn-equal laws_Galerkin}
\mbox{ the variables $(\unk ,W_{})$
and $(\tunk ,{\tilde{W}}_{{n}_{k}})$ have
the same laws on the Borel $\sigma$-algebra
$\mathcal{B}\big(\zcal_{T} \times \mathcal{C} ([0,T],\rK )\big)$
}
\end{equation}
and
\begin{equation}
  \mbox{$(\tunk ,{\tilde{W}}_{{n}_{k}}) $ converges to $(u, \tilde{W}) $ in $\zcal_{T} \times \mathcal{C}([0,T];\rK )$ almost surely on $\tOmega $.}
\end{equation}
In particular,
\begin{equation}
 \label{eqn-conv-as_Galerkin}
\mbox{$\tunk $ converges to $u $ in $\zcal_{T}$ almost surely on $\tOmega $.}
\end{equation}
We will denote the subsequence $(\tunk ,{\tilde{W}}_{{n}_{k}})$ again by $(\tun , {\tilde{W}}_{n})$.
Define a corresponding  sequence of  filtrations by
\begin{equation}\label{eqn-new filtration_Galerkin}
{\tilde{\mathbb{F}}}_{n} =({\tilde{\fcal }}_{n,t})_{t\geq 0}, \mbox{ where }
{\tilde{\fcal }}_{n,t} = \sigma \{ (\tun (s),{\tilde{W }}_{n}(s)), \, \, s \le t \},\;\;  t \in [0,T].
\end{equation}
To obtain \eqref{E:H_estimate_NS}, we modify the proof from
\cite{Brz+Motyl_2013} at pages 1650-51. Namely, using Lemma A.\ref{L:Galerkin_estimates}, we infer that the processes $\tun $, $n \in \nat $,  satisfy the following inequalities
\begin{equation} \label{E:H_estimate_Galerkin'_p}
\sup_{n \in \nat } \tilde{\mathbb{E}} \bigl( \sup_{s\in [0,T] } {|\tun (s)|}_{\rH}^{p} \bigr) \le {C}_{1}(p)
\end{equation}
and
\begin{equation} \label{E:V_estimate_Galerkin'}
\sup_{n \in \nat } \tilde{\mathbb{E}} \bigl[ \int_{0}^{T} \norm{ \nabla \tun (s)}{{L}^{2}}{2} \, ds \bigr] \le {C}_{2}(p).
\end{equation}
Let us emphasize that the constants  ${C}_{1}(p)$ and ${C}_{2}(p)$, being the same as in  Lemma A.\ref{L:Galerkin_estimates}, depend on $T$, ${R}_{1}$ and ${R}_{2}$.
Using inequality  \eqref{E:H_estimate_Galerkin'_p} we choose a subsequence, still denoted by $({\tilde{u}}_{n})$, convergent weak star in the space $ {L}^{p}(\tilde{\Omega };{L}^{\infty }(0,T;\rH))$ and infer that
\begin{equation} \label{E:H_estimate_NS'_p}
  \e \bigl[ \sup_{s\in [0,T]} \norm{u(s)}{\rH}{p} \bigr] \le {C}_{1}(p)
\end{equation}
and that the limit process $u$ satisfies  \eqref{E:H_estimate_NS'_p}, as well.
This completes the proof of inequality \eqref{E:H_estimate_NS_p}.
To prove \eqref{E:H_estimate_NS} let us fix $q \in [1,p)$. Notice that for every $t \in [0,T]$
\[
    {|u(t)|}^{q} =  {({|u(t)|}^{p} )}^{q/p}
    \le {\Bigl( \sup_{t\in [0,T]} {|u(t)|}^{p} \Bigr) }^{q/p}  .
\]
Thus, $ \sup_{t\in [0,T]}  {|u(t)|}^{q}
    \le {\Bigl( \sup_{t\in [0,T]} {|u(t)|}^{p} \Bigr) }^{q/p}$,
and so by the H\"{o}lder inequality
\begin{eqnarray*}
 \mathbb{E} \Bigl[ \sup_{t\in [0,T]}  {|u(t)|}^{q}  \Bigr]
 \le \mathbb{E} \Bigl[ {\Bigl( \sup_{t\in [0,T]} {|u(t)|}^{p} \Bigr) }^{q/p}  \Bigr]
 \le {\biggl( \mathbb{E} \Bigl[  \sup_{t\in [0,T]} {|u(t)|}^{p}   \Bigr]  \biggr) }^{q/p}
 \le  {\bigl( {C}_{1}(p) \bigr) }^{q/p},
\end{eqnarray*}
which means that inequality \eqref{E:H_estimate_NS} holds with the constant ${C}_{1}(p,q):={\bigl( {C}_{1}(p) \bigr) }^{q/p}$.

By inequality \eqref{E:V_estimate_Galerkin'} we infer that the sequence $({\tilde{u}}_{n})$ contains further subsequence, denoted again by $({\tilde{u}}_{n})$, convergent weakly in the space ${L}^{2}([0,T]\times \tilde{\Omega };\rV )$ to $u$. Moreover, it is clear that
\begin{equation} \label{E:V_estimate_NS'}
 \tilde{\mathbb{E}} \bigl[ \int_{0}^{T} \norm{ \nabla u (s)}{{L}^{2}}{2} \, ds \bigr] \le {C}_{2}(p)
\end{equation}
and the process $u$ satisfies \eqref{E:V_estimate_NS}.

To prove the second part of the theorem we  assume that $\ocal $ is a Poincar\'{e} domain and inequality \eqref{E:G} holds with ${\lambda }_{0}=0$.
In this case, by Lemma A.\ref{L:Galerkin_estimates},  instead of inequality \eqref{E:V_estimate_Galerkin'} we can use the following one corresponding to the uniform estimates \eqref{E:Poincare_Galerkin},
\begin{equation}  \label{E:Poincare_Galerkin'}
\frac{\eta }{2} \sup_{n \in \nat }    \e \biggl[ \int_{0}^{T} \norm{\nabla \tun (s)}{{L}^{2}}{2} \, ds \biggr]
   \le \mathbb{E} [\, \norm{{u}_{0}}{\rH}{2} \, ]
 + \frac{2}{\eta } \int_{0}^{T} \norm{f(s)}{{\rv }^{\prime }}{2} \, ds + \rho T,
\end{equation}
choose a subseqence convergent weakly in the space ${L}^{2}([0,T]\times \tilde{\Omega };\rV )$ to $u$ and infer that the limit process satisfies the same estimate, which proves estimate \eqref{E:Poincare-NS-ineq}.
We will prove that the system $(\tOmega ,\tilde{\mathcal{F}}, \tilde{\mathbb{F}}, \tilde{\p },u)$ is a martingale solution  of problem \eqref{E:NS}.

\bigskip  \noindent
\bf Step 1. \rm
Let us   fix $\varphi \in U $. Analogously to \cite{Brzezniak_Hausenblas_2010} and \cite{EM_2014}, let us denote
\begin{eqnarray}
 & & {\Lambda }_{n}({\tilde{u}}_{n},{\tilde{W}}_{n}, \varphi) (t)
:=  \ilsk{{\tilde{u}}_{n}(0)}{\varphi}{\rH }
  - \int_{0}^{t} \dual{{P}_{n}\acal {\tilde{u}}_{n}(s)}{\varphi}{}  ds
- \int_{0}^{t} \dual{{B}_{n}({\tilde{u}}_{n}(s))}{\varphi}{}  ds  \nonumber \\
& & + \int_{0}^{t} \dual{ {f}_{n}(s)}{\varphi}{}\, ds
  + \Dual{\int_{0}^{t} {P}_{n}G({\tilde{u}}_{n}(s))\,  d{\tilde{W}}_{n}(s)}{\varphi}{} , \quad t \in [0,T] ,
\label{E:Lambda_n_un_Galerkin}
\end{eqnarray}
and
\begin{align}
 &  \Lambda  (u, \tilde{W}, \varphi) (t)
:=  \ilsk{u(0)}{\varphi}{\rH}
    - \int_{0}^{t} \dual{ \acal u(s)}{\varphi}{}  ds
 - \int_{0}^{t} \dual{ B(u(s))}{\varphi}{}  ds  \nonumber \\
 & + \int_{0}^{t} \dual{ f(s)}{\varphi}{}\, ds
    +  \Dual{\int_{0}^{t}G(u(s))\,  d\tilde{W}(s)}{\varphi}{}  ,
\quad t \in [0,T] .
\label{E:Lambda_n_NS}
\end{align}

\noindent
Using Lemma 2.4(c) from \cite{Brz+Motyl_2013}, see also \cite[Lemma 5.4]{EM_2014},  we can prove the following  lemma analogous to Lemma  \ref{L:convergence_existence}.

\begin{lemmaB}  \label{L:convergence_existence_Galerkin}
For all $\varphi \in U$
\begin{itemize}
\item[(a)] $\lim_{n\to \infty } \tilde{\e } \bigl[ \int_{0}^{T} {|\ilsk{{\tilde{u}}_{n}(t)-u(t)}{\varphi}{\rH }|}^{2} \, dt \bigr] =0 $,
\item[(b)] $\lim_{n\to \infty } \tilde{\e } \bigl[ {|\ilsk{{\tilde{u}}_{n}(0)-u(0)}{\varphi}{\rH}|}^{2}  \bigr] =0 $,
\item[(c)] $\lim_{n\to \infty } \tilde{\e } \bigl[ \int_{0}^{T}
 \bigl| \int_{0}^{t}\dual{ {P}_{n}\acal {\tilde{u}}_{n}(s)-\acal u(s)}{\varphi}{} \,ds \bigr| \, dt \bigr] =0 $,
\item[(d)] $\lim_{n\to \infty } \tilde{\e } \bigl[ \int_{0}^{T}
 \bigl| \int_{0}^{t}\dual{ {B}_{n} ({\tilde{u}}_{n}(s))- B(u(s))}{\varphi}{} \,ds \bigr| \, dt \bigr] =0 $,
\item[(e)] $\lim_{n\to \infty } \tilde{\e } \bigl[ \int_{0}^{T}
 \bigl| \int_{0}^{t}\dual{ {P}_{n}{f}_{n}(s)-f(s)}{\varphi}{} \,ds \bigr| \, dt \bigr] =0 $,
\item[(f)] $\lim_{n\to \infty } \tilde{\e } \bigl[ \int_{0}^{T}
 \bigl| \dual{ \int_{0}^{t} [ {P}_{n}G({\tilde{u}}_{n}(s)) - G(u(s))] \, d \tilde{W}(s) }{\varphi}{}
 {\bigr| }^{2} \, dt \bigr] =0 $.
\end{itemize}
\end{lemmaB}

 \noindent
Directly from  Lemma \ref{L:convergence_existence_Galerkin} we get the following corollary

\begin{corB}
For every $\varphi \in U$,
\begin{equation} \label{E:un_convergence_Galerkin}
 \lim_{n \to \infty } \norm{\ilsk{{\tilde{u}}_{n}(\cdot )}{\varphi}{\rH }
 - \ilsk{u(\cdot )}{\varphi}{\rH } }{{L}^{2}([0,T]\times \tilde{\Omega })}{} =0
\end{equation}
and
\begin{equation} \label{E:Lambda_n_un_convergence_Galerkin}
     \lim_{n \to \infty } \norm{  {\Lambda  }_{n}({\tilde{u}}_{n}, {\tilde{W}}_{n},\varphi)
 - \Lambda  (u,\tilde{W}, \varphi) }{{L}^{1}([0,T]\times \tilde{\Omega })}{} =0 .
\end{equation}
\end{corB}

\begin{proof}
Assertion  \eqref{E:un_convergence_Galerkin} follows from the  equality
\[
 \norm{\ilsk{{\tilde{u}}_{n}(\cdot )}{\varphi}{\rH }
 - \ilsk{\tilde{u}(\cdot )}{\varphi}{\rH } }{{L}^{2}([0,T]\times \tilde{\Omega })}{2} \\
=  \tilde{\e } \Bigl[ \int_{0}^{T}
{| \ilsk{{\tilde{u}}_{n}(t) -\tilde{u}(t)}{\varphi}{\rH } |}^{2} \, dt \Bigr]
\]
and Lemma \ref{L:convergence_existence}  (a).
To prove  \eqref{E:Lambda_n_un_convergence_Galerkin} let us note
that by the Fubini theorem, we have
\begin{align*}
&  \norm{  {\Lambda  }_{n}({\tilde{u}}_{n}, {\tilde{W}}_{n},\varphi)
 - \Lambda  (u,\tilde{W}, \varphi) }{{L}^{1}([0,T]\times \tilde{\Omega })}{} \\
& = \int_{0}^{T} \tilde{\e } \bigl[ {| {\Lambda  }_{n}({\tilde{u}}_{n}, {\tilde{W}}_{n},\varphi)(t)
 - \Lambda  (u, \tilde{W} ,\varphi)(t) |}^{}\, \bigr] dt .
 \end{align*}
To complete the proof of \eqref{E:Lambda_n_un_convergence_Galerkin} it is sufficient to note that
by Lemma \ref{L:convergence_existence_Galerkin} (b)-(f),  each  term on the right hand side of
\eqref{E:Lambda_n_un_Galerkin} tends at least in ${L}^{1}([0,T]$ $\times \tilde{\Omega })$ to the corresponding term in \eqref{E:Lambda_n_NS}.
\end{proof}
\noindent

\bigskip  \noindent
\bf Step 2. \rm Since ${u}_{n}$ is a solution of the Galerkin equation, for all $t\in [0,T]$ and $\varphi \in U $
\[
   \ilsk{{u}_{n}(t)}{\varphi}{\rH } = {\Lambda  }_{n} ({u}_{n},{W}_{},\varphi)(t) , \qquad \p \mbox{-a.s.}
\]
In particular,
\[
  \int_{0}^{T} \e \bigl[ {|  \ilsk{\un (t)}{\varphi}{\rH }
 - {\Lambda  }_{n} ({u}_{n},{W}_{},\varphi)(t) |}^{} \, \bigr] \, dt  =0.
\]
Since $\lcal ({u}_{n},{W}_{}) = \lcal ({\tilde{u}}_{n},{\tilde{W}}_{n})$,
using \eqref{E:un_convergence_Galerkin} and \eqref{E:Lambda_n_un_convergence_Galerkin} we infer that
\[
\int_{0}^{T} \tilde{\e } \bigl[ {|
\ilsk{u(t)}{\varphi}{\rH } - {\Lambda  }_{} (u,\tilde{W},\varphi)(t) |}^{} \, \bigr] \, dt  =0.
\]
Hence for $l$-almost all $t \in [0,T]$ and $\tilde{\p }$-almost all $\omega \in \tilde{\Omega }$
\begin{equation}
\ilsk{u(t)}{\varphi}{\rH } - {\Lambda  }_{} (u,\tilde{W},\varphi )(t)=0, \label{E:solution_NS}
\end{equation}
Since $u$ is ${\zcal }_{T}$-valued  random variable, in particular $u\in \mathcal{C} ([0,T];{\rH }_{w})$, i.e. $u$ is weakly continuous,
we infer that
equality \eqref{E:solution_NS} holds  for all $t\in [0,T]$ and all $\varphi \in U  $.
Since $U $ is dense in $\rV $,   equality \eqref{E:solution_NS} holds for all $\varphi \in \rV  $, as well.
Putting $\tilde{\mathfrak{A}}:=(\tilde{\Omega }, \tilde{\fcal },\tilde{\p }, \tilde{\fmath })$,
by \eqref{E:solution_NS}  and \eqref{E:Lambda_n_NS}
 we infer that the system
$(\tilde{\mathfrak{A}}, \tilde{W}, u)$ is a martingale solution of  equation \eqref{E:NS}.
The proof of Theorem \ref{T:existence_NS} is thus complete.

\bigskip

\section{Kuratowski Theorem}\label{sec:Kuratowski}

The following is the classical form of the celebrated Kuratowski Theorem.
\begin{theoremC}[Kuratowski Theorem]\label{thm:kuratowski}
  Assume that $X_1,X_2$ are two  Polish spaces  with their Borel $\sigma$-fields denoted respectively by $\mathcal{B}(X_1),\mathcal{B}(X_2)$. If $\phi:X_1\longrightarrow X_2$ is an injective Borel measurable map, then for any $E_1\in\mathcal{B}(X_1)$, $E_2:=\phi(E_1)\in\mathcal{B}(X_2)$.
\end{theoremC}
Let us formulate a simple corollary to the above result.

\begin{propC} \label{prop:Kuratowski}
Suppose that $X_1,X_2$ are two  topological spaces  with their Borel $\sigma$-fields denoted respectively by $\mathcal{B}(X_1),\mathcal{B}(X_2)$. Suppose that  $\phi:X_1\longrightarrow X_2$ is an injective Borel measurable map such that  for any $E_1\in\mathcal{B}(X_1)$, $E_2:=\phi(E_1)\in\mathcal{B}(X_2)$.
Then if  $g:X_1\to \mathbb{R}$ is a Borel measurable map then   a function $f:X_2\to \mathbb{R}$  defined by
\begin{equation}\label{def-function f}
  f(x_2)=\begin{cases}
  g(\phi^{-1}(x_2)), &\mbox{ if } x_2\in \phi(X_1),\\
  \infty, & \mbox{if } x_2\in X_2\setminus \phi(X_1),
  \end{cases}
  \end{equation}
is also  Borel measurable.
\end{propC}

\begin{proof} Note that $g=f \circ \phi$.
\[f^{-1}(A)=\phi[ g^{-1}(A)],\;\;\; A \subset \mathbb{R}.
\]
Thus, if $A \in\mathcal{B}(\mathbb{R})$, then
 by assumptions $g^{-1}(A)\in \mathcal{B}(X_1)$. Hence by Theorem C.\ref{thm:kuratowski} we infer that $\phi[ g^{-1}(A)]\in \mathcal{B}(X_2)$ and thus by the equality above, we infer that $f^{-1}(A)\in \mathcal{B}(X_2)$. The proof is complete.
\end{proof}

One may wonder  if the following a generalization of the above result to non Polish spaces is valid.

\begin{theoremC}\label{thm-kurat-new} Let $X_1$ and $X_2$ be a topological spaces such that for each $i=1,2$ there exists a sequence $\{f_{i,m}\}$ of continuous functions $f_{i,m}:X_i\to\mathbb{R}$ that separate points of $X_i$.  Let us  denote by $\mathscr S_i$ the $\sigma$-algebra generated by the maps $\{f_{i,m}\}$. If $\phi:X_1\longrightarrow X_2$ is an injective  measurable map, then for any $E_1\in \mathscr S_1$, $E_2:=\phi(E_1)\in\mathscr S_2 $.
\end{theoremC}

The following Counterexample shows that the answer to the above question is No.

\begin{counterexampleC} \label{counterexample-Kuratowski}

\begin{trivlist}
\item[ \, 1)] Define $f_k(x)=e^{2ikx\pi}$, $x\in [0,1)$, for every integer $k$ (trigonometric functions).

\item[2)] Let $X_1$ be a non-Borel subset of $[0,1)$ equipped with the euclidean metric.

\item[3)] Let $X_2$ denote $[0,1)$ with the Euclidean metric.

\item[4)] Denote by $f^1_k$ the restriction of $f_k$ to $X_1$.

\item[5)] Then $f^1_k$ are continuous and separate points in $X_1$.

\item[6)] Then $f_k$ are continuous and separate points in $X_2$.

\item[7)] $\sigma(f_k)=Borel(X_2)$ by Stone-Weierstrass.

\item[8)] $\sigma(f^1_k)=\{A\cap X_1:A\in\sigma(f_k)\}=\{A\cap X_1:A\in Borel (X_2)\}=Borel(X_1)$.

\item[9)] Let $\varphi:X_1\to X_2$ be the identity mapping.

\item[10)] $\varphi$ is a continuous injection.

\item[11)] $\varphi[X_1]$ is not Borel in $X_2$.
\end{trivlist}

\end{counterexampleC}

\end{document}